\documentclass[11pt]{article}
\usepackage {amssymb} 
\usepackage {amsmath}
\usepackage{mathrsfs}
\usepackage{amsthm}
\usepackage{graphicx}
\usepackage {amscd}
\usepackage{subfigure}
\usepackage{stmaryrd}
\usepackage{color}
\usepackage{multirow}
\usepackage{tabu}
\usepackage[titletoc, title]{appendix}
\usepackage{diagbox}
\usepackage[colorlinks, citecolor=red]{hyperref}

\usepackage{geometry}  

\newcommand{\vol}{{\rm vol}}
\newcommand{\ord}{{\rm ord}}
\newcommand{\fm}{\mathfrak{m}}
\newcommand{\fa}{\mathfrak{a}}
\newcommand{\cO}{\mathcal{O}}
\newcommand{\bR}{\mathbb{R}}
\newcommand{\bC}{\mathbb{C}}
\newcommand{\bZ}{\mathbb{Z}}
\newcommand{\lct}{{\rm lct}}
\newcommand{\wt}{{\rm wt}}

\newcommand{\Val}{{\rm Val}}

\newcommand{\hvol}{{\widehat{\rm vol}}}

\newcommand{\Vol}{{\rm vol}}

\newcommand{\cF}{{\mathcal{F}}}
\newcommand{\bV}{{\mathbb{V}}}

\newcommand{\fb}{{\mathfrak{b}}}

\newcommand{\bQ}{{\mathbb{Q}}}

\newcommand{\cX}{{\mathcal{X}}}

\newcommand{\cL}{{\mathcal{L}}}
\newcommand{\ft}{{\mathfrak{t}}}
\newcommand{\cR}{{\mathcal{R}}}
\newcommand{\cY}{{\mathcal{Y}}}

\newcommand{\bG}{{\mathbb{G}}}

\newcommand{\bP}{{\mathbb{P}}}

\newcommand{\bin}{{\bf in}}

\newcommand{\cS}{{\mathcal{S}}}

\newcommand{\cA}{{\mathcal{A}}}

\newcommand{\sddb}{{\sqrt{-1}\partial\bar{\partial}}}
\newcommand{\cD}{{\mathcal{D}}}
\newcommand{\vphi}{{\varphi}}

\newcommand{\cE}{{\mathcal{E}}}

\newcommand{\Spec}{\mathrm{Spec}}

\newcommand{\mult}{{\rm mult}}
\newcommand{\bA}{{\mathbb{A}}}

\newcommand{\kb}{{\mathfrak{b}}}

\newcommand{\ka}{{\mathfrak{a}}}
\newcommand{\km}{{\mathfrak{m}}}

\newcommand{\DR}{{\mathcal{DR}}}
\newcommand{\dlt}{{\rm dlt}}

\newcommand{\Fut}{{\rm Fut}}
\newcommand{\rin}{\mathbf{in}}
\newcommand{\gr}{{\rm gr}}

\newcommand{\fD}{{\mathfrak{D}}}
\newcommand{\rint}{{\rm int}}
\newcommand{\bN}{{\mathbb{N}}}

\newcommand{\what}{\widehat}

\newtheorem{thm}{Theorem}[section]

\newtheorem{lem}[thm]{Lemma}
\newtheorem{def-prop}[thm]{Definition-Proposition}
\newtheorem{cor}[thm]{Corollary}
\newtheorem{defn}[thm]{Definition}

\newtheorem{prop}[thm]{Proposition}

\newtheorem{conj}[thm]{Conjecture}
\newtheorem{rem}[thm]{Remark}
\newtheorem{exmp}[thm]{Example}

\begin{document}

\title{Stability of Valuations: Higher Rational Rank}
\author{Chi Li and Chenyang Xu}
\date{}

\maketitle
\abstract{Given a klt singularity $x\in (X, D)$, we show that a quasi-monomial valuation $v$ with a finitely generated associated graded ring is the minimizer of the normalized volume function $\hvol_{(X,D),x}$, if and only if $v$ induces a degeneration to a K-semistable log Fano cone singularity.  Moreover, such a minimizer is unique among all quasi-monomial valuations up to rescaling. As a consequence, we prove that for a klt singularity $x\in X$ on the Gromov-Hausdorff limit of K\"ahler-Einstein Fano manifolds, the  intermediate K-semistable cone associated to its metric tangent cone is uniquely determined by the algebraic structure of $x\in X$, hence confirming a conjecture by Donaldson-Sun.}
\tableofcontents

\section{Introduction}

Throughout this paper, we work over the field $\mathbb{C}$ of complex numbers. In \cite{Li15a}, the normalized volume function was defined for any klt singularity $o\in (X,D)$ and the question of studying the geometry of its minimizer was proposed. 
See \cite{Li15b, Liu16, LL16, Blu16, LX16} for the results obtained recently. In particular, the existence of a minimizer of the normalized volume conjectured in \cite{Li15b} was confirmed in \cite{Blu16}. On the other hand, in \cite{LX16}, we intensively studied the case when the minimizer is a divisorial valuation.  In the current paper, we want to investigate the case when the minimizing valuations  are quasi-monomial of rational rank possibly greater than one. We note that this kind of cases do occur (see e.g. \cite{Blu16} or  Section \ref{sec-exmp}) and it was conjectured in \cite{Li15a} that any minimizer is quasi-monomial. 
 
 \subsection{The strategy of studying a minimizer}\label{ss-strategy}
After the existence being settled in \cite{Blu16}, the remaining work of the theoretic study of the minimizer of the normalized volume function is to understand its geometry (see the conjectures in \cite{Li15a, LX16}). Our method does not say much about the part of the conjecture saying that the minimizer must be quasi-monomial. Therefore, in the following we will always just assume that the minimizer is quasi-monomial. Partly inspired by the differential geometry theory on the metric tangent cone,  the strategy of understanding it consists of two steps, for which  we apply quite different techniques: 

\bigskip

In the first step, we deal with a special case of {\it a log Fano cone singularities}. Such a singularity has a good torus action and  a valuation induced by an element $\xi\in \ft_{\bR}^+$. This is indeed the case that has been studied in Sasakian-Einstein geometry on the link of a cone singularity (see e.g. \cite{MSY08, CS15} etc). In particular, when an isolated Fano cone singularity $(X_0, \xi_0)$ admits a Ricci-flat K\"{a}hler cone metric, it was shown in \cite{MSY08} that the normalized volume achieves its minimum at the associated Reeb vector field $\xi_0$ among all $\xi\in \ft^+_\bR$. Here we work on the algebraic side and only assume K-semistability instead of the existence of a Ricci-flat K\"{a}hler cone metric. We  also improve this result by removing the isolated condition on the singularity and more importantly showing that $\xi$ is  indeed the only minimizer among all quasi-monomial valuations centered at $o$, which form a much more complicated space than just the Reeb cone. In the rank 1 case, namely the case of a cone over a Fano variety, this question was investigated in \cite{Li15b, LL16} which used arguments from \cite{Fuj15}. Here to treat the higher rank case, we work along a somewhat different approach using more ingredients from the convex geometry inspired by a circle of ideas from the Newton-Okounkov body construction (see e.g. \cite{Oko96, LM09}) (see Section \ref{s-Tvarieties}). 

\medskip

In the second step, given a quasi-monomial valuation in $\Val_{X,x}$ which minimizes the normalized volume function $\hvol_{(X,D)}$, we aim to obtain a degeneration from the klt singularity $o\in(X,D)$ to a log Fano cone singularity $(X_0,D_0, \xi)$, and then we can study this degeneration family to deduce results for general klt singularities from the results for log Fano cone singularities.  Our study of the degeneration heavily relies on recent developments in the minimal model program (MMP)  based on \cite{BCHM10} (cf. e.g.  \cite{LX14, Xu14} etc.). In the rank 1 case, i.e., when the minimizing valuation is divisorial, we showed in \cite{LX16} that the valuation yields a Koll\'ar component (see also \cite{Blu16}).  With such a Koll\'{a}r component, we can indeed complete the picture. For the case that the quasi-monomial minimizing valuation has a higher rational rank, the birational models we construct should be considered as asymptotic approximations. In particular, unlike the rank 1 case,  we can not conclude the conjectural finite generation of the associated graded ring. Thus we have to post it as an assumption. Nevertheless, once we have the finite generation, then we can degenerate both the approximating models and the sequences of ideals to establish a process of passing the results obtained in the Step 1 for the log Fano cone singularity in the central fiber to prove results for the general fiber (see Section  \ref{s-constructmodel}). 

\bigskip


\subsection{Geometry of minimizers}
Now we give a precise statement of our results. Assume $x\in X$ is a singularity defined by the local ring $(R, \fm)$. Denote by $\Val_{X,x}$ the set of real valuations on $R$ that are centered at $\fm$. For any $v\in \Val_{X,x}$, we denote by  ${\rm gr}_v(R)$ the associated graded ring. 
First we prove the following result, which partially generalizes \cite[Theorem 1.2]{LX16} to the higher rank case. We refer to Section \ref{sec-Pre} for other notations used in the statements. 
\begin{thm}\label{t-unique}
Let $x\in (X,D)$ be a klt singularity.  Let $v$ be a quasi-monomial valuation in $\Val_{X,x}$ that minimizes $\hvol_{(X,D)}$ and has a  finitely generated associated graded ring  ${\rm gr}_v(R)$. Then the following properties hold:
\begin{enumerate}
\item There is a natural divisor $D_0$ defined as the degeneration of $D$ such that  
$$\big(X_0:={\rm Spec}\big({\rm gr}_v(R)\big), D_0\big)$$ is a klt singularity;
\item $v$ is a K-semistable valuation;
\item Let $v'$ be another quasi-monomial valuation in $\Val_{X,x}$ 
that minimizes $\hvol_{(X,D)}$. Then $v'$ is a rescaling of $v$.
\end{enumerate}
\end{thm}
For the definition of K-semistable valuations, see Definition \ref{d-ksemi}. The definition uses the notion of K-semistability of log Fano cone singularities (see \cite{CS12}), which in turn generalizes the original K-semistability introduced by Tian (\cite{Tia97}) and Donaldson (\cite{Don01}).
 In fact, this leads to a natural refinement of \cite[Conjecture 6.1]{Li15a}.
 \begin{conj}Given any arbitrary klt singularity $x\in (X={\rm Spec}(R),D)$. The unique minimizer $v$ is quasi-monomial with a finitely generated associated graded ring, and the induced degeneration
 $$(X_0={\rm Spec}(R_0),D_0, \xi_v)$$ is K-semistable. In other words, any klt singularity $x\in (X,D)$ always has a unique K-semistable valuation up to rescaling.  
 \end{conj}
\bigskip

\bigskip

 We shall also prove the following converse to Theorem \ref{t-unique}.2, which  was known in the rank 1 case by \cite[Theorem 1.2]{LX16}.
 
 \begin{thm}\label{t-KtoM}
 If $(X,D)$ admits a $K$-semistable valuation $v$ over $o$, then $v$ minimizes $\hvol_{(X,D)}$. In particular, if $(X,D,\xi)$ is a K-semistable Fano cone singularity, then $\wt_{\xi}$ is a minimizer of $\hvol_{(X,D)}$.
 \end{thm}

\subsection{Applications to singularities on GH limits}

As proposed in \cite{Li15a}, one main application of our work is to study a singularity $x\in X$ appearing on any Gromov-Hausdorff limit (GH) of K\"ahler-Einstein Fano manifolds. 
By the work of Donaldson-Sun and Tian (\cite{DS14, Tia90, Tia12}), we know that $M_\infty$ is homeomorphic to a normal algebraic variety. Donaldson-Sun (\cite{DS14}) also proved that $M_\infty$ has at worst klt singularities.
For any point on the Gromov-Hausdorff limt, we can consider the metric tangent cone $C$ using the limiting metric (see \cite{CC97, CCT02}). In \cite{DS15},  Donaldson-Sun described $C$ as a degeneration of a cone $W$ and they conjectured that both $C$ and $W$ only depend on the algebraic structure of the singularity. Intuitively, $W$ and $C$ should be considered respectively as a `canonical' K-semistable and a K-polystable degeneration of $o\in M_{\infty}$. We refer to Section \ref{sec-DS} for more details of this conjecture. 

Here we want to verify the semistable part of the conjecture, namely we will show that the valuation used to construct $W$ in \cite{DS15} is  K-semistable, and since such a valuation is unique by Theorem \ref{t-unique}, it does not depend on what kind of metric it carries but only the algebraic structure. We note that all the assumptions in Theorem \ref{t-unique} are automatically satisfied in this situation, which essentially follows from the work of \cite{DS15}. 
\begin{thm}[{see \cite[Conjecture 3.22]{DS15}}]\label{t-finiteunique}
 Denote by ${\rm Spec}(R)$ the germ of a singularity $o$ on a Gromov-Hausdorff limit $M_{\infty}$ of a sequence of  K\"ahler-Einstein Fano manifolds. Using the notation in \cite{DS15} (see Section \ref{ss-DS}),  the cone $W$  associated to the metric tangent cone is isomorphic to ${\rm Spec}({\rm gr}_v(R))$ where $v$ is the unique K-semistable valuation in $\Val_{M_{\infty},o}$.  In particular, it is uniquely determined by the algebraic structure of the singularity. 
\end{thm}

A standard point of the proof of Theorem \ref{t-finiteunique} is to deduce the stability from the metric. For now, we only need the following statement, which generalize \cite{CS12} to the case of non-isolated singularities (see Remark \ref{r-flatcone}.) 

\begin{thm}[=Theorem \ref{thm-irSE}]\label{thm-irSemi}
If $(X, \xi_0)$ admits a Ricci-flat K\"{a}hler cone metric, then $A_X(\xi_0)=n$ and $(X, \xi_0)$ is K-semistable.
\end{thm}

\begin{rem}\label{r-polystable}
In a forthcoming work \cite{LWX17}, we plan to complete the proof of  \cite[Conjecture 3.22]{DS15} by showing that the metric tangent cone $C$ is also uniquely determined by the algebraic structure of the germ  $(o\in M_{\infty})$.
In fact, after Theorem \ref{t-finiteunique}, what remains to show is that $C$ only depends on the algebraic structure of the K-semistable Fano cone singularity $(W,\xi_v)$. It follows from a combination of two well-expected speculations. The first one is an improvement of Theorem \ref{thm-irSemi} which says that  a Fano cone $C$ with a Ricci flat K\"ahler cone metric is indeed K-polystable. This was solved in \cite{Ber15} for the quasi-regular case.  In \cite{LWX17}, we plan to extend \cite{Ber15} to the irregular case. 
 The second one is that any K-semistable log Fano cone  has a unique K-polystable log Fano cone degeneration. In the case of smoothable Fano varieties (a regular case), this was proved independently in \cite{LWX15} and \cite {SSY16} using analytic tools. 
In \cite{LWX17}, we will use the algebraic tools developed in \cite{Li15b,LX16} and the current paper, especially the ones related to $T$-equivariant degeneration, to investigate the problem again and give a purely algebro-geometric treatment. 
 \end{rem}

As a consequence of our work,  we also obtain the following formula for singularities appearing on the Gromov-Hausdorff limits (GH) of K\"ahler-Einstein Fano manifolds, which sharpens \cite[Proposition 3.10]{SS17}  (see Corollary \ref{c-v=v}) as well as partially \cite[Theorem 1.3.4]{LX17}.
\begin{thm}\label{t-finitedegree}
Let $M_{\infty}$ be a GH limit of K\"{a}hler-Einstein Fano manifolds. Let $o\in M_{\infty}$ be a singularity, and $\pi\colon (Y,y)=(R',\fm_y)\to (M_{\infty},o)=(R,\fm_o)$ be a quasi-\'etale finite map, that is $\pi$ is \'etale in codimension one, then 
$$\hvol(Y,y)=\deg(\pi)\cdot\hvol(M_{\infty},o).$$
\end{thm}
In fact, for a quasi-\'etale finite covering $(Y,y)\to (X,x)$ between any klt singularities, as we expect the minimizer of $y\in Y$ is  unique and thus $G$-invariant,  such a formula should also hold. However,  for now we can not prove this in the full generality.

\bigskip




\subsection{Outline of the paper}

In this section, we give an outline as well as the organization of the paper. The paper is divided into two parts. In Part \ref{p-1}, we study the geometry of the minimizer of a klt singularity in general. We note that this part is completely {\it algebraic}.  

\bigskip

In Section \ref{s-preliminary}, we recall a few concepts and establish some background results, especially on valuations and $T$-varieties (i.e. varieties with a torus action). 

\bigskip

In Section \ref{s-Tvarieties}, we focus on studying log Fano cone singularity  (see Definition \ref{d-FS}) and put it into our framework as mentioned in Section \ref{ss-strategy}. In particular, we want to show that a K-semistable  log Fano cone singularity does not only minimize the normalized volumes among the valuations in the Reeb cone, but indeed also among all valuations in $\Val_{X,x}$. Furthermore, it is unique among all quasi-monomial valuations. 
 Our main approach is to use the ideas from the construction of Newton-Okounkov body to reduce the volume of valuations to the volume of convex bodies and then apply the known convexity of the volume function in such setting. 
This is obtained by three steps with increasing generality: we first consider toric singularities with toric valuations (see Section \ref{ss-toric}) where we set up the convex geometry problem; then general $T$-singularities with toric valuations (see Section \ref{sss-Ttoric}); and eventually $T$-singularities with $T$-invariant valuations (see Section \ref{sss-TT}). 

\bigskip

In Section \ref{s-models}, we investigate a quasi-monomial minimizer of a general klt singularities by intensively using the minimal model program and degeneration techniques.
First in Section \ref{s-vmodel}, we show that given a quasi-monomial minimizer $v\in \Val_{X}$ of $\hvol_{(X,D)}$, one can find birational models which can be considered to approximate the valuation.  This construction will be used later if we assume that the degeneration exists, i.e., the associated graded ring is finitely generated.  
In particular, we conclude that such a degeneration is also klt. Then in Section \ref{s-Ksta}, we show that any quasi-monomial minimizer is K-semistable, and in Section \ref{s-uniqueness}, we use the degeneration technique and the uniqueness of the quasi-monomial minimizer for a log Fano cone singularity obtained in Section \ref{s-Tvarieties} to conclude the uniqueness of the quasi-monomial minimizer for a general singularity if one minimizer has a finitely generated associated graded ring. 
\bigskip

In Part \ref{p-dg} (=Section \ref{sec-DS}), we apply our theory to the singularities appeared on the Gromov-Hausdorff limit $M_{\infty}$ of K\"ahler-Einstein Fano manifolds. Obviously, we need to establish the results that connect the previous existing differential geometry work to our algebraic results in Part \ref{p-1}. Our aim is to show that the semistable cone $W$ associated to the metric tangent cone in Donaldson-Sun's work depends only on the algebraic structure of the singularity. It follows from the work in \cite{DS15}  such a cone is induced by a valuation $v$ and it always carries an (almost) Sasakian-Einstein metric. What remains to show is then such a cone is K-semistable, i.e., the valuation $v$ is a K-semistable quasi-monomial valuation.  This is achieved in Section \ref{s-tangentconeksemistable}. Then the finite degree multiplication formula is deduced in Section \ref{s-finitedegree}. 
\bigskip

In Appendix \ref{sec-exmp}, we illustrate our discussion on $n$-dimensional $D_{k+1}$ singularities. In particular, we verify that all the candidates calculated out in \cite{Li15a}, including all those irregular ones,  are indeed minimizers of $\hvol$ (except possibly for $D_4$ in dimension 4).
\bigskip

\noindent {\bf Acknowledgement}: We want to thank Harold Blum, Yuchen Liu, Mircea Musta\c{t}\v{a} and Gang Tian for helpful discussions and comments. 
CL is partially supported by NSF DMS-1405936 and Alfred P. Sloan research fellowship.
Part of this work was done during CX's visiting of the Department of Mathematics in MIT, to which he wants to thank the inspiring environment.   CX is partially sponsored by `The National Science Fund for Distinguished Young Scholars (11425101)'.

\part{Geometry of minimizers}\label{p-1}

\section{Preliminary and background results}\label{s-preliminary}\label{sec-Pre}

{\bf Notation and Conventions}: We follow \cite{KM98} and \cite{Kol13} for the standard conventions in birational geometry. For a log pair $(X,D)$, we also use $a_l(E;X,D)$ to mean the log discrepancy of $E$ with respect to $(X,D)$, i.e., $a_l(E;X,D)=a(E;X,D)+1$. Similarly, we also define $a_l(E; X,D+c\cdot \fa)$ for an ideal $\fa\subset \mathcal{O}_X$ and $c\ge 0$. 

When $x\in X$, we use $\Val_{X,x}$ to denote all the real valuations of $K(X)$ whose center on $X$ is $x$. If $x\in (X,D)$ is a klt singularity, for any valuation $v\in \Val_{X,x}$, we can define its log discrepancy $A_{(X,D)}(v)$ (see \cite{JM12}) which is always positive, and its volume $\vol_{X,x}(v)$. Following \cite{Li15a}, we define the normalized volume
$$\hvol_{(X,D),x}(v)=A_{(X,D)}^n(v)\cdot \vol_{X,x}(v)$$
if $A_{X,D}(v)<+\infty$, or $\hvol_{(X,D),x}(v)=+\infty$ if $A_{X,D}(v)=+\infty$.
Let $X=\Spec(R)$ be a germ of algebraic singularity with $x\in X$ corresponding to the maximal ideal $\fm$. We will denote by ${\rm PrId}_{X,x}$ the set of all the $\fm$-primary ideals.

\medskip
In this note, any singularity $x\in X$ means the germ of an algebraic singularity, i.e., $X={\rm Spec}(R)$ where $R$ is essentially of finite type over $\mathbb{C}$.
Let $x\in (X,D)$ be a klt singularity. We call a divisorial valuation $\ord_S$ in $\Val_{X,x}$ gives a Koll\'ar component if there is a model $\mu\colon Y\to X$  isomorphic over  $X\setminus \{x\} $ with the exceptional divisor given by $S$ such that $(Y,\mu_*^{-1}D+S)$ is plt and  $-K_Y-\mu_*^{-1}D-S$ is ample over $X$. We denote by ${\rm Kol}_{X,D,x}$ the set of  Koll\'{a}r components over $x\in (X,D)$. 

\subsection{Normalized volumes}

In this section, we summarize some known results of the normalized volumes. 

\begin{defn}\label{Def-volX}
For any  klt singularity $x\in (X, D)$, we define its normalized volume to be the following positive number:
\[
\hvol(X,D, x):=\inf_{v\in \Val_{X,x}}\hvol(v).
\]
\end{defn}
We often abbreviate it as $\hvol(X,D)$ if $x$ is clear or $\hvol(X,x)$ if $D=0$. 
We have the following description of $\hvol(X,D)$ using the (normalized) multiplicities.

\begin{thm}[\cite{Liu16}]\label{t-liu}There is an equality
\begin{equation}\label{e-liu}
\hvol(X,D)=\inf_{v\in \Val_{X,x}}\hvol(v)=\inf_{\ka \in {\rm PrId}_{X,x}}\lct_{(X,D)}^n(\ka)\cdot \mult(\ka).
\end{equation}
\end{thm}

In \cite{LX16}, we showed that the right hand of \eqref{e-liu} is obtained by a minimizer $\fa$ if and only if the minimum is calculated by the valuative ideals of a Koll\'ar component. In general, \cite{Blu16} showed that if we replace an ideal $\fa$ by a graded ideal sequence $\{\fa_{\bullet}\}$, the minimum can always be obtained. Then we can easily show that a valuation $v$ that computes the log canonical threshold of such a graded ideal sequence $\{\fa_{\bullet}\}$ satisfies the identity $\hvol_{(X,D)}(v)=\hvol(X,D)$. 

\begin{thm}[\cite{Blu16}] There exists a valuation $v\in \Val_{X,x}$ such that 
$$\hvol(X,D)=\hvol_{(X,D)}(v).$$
\end{thm}

\begin{rem}
In \cite{JM12}, it was conjectured that any valuation computing the log canonical threshold of a graded sequence of ideals at a smooth point is quasi-monomial (see \cite[Conjecture B]{JM12}). We can naturally extend this conjecture from a smooth point to a klt pair $(X,D)$ and then have the following fact observed in \cite{Blu16}:
The strong version of Conjecture B for klt pair $(X,D)$ in \cite{JM12} implies that any minimizer $v$ of $\hvol_{(X,D)}$ is quasi-monomial. 
\end{rem}

Another characterization of the normalized volume is by using the volumes of models.

\begin{thm}[\cite{LX16}]\label{t-LX16}
For any model $Y\to X$ which is isomorphic over  $X\setminus \{x\} $ with a  nontrivial exceptional divisor, we can define its volume $\vol(Y/X)$ as in \cite{LX16}. Then  we have
$$\hvol(X,D)=\inf_{\small Y\to X \mbox{ as above} } \vol(Y/X)=\inf_{
S\in {\rm Kol}_{X,D,x}} \hvol(S).$$
\end{thm}

\subsection{Approximation}
 For the latter, we need the following result.   

\begin{lem}\label{l-diri}
Given real numbers $\alpha_i$ $(i=1,2,..,r)$  such that $\alpha_1,...,\alpha_r$ and 1 are $\mathbb{Q}$-linearly independent. Let $\delta_i\in \{-1,1\}$. Then for any $\epsilon>0$, we can find 
$p_1,....,p_r$ and $q\in \mathbb{Z}$ such that for any $i$ $(i=1,2,..,r)$,
  $$ 0<\delta_i\cdot (\frac{p_i}{q}-\alpha_i)\le \frac{\epsilon}{q}.$$
\end{lem}
\begin{proof}Let $v=(\alpha_1,...,\alpha_r)\in \mathbb{R}^r$ be a vector, then we consider the sequence 
$$\{v,2v,...,nv,...\} \ \ \mod \mathbb{Z}^r.$$ By Weyl's criterion for equidistribution, using the assumption that $\alpha_1,...,\alpha_r$ and 1 are $\mathbb{Q}$-linearly independent, we know that this sequence is equidistributed in $[0,1]^r$. So our lemma follows easily.   
\end{proof}

 We denote the norm $|\cdot |$ on $\mathbb{R}^r$ to be $|x|=\max_{1\le i\le r}|x_i|$.


\begin{lem}\label{l-approvector}
Let $v\in \mathbb{R}^r$ be a vector. For any $\epsilon>0$. There exists $r$ rational vectors $\{v_1$, $v_2$,..., $v_r\}$ and integers $\{q_1$,...,$q_r\}$  such that
\begin{enumerate}
\item $q_{i}v_i\in \mathbb{Z}^r$,
\item $v$ is in the convex cone generated by $v_1$,..., $v_r$, i.e., $v=\sum a_iv_i$ for some $a_i>0$; and  
\item $|v_i-v|<\frac{\epsilon}{q_i}$.
\end{enumerate}
\end{lem}
\begin{proof}After relabelling, we can assume that $v=(\alpha_1,...,\alpha_r)$ with $(1, \alpha_1,... ,\alpha_j)$ is linearly independent and span the space ${\rm span} (1, \alpha_1,...,\alpha_r)$. Then possibly replacing $\epsilon$ by a smaller one, once we could approximate $\alpha_1$,..., $\alpha_j$, we automatically get the approximation of all $\alpha_i$ $(1\le i \le r)$. Therefore, we may and will assume that  $(1, \alpha_1,... ,\alpha_r)$ is linearly independent.

Applying Lemma \ref{l-diri} for all $2^r$ choices of $\delta_1,...,\delta_r$,  we find $v_1$,...., $v_{2^r}$ vectors, it suffices to show that we can choose $r$ vectors out of them so that condition (2) is satisfied.   

Let $w_i=v_i-v$, then we know that the signs of the components of $w_1$,..., $w_{2^r}$ exhaust all $2^r$ possible. Clearly it suffices to show that $0$ can be written as a positive linear combination of $w_1$,..., $w_{2^r}$. We prove this by induction on $r$. Let
$w_1,...,w_{2^{r-1}}$ be precisely the vectors with positive  first component. Then using the induction, we know that there exist positive numbers  $a_1,...,a_{2^{r-1}}$ such that $\sum^{2^{r-1}}_{i=1} a_iw_i$ has the form $(a,0,...,0)$ with $a>0$. Similarly, we find positive numbers $a_{2^{r-1}+1},...,a_{2^r}$ such that $\sum^{2^{r}}_{i=2^{r-1}+1} a_iw_i$ is of the form $(-b,0,...,0)$ with $b>0$. Then we know that
$$(ba_1)w_1+\cdots+(ba_{2^{r-1}})w_{2^{r-1}}+(aa_{2^{r-1}+1})w_{2^{r-1}+1}+\cdots+(aa_{2^r})w_{2^r}=0.$$
\end{proof}

\subsection{Valuations and associated graded ring}\label{ss-valuation}



Given a valuation $v$ whose valuative semi-group is denoted by $\Phi$. Let $\Phi^g$ be the corresponding group generated by $\Phi$, and $\Phi^g_+\supset \Phi$ be the semi-group of nonnegative values. We can define the generalized Rees algebra as in \cite[Section 2.1]{Tei03}
\begin{equation}\label{eq-Rees}
\mathcal{R}_v:=\bigoplus_{\phi\in \Phi^g} \fa_{\phi}(v) t^{-\phi} \subset \mathcal{R}[t^{\Phi^g}]. 
\end{equation}

By \cite[Proposition 2.3]{Tei03}, we know $\mathcal{R}_v$ is faithfully flat over $k[\Phi^g_+]$. 

\bigskip

The valuations that our approach can deal with are called {\it quasi-monomial valuations}. It is known that it is the same as  Abhyankar valuation (see e.g. \cite[Proposition 2.8]{ELS03}). 

\begin{defn}\label{d-quasimono}
$v=v_{\alpha}$ is called a quasi-monomial valuation over $x$ of rational rank $r$: if there exists a log resolution $Z\to X$ with divisors $E_1$,...., $E_r$ over $x$, and $r$ nonnegative numbers $\alpha=(\alpha_1$,...., $\alpha_r)$ which are $\mathbb{Q}$-linearly independent, such that 
\begin{enumerate}
\item $\bigcap^r_{i=1} E_i\neq \emptyset$;
\item There exists a component $C\subset \cap E_i$, such that around the generic point $\eta$ of $C$, $E_i$ is given by the equation $z_i$ and 
\item  for $f\in \mathcal{O}_{X,\eta}\subset  \hat{\mathcal{O}}_{X,\eta}$ can be written as $f=\sum c_{\beta}z^{\beta}$, with either $c_{\beta}=0$ or $c_{\beta}(\eta)$ is a unit,
then 
$$v_{\alpha}(f)=\min \{\sum \alpha_i\beta_i|\ c_{\beta}(\eta)\neq 0\}.$$
\end{enumerate}
\end{defn}

In fact, for any $(\alpha'_1,...., \alpha'_r)\in \mathbb{R}^r_{\ge 0}$, we can define a valuation as above, which is quasi-monomial with rational rank equal to the dimension of the 
$\mathbb{Q}$-linear vector space 
$${\rm span}_{\mathbb{Q}}\{\alpha'_1,..., \alpha_r'\}\subset \mathbb{R}.$$ On the other hand, if we choose $\alpha\in \mathbb{Z}^r_{\ge 0}$, the corresponding valuation is given by the toroidal divisor coming from the weighted blow up with weight $\alpha$. 
\bigskip

In the following, we also need a slight generalization of Definition \ref{d-quasimono}, namely instead of assuming that $(Z, E_1+\cdots+E_r)$ is simple normal crossing at $\eta$, we assume
$$(\eta\in Z,E_1+\cdots +E_r)\cong (\eta' \in Z',E'_1+\cdots +E'_r)/G,$$
where $(\eta' \in Z',E'_1+\cdots +E'_r)$ is a semi-local snc scheme, and $ G$ is an abelian group. Denote the pull back of $E_i$ by $n_iE_i'$. Then for $\alpha=(\alpha_1,...., \alpha_r)\in \mathbb{R}^r_{\ge 0}$, we can define $v_{\alpha}$ to be the restriction of $v_{\alpha'}$ at $\eta'\in (Z',E'_1+\cdots +E'_r)$ as in Definition \ref{d-quasimono}, where $\alpha'=(n_1\alpha_1,..., n_r\alpha_r)$. 

\begin{defn}\label{d-quasi}
Let $(Z,E)$ be a pair which is (Zariski) locally an abelian group quotient of a snc pair (with reduced boundary).  Denote the center of $v$ on $(Z,E)$ by $\eta$, we say $v$ is {\it computed on  $\eta \in (Z,E)$}, if
\begin{enumerate}
\item $\overline{\eta}$ is a component of the intersection of component of $E$;  
\item $v=v_{\alpha}$ for some $\alpha\in \mathbb{R}^{r}_{>0}$.
\end{enumerate}

For a general pair of a normal variety $Z$ and a $\mathbb{Q}$-divisor $E$, we say that $v$ is computed on the center $\eta$ of $v$ in $(Z,E)$ if in a neighborhood $U$ of $\eta$, $(U, E|_U(=E^{=1}|_U))$ is locally a quotient of snc pair and $v$ is computed on $(U, E|_U(=E^{=1}|_U))$ in the above sense. 
\end{defn}

\begin{lem}\label{l-finiteass}
Let $v_{\alpha}$ be a quasi-monomial valuation as defined in Definition \ref{d-quasimono} with the maximal rational rank $r$, whose associated graded ring $ {\rm gr}_{v_{\alpha}}(R)$ is finitely generated. Then we can choose $\epsilon$ sufficiently small such that for any  $\alpha'\in \mathbb{Q}^r$ with $|\alpha-\alpha'|<\epsilon$, 
there is an isomorphism ${\rm gr}_{v_{\alpha}}(R)\to {\rm gr}_{v_{\alpha'}}(R)$ induced by a morphism sending a set of homogeneous generators of ${\rm gr}_{v_{\alpha}}(R)$ to one of ${\rm gr}_{v_{\alpha'}}(R)$. 
\end{lem}
\begin{proof}Since ${\rm gr}_{v_{\alpha}}(R)$ is finitely generated, we can find a finite set of homogenous generators $f_1,...,f_k$, and construct a surjection of graded rings
$$\phi\colon k[t_1,...,t_k]\to {\rm gr}_{v_{\alpha}}(R),\qquad t_i\to f_i,$$
where $t_i$ has the same degree as $\deg(f_i)$, which by our assumption can be written as $\sum^r_{i=1}b^j_{i}\alpha_i$. 

We can lift $f_i$ to $g_i\in R$ such that $\bin_v(g_i)=f_i$, and make the morphism
$$\psi\colon k[t_1,...,t_k]\to R,\qquad t_i\to g_i.$$
Consider the filtration 
$$\{ F_b\subset k[t_1,...,t_k]|b\in \Phi, f\in F_b\mbox{ iff all monomials of $f$ have degree at least $b$}  \}$$
for any $b\in \Phi$.
We can similarly construct a Rees algebra
$$\mathcal{R}^*= \sum_{b\in  {\Phi}^g} F_{b} s^{b} \subset k[t_1,...,t_k]\otimes k[\Phi^g].$$
There is a surjection $\mathcal{R}^*\to \mathcal{R}_{v_{\alpha}}$, which degenerates $\psi$ to $\phi$ over ${\rm Spec}R(\Phi)$. Denote by $I$ the kernel of $\psi$, which  we know degenerates to $I_0$ the Kernel of $\phi$. By the flatness over $R(\Phi)$, we know any element $h'\in I_0$ can be lift to an element $h\in I$. 
Geometrically, this gives a pointed embedding $(x\in X)\subset (0\in \mathbb{C}^k)$, whose degeneration along one direction, denoted by $\xi_{\alpha}$, on $\mathbb{C}^n$ gives the embedding $(o\in X_0={\rm Spec}({\rm gr}_vR))\subset (0\in \mathbb{C}^r)$. 

\bigskip


 Let $h_1,....,h_m$ be elements in $I$ whose degenerations $h_1',..., h_m'$ generate $I_0$. 
Assume $h_j=h'_j+h_j''$ with $\deg(h'_j)=\sum^r_{i=1} c^j_i\alpha_i$ and the monomials of $h_j''$ has degree larger than $\sum^r_{i=1} c^j_i\alpha_i$. We can choose $\epsilon$ sufficiently small such that if $\alpha'=(\alpha'_1,...., \alpha'_r)$ satisfies $|\alpha'-\alpha|<\epsilon$, 
then any monomial of $h_j''(t_1,...,t_k)$ has a corresponding degree larger than $c^j_i\alpha'_i$ where $\deg(t_i)$ is set to be $\sum^r_{i=1}b^j_{i}\alpha'_i$. Then the condition that $v$ has the maximal rational rank implies that $h_j'$ is the leading term of $h_j$ if $\deg(t_i)=\sum^r_{i=1}b^j_{i}\alpha'_i$.

Consider the filtration given by setting the degree of $t_j$ to be $\sum_i b^j_i\alpha'_i$ and it induces a filtration on $R$ by its image. We denote the corresponding vector of the degeneration by $\xi_{\alpha'}$.  Our argument above says for the filtration induced by $\xi_{\alpha'}$,  the associated graded ring given by the filtration coincides with ${\rm gr}_v(R)$. Since the graded ring is an integral domain, by Lemma \ref{lem-quasi}(1) the filtration comes from a valuation $v'$. Now we claim 
$v'$ is the same as the one given by $v_{\alpha'}$, which implies what we aim to prove.  

\bigskip

In fact, by \cite{Tev14}, we know that for our embedding  $(x\in X)\subset (0\in \mathbb{C}^k)$, we can indeed assume there is a toroidal morphism $V\to \mathbb{C}^k$, such that the birational transform of $X$ in $V$ gives the model $Z$ in Definition \ref{d-quasimono} and the divisors $E_i$ are from the transversal intersection of $Z$ and components $F_i$ of ${\rm Ex}(V/\mathbb{C}^k)$. Assume $F_i$ at $\eta\in \cap^r_{i=1} F_i$ yield a coordinate $y_i$, and its restriction to $Z$ induces the coordinate $z_i$ for $E_i$. By the transversality of $Z$ and components of $ {\rm Ex} (V/\mathbb{C}^k)$, we know that for any function $f\in R$, if we lift it as a function $\tilde{f} \in k[x_1,...,x_k]$, then
$$v_{\alpha'}(f)= \xi_{\alpha'}(\tilde{f}),$$
where $\xi_{\alpha'}$ is the corresponding toroial valuation induced by $\alpha'$ over $\mathbb{C}^k$. However, by our definition $\xi_{\alpha'}(\tilde{f})=v'(f)$, which implies $v_{\alpha'}=v'$. 
\end{proof}

Let $\Phi^g$ be an ordered subgroup of the real numbers $\bR$. Let $(R, \fm)$ be the local ring at a normal singularity $o\in X$. A $\Phi^g$-graded filtration of $R$, denoted by $\cF:=\{\cF^m\}_{m\in \Phi^g}$, is a decreasing family of $\fm$-primary ideals of $R$ satisfying the following conditions:

{\bf (i)} $\cF^m\neq 0$ for every $m\in \Phi^g$, $\cF^m=R$ for $m\le 0$ and $\cap_{m\ge 0}\cF^m=(0)$;

{\bf (ii)} $\cF^{m_1}\cdot \cF^{m_2}\subseteq \cF^{m_1+m_2}$ for every $m_1, m_2\in \Phi^g$.

Given such an $\cF$, we get an associated order function 
$$v=v_{\cF}: R\rightarrow \bR_{\ge 0} \qquad v(f)=\max\{m; f\in \cF^m\} \mbox{\ \ for any $f\in R$}.$$ Using the above {\bf (i)-(ii)}, it is easy to verify that $v$ satisfies $v(f+g)\ge \min\{v(f), v(g)\}$ and $v(fg)\ge v(f)+v(g)$. We also have the associated graded ring:
\[
\gr_{\cF}R=\sum_{m\in \Phi^g} \cF^m/\cF^{>m}, \text{ where } \cF^{>m}=\bigcup_{m'> m}\cF^{m'}.
\]
For any real valuation $v$ with valuative group $\Phi^g$, $\{\cF^m\}:=\{\fa_m(v)\}$ is a $\Phi^g$-graded filtration of $R$. 
We need the following known facts.
\begin{lem}[see \cite{Tei03, Tei14}]\label{lem-quasi}
With the above notations, the following statement holds:

{\rm (1)} (\cite[Page 8]{Tei14}) If $\gr_{\cF}R$ is an integral domain, then $v=v_{\cF}$ is a valuation centered at $o\in X$. In particular, $v(fg)=v(f)+v(g)$ for any $f,g\in R$.

{\rm (2) (\cite{Pil94})} A valuation $v$ is quasi-monomial if and only if the Krull dimension of $\gr_v R$ is the same as the Krull dimension of $R$.
\end{lem}

\subsection{Singularities with good torus actions}

For general results of $T$-varieties in algebraic geometry,  see  \cite{AH06, PS08, LS13, AIPSV11}. 
Assume $X={\rm Spec}_{\bC}(R)$ is an affine variety with $\bQ$-Gorenstein klt singularities. Denote by $T$ the complex torus $(\bC^*)^r$. Assume $X$ admits a good $T$-action in the following sense. 
\begin{defn}[see \cite{LS13}]\label{d-good}
Let $X$ be a normal affine variety. We say that a $T$-action on $X$ is {\it good} if 
it is effective and there is a unique closed point $x\in X$ that is in the orbit closure of any $T$-orbit. We shall call $x$ the vertex point of the $T$-variety $X$.
\end{defn}
For a singularity $x\in X$ (sometimes also denote by $o\in X$) with a good $T$-action, we will also call it a $T$-singularity for simplicity. 

Let $N={\rm Hom}(\mathbb{C}^*, T)$ be the co-weight lattice and $M=N^*$ the weight lattice. We have a weight space decomposition 
\[
R=\bigoplus_{\alpha\in \Gamma} R_\alpha \mbox{  where \ } \Gamma =\{ \alpha\in  M  |\ R_{\alpha}\neq 0\}.
\]
The action being good implies $R_0=\mathbb{C}$, which will always be assumed in the below. An ideal $\ka$ is called homogeneous if $\fa=\bigoplus_{\alpha\in \Gamma}\ka\cap R_\alpha$. Denote by $\sigma^{\vee}\subset M_{\mathbb{Q}}$ the cone generated by $\Gamma$ over $\mathbb{Q}$, which is called the {\it weight cone} (or the {\it moment cone}).  
The cone $\sigma\subset N_{\bR}$, dual to $\sigma^\vee$, is the same as the following set 
$$\mathfrak{t}^+_{\bR}:=\{\ \xi \in N_{\mathbb{R}}\ \ | \  \langle \alpha, \xi \rangle>0 \mbox{ for any }\alpha\in \Gamma\backslash\{0\}\}.$$

For the convenience and by comparison with Sasaki geometry, we will introduce:
\begin{defn}\label{defn-Reeb}
With the above notations, a vector $\xi\in \ft^+_\bR$ will be called a Reeb vector on the $T$-variety $X$.
\end{defn} 


We recall the following structure results for any $T$-varieties.
\begin{thm}[{\cite[Theorem 3.4]{AH06}}]
Let $X={\rm Spec}(R)$ be a normal affine variety and suppose $T={\rm Spec}\left(\bC[M]\right)$ acts effectively on $X$ with the weight cone $\sigma^{\vee} \subset M_{\bQ}$. Then there exists a normal semiprojective variety $Y$ such that $\pi \colon X\to Y$ is the good quotient under $T$-action and a polyhedral divisor $\fD$ and there is an isomorphism of graded algebras:
\[
R\cong H^0(X, \mathcal{O}_X)\cong \bigoplus_{u\in \sigma^{\vee} \cap M} H^0 \big(Y, \mathcal{O}(\fD(u))\big)=: R(Y, \fD).
\] 
In other words, $X$ is equal to ${\rm Spec}_\bC\big( \bigoplus_{u\in \sigma^{\vee} \cap M} H^0(Y, \mathcal{O}(\fD(u)) ) \big)$.
\end{thm}
Here a variety $Y$ being semiprojective variety means it is projective over an affine variety $Z$, which can be chosen to be
$Z={\rm Spec}(H^0(Y,\mathcal{O}_Y))$. 

We refer to Example \ref{exmp-Tvar} for a concrete example.
From now on, let $X$ be an affine variety with a good action such that $o\in X$ is the vertex point. Then we know that $Y$ is projective since 
$$H^0(Y,\mathcal{O}_Y)=R^T=R_0=\mathbb{C}$$ (see \cite{LS13}). We collect some known results about valuations on $T$-varieties. 

\begin{thm}\label{t-Tcano}
Assume a $T$-variety $X$ is determined by the data $(Y, \sigma, \fD)$ such that $Y$ is a projective variety, $\sigma$ is a maximal dimension cone in $N_{\bR}$ and $\fD$ is a polyhedral divisor.\begin{enumerate}
\item
For any $T$-invariant quasi-monomial valuation $v$, there exist a quasi-monomial valuation $v^{(0)}$ over $Y$ and $\xi\in M_{\bR}$ such that for any $f\cdot \chi^u\in R_u$, we have:
\[
v(f\cdot \chi^u)=v^{(0)}(f)+\langle u, \xi\rangle.
\]
We will use $(v^{(0)}, \xi)$ to denote such a valuation. 

\item
$T$-invariant prime divisors on $X$ are either vertical or horizontal. Any vertical divisor is determined by a divisor $Z$ on $Y$ and a vertex $v$ of $\fD_Z$, and will be denoted by $D_{(Z,v)}$. Any horizontal divisor is determined by a ray $\rho$ of $\sigma$ and will be denoted by $E_\rho$.

\item Let $D$ be a $T$-invariant vertical effective $\mathbb{Q}$-divisor.  If $K_X+D$ is $\bQ$-Cartier, then the log canonical divisor has a representation $K_X+D=\pi^*H+{\rm div}(\chi^{-u_0})$ where $H=\sum_Z a_Z\cdot Z$ is a principal $\bQ$-divisor on $Y$ and $u_0\in M_{\bQ}$. Moreover, the log discrepancy of the horizontal divisor $E_\rho$ is given by:
\begin{equation}
A_{(X,D)}(E_\rho)=\langle u_0, n_\rho\rangle,
\end{equation}
where $n_\rho$ is the primitive vector along the ray $\rho$.

\end{enumerate}
\end{thm}
\begin{proof} In the first statement, the case for valuations with rational rank 1 follows from \cite[11]{AIPSV11}. It can be extended to quasi-monomial valuations trivially since any such valuation is a limit of valuation of rational rank 1. The second statement is in \cite[Proposition 3.13]{PS08} and the third statement  is in  \cite[Section 4]{LS13}.
\end{proof}

As mentioned, we have the identity $\sigma=\ft^+_{\bR}$. 
For any $\xi \in \mathfrak{t}^+_{\bR}$, we can define a valuation 
$$\wt_{\xi}(f) = \min_{\alpha\in \Gamma}\{\langle \alpha,\xi \rangle \ | \ f_{\alpha}\neq  0\}.$$
It is easy to verify that $\wt_\xi\in \Val_{X,x}$. We also define the rank of $\xi$, denoted by ${\rm rk}(\xi)$, to be the dimension of the subtorus $T_\xi$ (as a subgroup of $T$) generated by $\xi\in \ft$.

\begin{lem}\label{lem-Tqmv}
For any $\xi\in \mathfrak{t}^+_{\bR}$, $\wt_\xi$ is a quasi-monomial valuation of rational rank equal to the rank of $\xi$. Moreover, the center of $\wt_\xi$ is $x$.
\end{lem}

\begin{proof}
This follows from the work on $T$-varieties. By \cite{AH06}, $X={\rm Spec}_Y \left(R(Y, \fD)\right)$ for a polyhedral divisor $\fD$ over a semi projective variety $Y$. The rational 
rank of $\wt_\xi$ is equal to the rank of $\xi$. Let $Y_{\xi}=X\sslash T_{\xi}$, then 
$${\rm tr.deg.}(\wt_\xi)=\dim Y_{\xi}=\dim X-\dim T_{\xi}=\dim X-{\rm rat.rk.} \wt_\xi.$$ 

More explicitly, we will realize $\wt_{\xi}$ as a monomial valuation on a log smooth model. By \cite[Theorem 3.4]{AH06}, if we let 
\[
\tilde{X}={\rm Spec}_Y \bigoplus_{u\in \sigma^{\vee}\cap M}H^0(Y, \fD(u)),
\]
then there exists a birational morphism $\mu: \tilde{X}\rightarrow X$ and a projection $\pi: \tilde{X}\rightarrow Y$ such tfiber generic fiber of $\pi$ is a normal affine toric variety of dimension $r$. 
This toric variety, denoted by $Z$, is associated to the polyhedral cone $\sigma\subset N_{\bR}$. 
Each valuation $\wt_{\xi}$ corresponds to a vector $\xi$ contained in the interior of $\sigma$.  Let $U$ be a Zariski open set of $Y$ such that the fiber of $\pi: \tilde{X}\rightarrow X$ over any point $p\in U$ is isomorphic to $Z$. Then $\wt_\xi$ is the natural extension of the corresponding toric valuation on $Z$. As a consequence, it is a quasi-monomial valuation on $U\times Z$ and hence on the original $X$.

Next we can also realize $\wt_\xi$ on a log smooth model. 
Let $\tilde{Z}\rightarrow Z$ be a fixed toric resolution of singularities.
Then we can follow the construction in \cite[Section 2]{LS13} to obtain a toroidal resolution $\mathscr{X}\rightarrow X$ that dominates $\tilde{X}$ and its restriction over $U$ is isomorphic to $\tilde{Z}\times U$. Let $q\in \tilde{Z}$ be a contracting point of the torus action generated by $\xi$ and choose a point $p\in U$. Then it is easy to see that $\wt_{\xi}$ is realized as a monomial valuation with non-negative weights at $(p, q)\in U\times \tilde{Z}$.  
\end{proof}
\begin{rem}
The quasi-monomial property also follows from Lemma \ref{lem-quasi}.(2).
\end{rem}

By the construction in the above proof, the log discrepancy of $\wt_{\xi}$ can indeed be calculated in a similar way as in the toric case, and the toric case is well-known (see e.g. \cite{Amb06}, \cite[Proposition 7.2]{BJ17}). Assume $X$ is a normal affine variety with $\bQ$-Gorenstein klt singularities and a good $T$-action. Let $D$ be a $T$-invariant vertical divisor. As in \cite[2.7]{MSY08}, we can solve for a nowhere-vanishing section $T$-equivariant section $s$ of $m(K_X+D)$ where $m$ is sufficiently divisible (also see Remark \ref{rem-Tnonvanish}). 
\begin{lem}\label{lem-ldwt}
Using the same notion as in the Theorem \ref{t-Tcano}, the log discrepancy of $\wt_{\xi}$ is given by:
$A_{(X,D)}(\wt_\xi)=\langle u_0, \xi \rangle$. Moreover, let $s$ be a $T$-equivariant nowhere-vanishing holomorphic section of $|-m(K_X+D)|$, and denote $\cL_\xi$  the Lie derivative with respect to the holomorphic vector field associated to $\xi$. Then $A_{(X,D)}(\xi)=\lambda$ if and only if
\[
\mathcal{L}_{\xi}(s)=m \lambda s \quad \text{ for } \quad \lambda>0.
\]
\end{lem}
\begin{proof}
Let $\mathscr{X}\rightarrow \tilde{X} \rightarrow X$ be the same morphisms as in the proof of Lemma \ref{lem-Tqmv} and let $\mathscr{D}$ and $\tilde{D}$ be the strictly transform of $D$ on $\mathscr{X}$ and $\tilde{X}$ respectively. Then the situation can be reduced to the toric case (see also \cite[Section 4]{LS13}):
\begin{eqnarray*}
A_{(X,D)}(\wt_\xi)&=&A_{(\mathscr{X},\mathscr{D})}(\wt_\xi)+\wt_\xi(K_{(\mathscr{X}, \mathscr{D})/(X,D)})\\
&=&A_{\tilde{Z}}(\wt_\xi)+\wt_{\xi}(K_{\tilde{Z}/Z})\\
&=&A_{Z}(\wt_\xi)=\langle u_0, \xi\rangle.
\end{eqnarray*}

\bigskip

Next we discuss the second statement, because the map $\xi\mapsto \cL_{\xi}(s)/s$ is linear, we just need to verify the statement for rational $\xi$. Then in the case $D=\emptyset$, this follows from what was already showed in \cite{Li15a} and \cite[Proof of Proposition 6.16]{Li15b}. The same argument applies in the logarithmic case.
\end{proof}
\begin{rem}\label{rem-Tnonvanish}
Using the structure theory of $T$-varieties and under the assumption that $K_X+D$ is $\bQ$-Gorenstein, one can write down a nowhere-vanishing holomorphic section $s$ explicitly by using \cite[Theorem 3.21]{PS08} and \cite[Proposition 4.4]{LS13}. So one can also directly verify the equality $\langle u_0, \xi\rangle=\frac{1}{m}(\cL_\xi s/s)$. 
\end{rem}

As a consequence of the above lemma, we can extend $A_{(X,D)}(\xi)$ to a linear function on $\ft_{\bR}$.
\begin{defn}\label{defn-linearA}
Using the same notation as in the Theorem \ref{t-Tcano}, for any $\eta\in \ft_{\bR}$, we define:
\begin{equation}\label{eq-linearA}
A_{(X,D)}(\eta)=\langle u_0, \eta\rangle.
\end{equation}
\end{defn}
By Lemma \ref{lem-ldwt}, $A_{(X,D)}(\eta)=\frac{1}{m} \cL_\eta s/s$ where $s$ is a $T$-equivariant nowhere-vanishing holomorphic section of $|-m(K_X+D)|$.

We will need the following important convexity property originally discovered in \cite{MSY08} for cones with isolated singularities (see also \cite{DS15} for the case of metric tangent cones).
\begin{prop}[see Proposition \ref{prop-Tconvex}]\label{prop-Tconvex1}
The volume function $\vol=\vol_{X,x}$ is strictly convex on $\ft_{\bR}^+$.
If $\xi_0$ is a minimizer of $\hvol|_{\ft_{\bR}^+}$,  then for any vector $\xi\in \ft^+_\bR$, we have the inequality 
$$\hvol_{(X,D)}(\wt_{\xi_0})\le \hvol_{(X,D)}(\wt_{\xi}),$$
with the equality holds if and only if $\xi$ is a rescaling of $\xi_0$.
\end{prop}
Since we allow any klt singularity with good torus action, this is a gentle generalization of Martelli-Sparks-Yau's result. 
We will give an algebraic proof of the above result in Section \ref{sec-Tunique}. In particular, we will interpret this as a phenomenon in convex geometry.

For klt $T$-singularities, we have the following improvement of Theorem \ref{t-liu} in the equivariant case.
\begin{thm}[See \cite{LX16}]\label{thm-minKol}
Let $(X,D)$ be a $T$-equvariant klt singularity. 
Denote by $\Val_{X,x}^{T}$ the set of $T$-invariant valuations centered at $x$,  ${\rm PrId}^T_{X,x}$ the set of homogeneous $\km$-primary ideals,  and ${\rm Kol}^T_{X,x}$ the set of $T$-invariant Koll\'{a}r component.  Then we have the identity:
\begin{equation}
\hvol(X,D,x)=\inf_{S\in {\rm Kol}^T_{X,x}} A_{(X,D)}(S)^n\cdot \vol(\ord_S)=\inf_{v\in \Val^T_{X,x}} \hvol(v)
=\inf_{\ka \in {\rm PrId}^T_{X,x}} \lct^n(\ka)\cdot \mult(\ka).
\end{equation}
\end{thm}
\begin{proof} We first show $$\hvol(X,D,x)=\inf_{\ka \in {\rm PrId}^T_{X,x}} \lct^n(\ka)\cdot \mult(\ka).$$
In fact, given any $\km$-primary ideal $\ka$, we consider the initial ideal sequence
$$\{\kb_{\bullet}\}=\bin(\ka^k),$$
which we know is a graded sequence of $T$-equivariant ideals.
Then we know that 
$$\lim_{m\to \infty} \lct^n(\kb_m)\cdot \mult(\kb_m)= \lct^n(\kb_{\bullet})\cdot\mult(\kb_{\bullet})\le \lct^n(\ka)\cdot  \mult(\ka),$$
which confirms our claim. 

\medskip

Then to finish the proof, it suffices to show $$ \inf_{\ka \in {\rm PrId}^T_{X,x}} \lct^n(\ka)\cdot \mult(\ka)=\inf_{S\in {\rm Kol}^T_{X,x}} A_{(X,D)}(S)^n\cdot \vol(\ord_S).$$
We follow the strategy in the proof of \cite{LX16}. Given a $T$-equivariant primary ideal $\ka$, we can take the an $T$-equivariant dlt modification $Y\to X$ by running a $T$-equivariant model on a $T$-equivariant resolution. Then  any exceptional divisor $S$ on $Y/X$ is equivariant, and we know that 
$$  A_{(X,D)}(S)^n\cdot \vol(\ord_S)\le \vol(Y/X)\le  \lct^n(\ka)\cdot \mult(\ka),$$
where the equalities follow from \cite{LX16}. 
\end{proof}
\subsection{K-semistability of log Fano cone singularity}

For a $T$-equivariant singularity, the valuations induced by vector fields in the Reeb cone plays a special role, so we give the following 
\begin{defn}[See also \cite{CS15}]\label{d-FS}
Let $(X, D)$ be an affine klt pair with a good $T$ action (see Definition \ref{d-good}). For any $\xi\in \ft^+_\bR$, we say that the associated valuation ${\rm wt}_{\xi}$ gives a toric valuation. 
For a fixed $\xi$, we call the triple $(X,D,\xi)$ a klt singularity with a {\it log Fano cone} structure that is polarized by $\xi$.
\end{defn}

We proceed to study the K-semistable  log Fano cone singularity $(X,D, \xi)$ in the sense of Collins-Sz\'{e}kelyhidi (\cite{CS12,CS15}), which generalizes the K-semistability for Fano varieties (see \cite{Tia97, Don01}). We first define the special test configurations of log Fano cone singularities.
\begin{defn}[see \cite{CS12, CS15}]\label{d-stc}
Let $(X, D, \xi_0)$ be a log Fano cone singularity and $T$ be the torus generated by $\xi_0$. 

A $T$-equivariant special test configuration (or $T$-equivariant special degeneration) of $(X, D, \xi_0)$ is a quadruple $(\cX, \cD, \xi_0; \eta)$ with a map $\pi: (\cX, \cD)\rightarrow \bA^1(= \bC)$ satisfying the following conditions:
\begin{enumerate}
\item $\pi$ is a flat family of log pairs such that the fibres away from $0$ are isomorphic to $(X, D)$ and $\cX={\rm Spec}(\cR)$ is affine, where $\cR$ is a finitely generate flat $\bC[t]$ algebra. The torus $T$ acts on $\cX$, and we write $\mathcal{R}=\bigoplus_{\alpha}\mathcal{R}_{\alpha}$ as the decomposition into weight spaces;

\item $\eta$ is an algebraic holomorphic vector field on $\cX$ generating a $\bC^*$-action on $(\cX,\cD)$ such that $\pi$ is $\bC^*$-equivariant where $\bC^*$ acts on the base $\bC$ by multiplication (so that $\pi_*\eta=t\partial_t$ if $t$ is the affine coordinate on $\bA^1$)
and there is a $\bC^*$-equivariant isomorphism $\phi: (\cX, \cD) \times_{\bC}\bC^*\cong (X,D) \times \bC^*$;
\item the algebraic holomorphic vector field $\xi_0$ on $\cX\times_{\bC}\bC^*$ (via the isomorphism $\phi$) extends to a holomorphic vector field on $\cX$ (still denoted by $\xi_0$) and generates a $T$-action on $(\cX, \cD)$ that commutes with the $\bC^*$-action generated by $\eta$ and preserves $(X_0, D_0)$;
\item $(X_0, D_0)$ has klt singularities and $(X_0, D_0, \xi_0|_{X_0})$ is a log Fano cone singularity (see Definition \ref{d-FS}).
\end{enumerate}
$(\cX, \cD, \xi_0; \eta)$ is a product test configuration if there is a $T$-equivariant isomorphism $(\cX, \cD)\cong (X, D)\times \bC$ and $\eta=\eta_0+t\partial_t$ with $\eta_0\in \ft$.

 By abuse of notation, we will denote $\xi_0|_{X_0}$ by $\xi_0$.  
For simplicity, we will just say that $(\cX, \cD)$ is a ($T$-equivariant) special test configuration if $\xi_0$ and $\eta$ are clear. We also say that $(X,D,\xi_0)$ specially degenerates to $(X_0,D_0,\xi_0; \eta)$ (or simply to $(X_0, D_0)$).

If $(\cX, \cD, \xi_0;\eta)$ is a special test configuration, then under the base change $\bA^1\rightarrow\bA^1, t\mapsto t^d $, we can pull back it to get a new special test configuration $(\cX\times_{\bA^1,t^d}\bA^1, \cD\times_{\bA^1, t^d}\bA^1, \xi_0; d \cdot (t^d)^*(\eta))$, which will be simply denoted  by $(\cX, \cD, \xi_0;\eta)\times_{\bA^1,t^d}\bA^1$.
\end{defn}

Let $(\cX, \cD, \xi_0; \eta)$ be a $T$-equivariant special test configuration of $(X, D, \xi_0)$. We can define the Futaki invariant $\Fut(X_0, D_0, \xi_0; \eta)$ following \cite{CS12, CS15} where the index character was used. However, for our purpose, we reformulate the definition as the derivative of the normalized volume and we only consider special test configurations.

Since $T$-action and $\bC^*$-action commute with each other, $X_0$ has a $T'=(T\times\bC^*)$-action generated by $\{\xi_0, \eta\}$. Let $\ft'={\rm Lie} (T')$. 
For any $\xi\in \ft'^{+}_{\bR}$, we have $\wt_{\xi}\in \Val_{X_0,o'}$ where $o'\in X_0$ is the vertex point of the central fiber $X_0$. So we can define its volume $\vol(\wt_\xi)$ and normalized volume $\hvol(\wt_\xi)$. For simplicity of notations, we will frequently just write $\xi$ in place of $\wt_\xi$. Recall that the volume $\vol(\xi)$ is equal to:
\begin{equation}
\vol(\xi):=\vol_X(\wt_\xi)=\lim_{m\rightarrow+\infty} \frac{\dim_{\bC} R/\fa_m(\wt_\xi)}{m^n/n!};
\end{equation}
and the normalized volume is given by:
\[
\hvol(\xi):=\hvol_{(X_0,D_0)}(\wt_\xi)=A(\xi)^n\cdot \vol(\xi).
\]
Here $A(\xi)=A_{(X_0,D_0)}(\wt_\xi)$.

\begin{rem}
By \cite{MSY08, CS15}, for any $\xi\in \ft'^{+}_{\bR}$, the volume of $\wt_\xi$ can be obtained by using the index characters. Let
$X_0={\rm Spec}_{\bC} (B)$ and 
$
B=\bigoplus_{\alpha'} B_{\alpha'}
$
be the weight decomposition with respect to $T'$. For any $\xi\in \ft'^+_{\bR}$, the index character is defined as:
\begin{equation}
\Phi(t, \xi)=\sum_{\alpha} e^{-t\alpha'(\xi)}\dim B_{\alpha'}.
\end{equation}
Then by \cite{MSY08, CS15}, 
$\Phi(t, \xi)$ has the expansion (recall that $\dim X=n$):
\begin{equation}
\Phi(t, \xi)=\frac{\vol(\xi)}{t^{n}}+O(t^{1-n}).
\end{equation}
\end{rem}
\begin{defn}[see \cite{CS12, CS15}]\label{def-Futaki}
Let $(X_0, D_0, \xi_0)$ be a log Fano cone singularity with an good action by $T'\cong (\bC^*)^{r+1}$. Denote $\vol=\vol_{(X_0,D_0)}$ on $\ft'^+_\bR$ and $A=A_{(X_0,D_0)}$ on $\ft'_\bR$ (see Definition \ref{defn-linearA}).
Assume $\xi_0\in \ft'^{+}_{\bR}$. 
For any $\eta\in \ft'_{\bR}$, we define: 
\begin{eqnarray*}
{\rm Fut}(X_0, D_0, \xi_0; \eta)&:=&(D_{-\eta}\hvol)(\xi_0)\\
&=&n A(\xi_0)^{n-1} A(-\eta) \vol(\xi_0)+A(\xi_0)^n\cdot (D_{-\eta} \vol)(\xi_0).
\end{eqnarray*}
If $(\cX, \cD, \xi_0; \eta)$ is a special test configuration of $(X, D, \xi_0)$, then the Futaki invariant of $(\cX, \cD, \xi_0; \eta)$, 
denoted by $\Fut(\cX, \cD, \xi_0; \eta)$ is defined to be $\Fut(X_0, D_0, \xi_0; \eta)$. 
\end{defn}

\begin{rem} When $\xi_0$ generates a one dimensional torus (i.e., $T\cong \mathbb{C}^*$), then taking quotient by $T$, we get a special test configuration $(\mathcal{Y},\cE)$ of the log Fano pair $(Y,E)=(X,D)\setminus\{x\}/\langle \xi_0\rangle $, and we have $\Fut(\cX, \cD, \xi_0; \eta)$ is a rescaling of the Futaki invariant of $(\cY, \cE)$ defined in \cite{Tia97, Don01} (see e.g. \cite[Lemma 6.20]{Li15b}). This also verifies that the definition coincides with the one in \cite{CS15} (up to a constant) as any vector can be approximated by rational ones, and the Futaki invariants in both definitions are continuous and coincide when $\xi_0$ is rational. 
\end{rem}

We will need another form of the Futaki invariant later. For any $\xi\in \ft'^{+}_{\bR}$, if we denote $\hat{\xi}=\frac{\xi}{A(\xi)}$  
such that $\hat{\xi}$ lies on the truncated affine hyperplane 
\begin{equation}
P=\{\xi\in \ft'^+_\bR; A(\xi)=1\},
\end{equation} 
then we can transform the normalized volume to the usual volume: 
\begin{equation}\label{eq-hatxi}
\hvol(\xi)=A(\xi)^n\vol(\xi)=\vol\left(\frac{\xi}{A(\xi)}\right)=\vol(\hat{\xi}).
\end{equation} 
Moreover we can calculate:
\begin{eqnarray}\label{eq-Dvol}
\Fut(X_0, D_0, \xi_0; \eta)&=&D_{-\eta}\hvol(\xi_0)=\left.\frac{d}{ds}\right|_{s=0}\hvol(\xi_0-s\eta)\nonumber\\
&=&\left.\frac{d}{ds}\right|_{s=0}\vol(\widehat{\xi_0-s\eta})=\left.\frac{d}{ds}\right|_{s=0}\vol(\widehat{\xi_0}-s\cdot \what{T}_{\xi_0}(\eta))\nonumber\\
&=&(D \vol(\widehat{\xi_0}))\cdot (-\what{T}_{\xi_0}(\eta)),
\end{eqnarray}
where we have denoted:
\begin{equation}\label{eq-Txieta}
\what{T}_{\xi_0}(\eta):=\frac{A(\xi_0)\eta-A(\eta)\xi_0}{A(\xi_0)^2}\in \ft'_\bR.
\end{equation}
Notice that $\what{T}_{\xi_0}(\eta)$ in the tangent space of $P$ at $\what{\xi}_0$. In other words $A(\what{T}_{\xi_0}(\eta))=0$ (see \eqref{eq-linearA}).
 
The calculation \eqref{eq-Dvol} amounts to showing that re-normalization of the test configuration does not change the Futaki invariant:
\[
\Fut(\cX,\cD,\xi_0;\eta)=\Fut(X_0,D_0,\what{\xi_0};\what{T}_{\xi_0}(\eta)).
\]


\begin{defn}\label{d-ksemiSE}
We say that $(X, D, \xi_0)$ is K-semistable, if for any $T$-equivariant special test configuration $\cX$ that degenerates $(X,D,\xi_0)$ to $(X_0, D_0,\xi_0; \eta)$, we have 
$$\Fut(X_0, D_0, \xi_0; \eta)\ge 0.$$ 
\end{defn}

Applying the above discussion, we can then put the study of K-semistablity of a local singularity in the framework of the minimization of normalized volumes.

Let $\ft'$ be the Lie algebra of $T'=T\times\bC^*$, $N'={\rm Hom}(\bC^*, T')$ and $t'_{\bR}=N'\otimes_{\bZ}\bR$. Denote by $\ft'^{+}_{\bR}$ the positive cone of $\ft'_{\bR}$ on the central fiber. 
Denote the ray in $\ft'_{\bR}$ emanating from $\xi_0$ in the direction of $\eta$ by:
\[
\xi_0+\bR_{\ge 0}(-\eta)=\left\{\xi_0-\lambda \eta; \lambda\in \bR_{\ge 0}\right\}.
\]

\begin{lem}\label{lem-Tvolmin}
If $(X,D,\xi_0)$ is K-semistable, then for any special test configuration $(\cX, \cD, \xi_0; \eta)$, 
$$\hvol_{(X_0,D_0)}(\xi)\ge \hvol_{(X_0,D_0)}(\xi_0)$$ for any $\xi\in (\xi_0+\bR_{\ge 0}(-\eta))\cap \ft'^{+}_{\bR}$, where $\ft'_{\bR}=\ft_{\bR}\oplus \bR(\eta)$.
\end{lem}
\begin{proof}
For any $\xi_\lambda=\xi_0-\lambda\eta\in \ft'^+_\bR$, we have:
$\hvol(\xi_\lambda)=\vol(\widehat{\xi_\lambda})$ (see \eqref{eq-hatxi}).
Notice that:
\begin{eqnarray*}
\widehat{\xi_\lambda}-\widehat{\xi_0}&=&\frac{\xi_0-\lambda\eta}{A(\xi_0)-\lambda A(\eta)}-\frac{\xi_0}{A(\xi_0)}=-\frac{A(\xi_0)\eta-A(\eta)\xi_0}{A(\xi_0) A(\xi_\lambda)}\\
&=&-\what{T}_{\xi_0}(\eta) \frac{A(\xi_0)}{A(\xi_\lambda)}.
\end{eqnarray*}
So $\what{\xi_\lambda}\in \what{\xi_0}+\bR_{\ge 0}(-\what{T}_{\xi_0}(\eta))\cap \ft'^+_\bR$.
Consider the function $f(s)=\vol(\what{\xi_0}+s(\what{\xi_\lambda}-\what{\xi_0}))$.  Then $f(0)=\hvol(\xi)$ and
\[
f'(0)=(D \vol)(\what{\xi_0})\cdot (-\what{T}_{\xi_0}(\eta)) \frac{A(\xi_0)}{A(\xi_\lambda)}=\Fut(X_0, D_0, \xi_0; \eta)\frac{A(\xi_0)}{A(\xi_\lambda)}.
\]
By the K-semistability assumption we have $f'(0)\ge 0$. By Proposition \ref{prop-Tconvex1} (=Proposition \ref{prop-Tconvex}), $f(s)$ is a convex function. So we get $f(1)=\hvol(\xi_\lambda)\ge f(0)=\hvol(\xi_0)$.
\end{proof}

\section{Normalized volumes over log Fano cone singularities}\label{s-Tvarieties}

In this section, we will study {\it log Fano cone singularities $(X,D,\xi)$} (see Definition \ref{d-FS}). In differential geometry, the stability theory in such settings has been investigated in the context of searching for a Sasakian-Einstein metric (see \cite{MSY08, CS15} etc.). In particular, we will focus on the case that when the singularity is K-semistable, and show that in this case the natural toric valuation ${\rm wt}_{\xi}$ is the only minimizer up to rescaling among all $T$-invariant quasi-monomial valuations. For invoking different tools,  we divide the argument in Section \ref{sec-Tunique} into three steps with increasing generality: we first consider toric singularities with toric valuations, then general $T$-singularities with toric valuations and eventually $T$-singularities with $T$-invariant valuations.

\subsection{Special test configurations from Koll\'{a}r components}\label{ss-kollar}

In this section, we study the special test configuration of $T$-varieties associated to Koll\'{a}r components. 

Let $S$ be a Koll\'{a}r component over $o\in (X,D)$ and $\pi: Y\rightarrow X$ be the plt blow up extracting $S$ and let 
$$K_Y+\pi^{-1}_{*}D+S|_S=:K_S+\Delta_S.$$ In \cite{LX16} we used the deformation to the normal cone construction to get a degeneration of $X$ to an orbifold cone over $S$ with codimension one orbifold locus $\Delta_S$. We will simply  call it an orbifold cone over $(S,\Delta_S)$. Here we recall the corresponding algebraic description.

Denote the associated graded ring of $v_0=\ord_S$ by
\[
A=\bigoplus_{k=0}^{+\infty} \fa_k(v_0)/\fa_{k+1}(v_0)=\bigoplus_{k=0}^{+\infty} A_k.
\]

From now on, we always assume the Koll\'{a}r component $S$ is $T$-invariant so that $T$ acts equivariantly on $Y\rightarrow X$, and we have a decomposition:
\[
\fa_k(v_0)=\bigoplus_{\alpha} \fa_k^\alpha(v_0)=\bigoplus_{\alpha}R_\alpha \cap \fa_k(v_0).
\]
$T$ acts equivariantly on the extended Rees algebra:
\[
\cR'=\bigoplus_{k\in \bZ} \cR'_k:=\bigoplus_{k\in \bZ} \fa_k(v_0) t^{-k}\subset R[t, t^{-1}].
\]
Let $\cX=\Spec (\cR')$. Then we get a flat family $\pi: \cX\rightarrow \bA^1$ satisfying $X_t=\cX\times_{\bA^1} \{t\}=X$ and $X_0=\cX\times_{\bA^1}\{0\}=\Spec(A)$. Let $\cD$ be the strict transform 
of $D\times\bA^1$ under the birational morphism $\cX\dasharrow X\times\bA^1$. 
\begin{defn}\label{d-xiS}
Assume that $o\in(X,D)$ is a klt singularity with a good $T$-action and $S$ is a $T$-equivariant Koll\'ar component. Let $\mathcal{X}\to \mathbb{A}^1$ be the associated degeneration which degenerates $(X,D)$ to a $(X_0,D_0)$ and admits a $T'=T\times \mathbb{C}^*$-action.
For any $f=\sum f_k\in \cR'$, $\ord_S(f)=\min\{k; f_k\neq 0\}$. Over $X_0$, $\ord_S$ corresponds to the $\bC^*$-action corresponding to the $\bZ$-grading. Denote the generating vector by 
 $\xi_S\in \ft'^+_\bR$. 
 
With the above notations, we say that $(\cX, \cD, \xi_0; \xi_S)$ is the special test configuration associated to the Koll\'{a}r component
$S$. If $\xi_0$ and $\xi_S$ are clear, we just use $(\cX, \cD)$ to denote the special test configuration.
\end{defn}

\begin{lem}\label{l-degpre}
Let $(\cX, \cD, \xi_0; \xi_S)$ denote the special test configuration associated to a $T$-invarint Koll\'{a}r component $S$. Let $(X_0,D_0)$ be the corresponding pair on the special fiber.  For any $\xi_0\in \ft_{\bR}^+$, let $\xi_0$ also denote the induced Reeb vector on $X_0$ (see Definition \ref{defn-Reeb}). Then we have the following equalities:
\begin{enumerate}
\item
$A_{(X,D)}(\ord_S)=A_{(X_0, D_0)}(\wt_{\xi_S})$ and $\vol_{(X,D)}(\ord_S)=\vol_{(X_0,D_0)}(\wt_{\xi_S})$;
\item
$A_{(X,D)}(\wt_{\xi_0})=A_{(X_0,D_0)}(\wt_{\xi_0})$ and $\vol_{(X,D)}(\wt_{\xi_0})=\vol_{(X_0,D_0)}(\wt_{\xi_0})$ .
\end{enumerate}
\end{lem}
\begin{proof}
The first statement is clear (see \cite{LX16}). For the second statement, we first show the equality for log discrepancies. By Lemma \ref{lem-ldwt}, we just need to show the weights of holomorphic pluricanonical forms with respect to $\xi_0$, on $X$ and $X_0$ respectively, are equal.
Assume that $s$ is a non-vanishing $T$-equivariant  section of $\mathcal{O}(m (K_X+D))$. Let $\mu: Y\rightarrow X$ be the extraction of $S$. Then because of the identity 
$$K_Y+\mu_*^{-1}D=\mu^*(K_X+D)+(A-1) S$$ with $A=A_{(X,D)}(S)$, if we denote by $f$ a $T$-invariant local section of $\mathcal{O}_Y(mS)$, locally around $S$ we have the identity:
\[
\mu^*(s)=s' \cdot f^{A-1},
\]
where $s'$ is a local generator of $m (K_Y+\mu_*^{-1}D)$. By taking the Poincar\'e residue, 
$\mathcal{O}_S(m(K_S+\Delta_S))$ is generated by $\left.(s'/df)\right|_S=:dz$. 
For any $u\in T\cong (\bC^*)^r$, we have $u\circ s=u^\beta s$ and $u\circ f=u^{\gamma} f$, for some $\beta, \gamma\in \bZ^r$, by the equivariance of the the data. Then
\[
u^\beta \mu^*s=\mu^*( u\circ s)=u\circ \mu^*s=(u\circ s')\cdot (u\circ f)^{A-1}.
\]
So $u\circ s'=u^{\beta-(A-1)\gamma} s'$. As a consequence:
\[
u\circ dz=u\circ \frac{s'}{df}=u^{\beta-A \gamma} \frac{s'}{df}=u^{\beta-A\gamma} dz.
\]
A non-vanishing holomorphic section of $m(K_{X_0}+D_0)$ on  $X_0=C(S,\Delta)$ is given by $dz \otimes (df^{A})$ (see \cite[6.2.2]{Li15b} and Theorem \ref{t-Tcano}.3). Therefore, 
\begin{eqnarray*}
u\circ (dz\otimes df^A)&=&(u\circ dz)\otimes (u\circ df^A)= (u\circ \frac{s'}{df})\otimes(u\circ df^A)\\
&=& u^\beta (dz\otimes df^A).
\end{eqnarray*}

We also note that 
\begin{align*}
R/\fa_m(\wt_{\xi_0})= \bigoplus_{\alpha: \langle \alpha, \xi_0 \rangle< m} R_\alpha, \quad\quad 
A/\fa_m(\wt_{\xi_0})= \bigoplus_{\alpha: \langle \alpha, \xi_0 \rangle < m} \bigoplus_{k} A_k^\alpha.
\end{align*}
Now the equality of the volumes follows from the identity:
\[
\dim R_\alpha=\sum_k \dim\left( \frac{R_\alpha\cap \fa_k(\ord_S)}{R_\alpha\cap \fa_{k+1}(\ord_S)} \right)=\sum_k \dim A^\alpha_k.
\]
\end{proof}
For a later purpose, we need a little more:
\begin{prop}\label{prop-Teqsec}
Let $S$ be a $T$-invariant Koll\'{a}r component and $(\cX, \cD)$ be the special test configuration associated to $S$. Then there is a $T$-equivariant nowhere-vanishing holomorphic section $\mathscr{S}$ of $m(K_{\cX/\bA^1}+\cD)$ for $m$ sufficiently divisible. 
\end{prop}
\begin{proof}
We will use the same notations as used in the above proof. Let $\tilde{\mu}=\mu\times {\rm id}: Y\times\bA^1\rightarrow X\times \bA^1$ be the extraction of $S\times \bA^1$. To construct the special test configuration, we first consider the deformation to the normal cone by blowing up $S\times \{0\}$. Then $\tilde{\mu}^*(t^{-mA} s\wedge dt)$ is a meromorphic section of $\tilde{\mu}^* m(K_{X\times\bA^1}+ (D\times \bA^1))$ on $Y\times\bA^1$ satisfying:
\[
\tilde{\mu}^*(t^{-mA} s\wedge dt)=t^{-mA} s'\cdot f^{A-1}\wedge dt. 
\]
In local coordinates, $s'=df\wedge dz$. After blowing up $S\times\{0\}$, $t^{-m}f$ becomes a coordinate, denoted by $w$, along the fiber of the normal bundle of $S\subset Y$. So we have:
\begin{eqnarray*}
\tilde{\mu}^*(t^{-mA} s\wedge dt)&=&t^{-mA} d( t^m w)\wedge dz \cdot (t^m w)^{A-1}\\
&=& \left(w^{A-1} dw\wedge dz+m t^{-1} w^A dt\wedge dz\right)\wedge dt\\
&=&A^{-1} d (w^A)\wedge dz\wedge dt.
\end{eqnarray*}
Notice that $d(w^A)\wedge dz$ is a non-vanishing holomorphic section of $m(K_{X_0}+D_0)$ over $X_0=C(S,\Delta)$, thus $\tilde{\mu}^*(t^{-mA}s\wedge dt)\otimes\partial_t$ descends to a $T$-equivariant holomorphic section 
$\mathscr{S}$ of $m(K_{\cX/\bA^1}+\cD)$. 
\end{proof}
\begin{rem}\label{rem-Ltpt}
Notice that, we have
\begin{equation}\label{eq-Ltpt}
\mathcal{L}_{t\partial_t} \mathscr{S}=-m A \mathscr{S}.
\end{equation}
The minus sign is compatible with the fact that on the central fiber $t\partial_t=-\xi_S$. Moreover, if we restrict both sides of \eqref{eq-Ltpt}, then we see that $A_{(X_0,D_0)}(S)=A_{(X,D)}(S)$ as mentioned. 
\end{rem}

We can now generalize the minimization result in \cite{Li15b, LL16, LX16} to the higher rank case.

\begin{thm}\label{thm-semin}
$(X, D, \xi_0)$ is K-semistable if and only if $\wt_{\xi_0}$ is a minimizer of $\hvol_{(X,D)}$ in $\Val_{X,x}$. 
\end{thm}


\begin{proof}

First we assume $(X, D, \xi_0)$ is K-semistable. By Theorem \ref{thm-minKol}, we only need to show that for any $T$-invariant Koll\'{a}r component $S$, 
$$\hvol(\ord_S)\ge \hvol(\wt_{\xi_0}).$$

Let $(\cX, \cD)$ be the special test configuration associated to $S$. Since $S$ is $T$-invariant, there is a $T$-action on $\cX$. Let $T'=T\times\bC^*=(\bC^*)^{d+1}$. Then there is a $T'$-action on $\cX$ which is effective if $C(S)$ is not isomorphic to $X$. The canonical valuation $\ord_S$ corresponds to $\wt_{\xi_S}$ for some $\xi_S\in \ft'^{+}_\bR$ which is taken to be $-\eta$ in the definition of the special test configuration. As a consequence the ray $\xi_0+\bR_{\ge 0} (-\eta)$ is equal to
\[
\xi_0+\bR_{\ge 0}(-\eta)=\xi_0+\bR_{\ge 0} (\xi_S).
\] 
For any $b\in \bR_{\ge 0}$, denote $\xi_b=\xi_0+b \xi_S$, then we have
\begin{eqnarray*}
\hvol_{(X_0,D_0)}(\wt_{\xi_b})
&\ge& \hvol_{(X_0, D_0)}(\wt_{\xi_0})\ \   (\mbox{by Lemma \ref{lem-Tvolmin}})  \\
&=&\hvol_{(X,D)}(\wt_{\xi_0})  \ \   (\mbox{by Lemma  \ref{l-degpre}.2}).
\end{eqnarray*}
On the other hand, because $\xi_b/(1+b)\rightarrow \xi_S$ as $b\rightarrow+\infty$, by the rescaling invariance of $\hvol$ and the continuity of $\hvol$ on $\ft'^{+}_{\bR}$, we have
\begin{eqnarray*}
\lim_{b\rightarrow+\infty} \hvol_{(X_0,D_0)}(\wt_{\xi_b})&=&\lim_{b\rightarrow+\infty} \hvol_{(X_0, D_0)}(\wt_{\xi_b}/(1+b))\\
&=&\hvol_{(X_0, D_0)}(\wt_{\xi_S})\\
&=&\hvol_{(X, D)}(\wt_{\xi_S})\ \ (\mbox{by Lemma  \ref{l-degpre}.1}) .
\end{eqnarray*}
Combining the above inequalities, we get $\hvol_{(X,D)}(\ord_{\xi_S})\ge \hvol_{(X,D)}(\wt_{\xi_0})$.
\bigskip

For the converse direction, we assume $\hvol_{(X,D)}$ obtains its minimum at $\wt_{\xi_0}$ and  
let $(\cX, \cD, \xi_0; \eta)$ be any special test configuration. 

If we let $\xi_\epsilon=\xi_0 - \epsilon \eta\in \ft'^+_\bR$, then $\wt_{\xi_\epsilon}$ can be considered as a valuation on $\cX$. Using the embedding $\bC(X)\rightarrow \bC(\cX)=\bC(X\times \bC^*)=\bC(X\times\bC)$, 
$\wt_{\xi_\epsilon}$ can be restricted to become a valuation $w_\epsilon$ on $X$ (see \cite{BHJ17, Li15b} for the regular case). Alternatively by equivariantly embedding of $\cX$ into $\bC^N\times\bC$, $\wt_{\xi_\epsilon}$ is induced by a linear holomorphic vector field, still denoted by $\xi_\epsilon$, on $\bC^N$. The weight function associated to $\xi_\epsilon$ induces a filtration on $R$ whose associated graded ring is equal to 
the coordinate ring of $X_0$. By Lemma \ref{lem-quasi}, this filtration is indeed determined by a valuation $w_\epsilon$ on $X$. As a consequence we have $\vol_{(X,D)}(w_\epsilon)=\vol_{(X_0, D_0)}(\wt_{\xi_\epsilon})$ because $w_\epsilon$ and $\wt_{\xi_\epsilon}$ have the same associated graded ring. On the other hand, we claim that for each fixed $\epsilon$, 
\begin{equation}\label{eq-equalld}
A_{(X,D)}(w_\epsilon)=A_{(X_0,D_0)}(\wt_{\xi_\epsilon}).
\end{equation} 
This follows from the general construction in the proof of Lemma \ref{l-finiteass} (see Lemma \ref{lem-irrid}). If $\epsilon\ll 1$ we can choose a sequence of rational vector fields $\xi_{k,\epsilon}\in \ft^+_{\bQ}$ approaching $\xi_\epsilon$ as $k\rightarrow+\infty$. Then the $\bC^*$-action generated by $\xi_{k,\epsilon}$ corresponds to a Koll\'{a}r component $S_{k,\epsilon}$ which is isomorphic to the quotient $X_0/\langle \exp(\bC\cdot \xi_{k,\epsilon})\rangle$. In other words, up to a base change, $(\cX, \cD, \xi_0; \xi_{k,\epsilon})$ is equivalent to the special test configuration associated to $S_{k,\epsilon}$ and there exists constants $c_{k,\epsilon}>0$ such that $\wt_{\xi_{k,\epsilon}}|_{\bC(X)}=c_{k,\epsilon} \cdot \ord_{S_{k,\epsilon}}\rightarrow w_\epsilon$ as $k\rightarrow+\infty$. So by Lemma \ref{l-degpre}.1, we have $A_{(X,D)}(c_{k,\epsilon}\cdot \ord_{S_{k,\epsilon}})=A_{(X_0,D_0)}(\wt_{\xi_{S_{k,\epsilon}}})$ (see also \cite[Proposition 6.16.2]{Li15b}). Taking $k\rightarrow+\infty$, we get the identity \eqref{eq-equalld}.

Thus we get $\hvol_{(X,D)}(w_\epsilon)=\hvol_{(X_0,D_0)}(\wt_{\xi_\epsilon})$. As a consequence:
\begin{eqnarray*}
{\rm Fut}(X_0,D_0, \xi_0; \eta)&=&\left.\frac{d}{d \epsilon} \hvol_{(X_0,D_0)}(\wt_{\xi_\epsilon})\right|_{\epsilon=0}\\
&=&\left.\frac{d}{d \epsilon} \hvol_{(X,D)}(w_\epsilon)\right|_{\epsilon=0}\ge 0.
\end{eqnarray*}
The last inequality is because $w_{0}=\wt_{\xi_0}$ on $\bC(X)$ and hence by assumption $\hvol_{(X,D)}(w_\epsilon)$ obtains the minimum at $\epsilon=0$.


\end{proof}

By the construction in the above proof and Theorem \ref{thm-minKol}, we also get the following:
\begin{prop}\label{prop-testKollar}
To test K-semistability of a log Fano cone singularity $(X,D,\xi_0)$, we only need to test on the special test configurations associated to Koll\'{a}r components, i.e., we only need to check for any $T$-equivariant Koll\'ar component $S$, the generalized Futaki invariant ${\rm Fut}(X_0,D_0,\xi_0;\eta)\ge 0$, where $(X_0,D_0)$ is the induced special degeneration by $S$ and $\eta=-\xi_S$.
\end{prop}

The following purely algebro-geometric statement can be seen as a generalization of a result in \cite{Li13}, which was proved there by an analytic method.
\begin{prop}\label{l-degKsemi}
Let $(X,D,\xi_0)$ be a log Fano cone singularity, which specially degenerates to $(X_0,D_0,\xi_0; \eta)$ via a $T$-equivariant special test configuration.
Then $(X,D,\xi_0)$ is K-semistable if $(X_0,D_0,\xi_0)$ is K-semistable.  
\end{prop}
\begin{proof}
We use a degeneration argument which similar to the one used in \cite{LX16,Blu16}.

Let $\fa$ be an $m$-primary ideal on $(X,x)$. Let $T'\cong (\bC^*)^{r+1}$ denote the torus generated by $\xi_0$ and $\eta$, then by choosing a lexicographic order on $\bZ^{r+1}$, we can degenerate $\fa^m$ to $\fb_m:=\rin(\fa^m)$ on $(X_0, o')$ which is $T'$-equivariant. Denote $X_0={\rm Spec}(R^{(0)})$.  By the lower semicontinuity of log canonical thresholds and the flatness of the degeneration, we get:
\begin{eqnarray*}
\lct(\fa^m)^n\cdot l_R(R/\fa^m)\ge \lct(\rin(\fa^m))^n\cdot l_{R^{(0)}}(R^{(0)}/\rin(\fa^m)).
\end{eqnarray*}
Taking $m\rightarrow+\infty$, we get:
\begin{eqnarray*}
\lct(\fa)^n\cdot \mult(\fa)&\ge& \lct(\rin(\fb_\bullet))^n \mult(\fb_\bullet)\\
&\ge & \inf_{v}\hvol_{(X_0, D_0)}(v)=\hvol_{(X_0,D_0)}(\xi_0)\\
&=&\hvol_{(X,D)}(\xi_0).
\end{eqnarray*}
So we know that $(X,D,\xi_0)$ is K-semistable by Theorem \ref{thm-semin}.

\end{proof}

\subsection{Convexity and uniqueness}\label{sec-Tunique}

The results in this section do not depend on other sections.
Here we aim to show the volume function on the Reeb cone is strictly convex and hence conclude the uniqueness of the minimizer when the log Fano cone singularity is K-semistable. Our main approach is to describe the volume of a valuation as the volume of a convex body and then reduce the question to a known result in convex geometry. This is standard in our first step where we treat toric singularities. Then we generalize it to an arbitrary log Fano cone singularity with only toric valuations, by considering valuations with real rank larger than one. In the last step, we study a K-semistable log Fano cone singularity and compare its volume with those of $T$-invariant quasi-monomial valuations. 

\subsubsection{Toric valuations on toric varieties}\label{ss-toric}
In this section we will consider the case of toric singularities, where the minimization problem can be reduced to a convex geometric problem as known from \cite{MSY08}. For the latter application, we will consider a more general setting of strictly convex cones. This problem was considered by Gigena \cite{Gig78} in a different context.

Let $\sigma\subset N\otimes_{\bZ}{\mathbb{R}}$ be a strictly convex cone.
For $u_0\in {\rm relint}(\sigma^{\vee})$, denote $H_{0}=\{\xi\in N_{\bR}; \langle u_0, \xi \rangle =1\}$. Then $\hat{H}_0:=H_0\cap \sigma$ is a bounded convex set on the affine space $H_0$. 
For any $\xi\in {\rm relint}(\sigma)$, we denote the bounded polyhedron 
\[
\Delta_{\xi}=\{y\in \sigma^{\vee}; \langle y, \xi\rangle \le 1\}.
\]
To state the next key result, we denote by $\vol|_{\hat{H}_0}$ the restriction of the function $\xi\mapsto \vol(\Delta_\xi)$ to $\xi\in \hat{H}_0$ and let
\[
\Pi_\xi=\{u\in \sigma^{\vee}; \langle u, \xi\rangle=1\}.
\]
\begin{lem}[{\cite{Gig78}}]\label{lem-Gig}
\begin{enumerate}
\item
The function $\xi\mapsto \vol(\Delta_\xi)$ is a strictly convex function for $\xi\in {\rm relint}(\sigma)$, where ${\rm relint}(\sigma)=\sigma^{\circ}$ denotes the relative interior of the cone $\sigma$. 

\item
$\vol|_{\hat{H}_0}$ is a strictly convex proper function for $\xi\in \hat{H}_0=H_{0}\cap \sigma$. As a consequence, there is a unique minimizer $\xi_0$ of $\vol|_{\hat{H}_0}$. Moreover, 
$\xi_0\in \hat{H}_0=H_0\cap \sigma$ is a minimizer of $\vol(\Delta_\xi)|_{\hat{H}_0}$ if and only if $u_0$ is the barycenter of $\Pi_{\xi_0}$. 
\end{enumerate}
\end{lem}
For the reader's convenience, we provide a proof of this result by deriving the volume formula and reducing the minimization to a calculus problem as in \cite{Gig78}. 
\begin{proof}
We first derive a formula for $\vol(\xi)=\vol(\Delta_\xi)$ for any $\xi\in {\rm relint}(\sigma)$. To do this, fix a cross section $\tilde{u}: \tilde{\Pi} \rightarrow \sigma^{\vee}$. For example, we can choose $\tilde{\xi}\in {\rm relint}(\sigma)$ and let $\tilde{\Pi}=\sigma^{\vee}\cap \{u\in y; \langle u, \tilde{\xi}\rangle=1\}$. Consider the parametrization:
\begin{equation}
U: [0,1]\times \tilde{\Pi} \rightarrow \Delta_{\xi} , \quad 
(t, y)\mapsto t \frac{\tilde{u}(y)}{\langle \tilde{u}(y), \xi\rangle}.
\end{equation}
Denote $u(y)=\tilde{u}(y)/\langle \tilde{u}(y), \xi\rangle$. 
The Jacobian determinant of $F$ is equal to
\begin{eqnarray*}
\det({\rm Jac}(U))&=&\det\left(u(y), t \partial_{y_1}u(y), \dots, t \partial_{y_{n-1}}u(y)\right)\\
&=&t^{n-1}\det\left(u(y), \partial_{y_1}u(y), \dots, \partial_{y_{n-1}}u(y) \right)\\
&=&t^{n-1}\langle \tilde{u}(y), \xi\rangle^{-n}\det(\tilde{u}, \partial_{y_1}\tilde{u}, \dots, \partial_{y_{n-1}}\tilde{u}).
\end{eqnarray*}

So we get:
\begin{eqnarray}\label{eq-volDeltaXi}
\vol(\Delta_{\xi})&=&\int_{0}^1 \int_{\tilde{\Pi}}t^{n-1}\det({\rm Jac}(U)) dt dy=\frac{1}{n}\int_{\tilde{\Pi}} \det\left(u, \partial_{y_1}u, \dots, \partial_{y_{n-1}}u\right)dy\nonumber\\
&=&\frac{1}{n}\int_{\tilde{\Pi}}\frac{\det(\tilde{u}, \partial_{y_1}\tilde{u}, \dots, \partial_{y_{n-1}}\tilde{u})}{\langle \tilde{u}(y), \xi\rangle^{n}}dy.
\end{eqnarray}
For the simplicity, denote $\Phi(\xi)=\vol(\Delta_\xi)$. Then its first order derivative is equal to:
\begin{eqnarray}\label{eq-diffvol}
\frac{\partial}{\partial \xi}\Phi(\xi)&=&-\int_{\tilde{\Pi}}\frac{\tilde{u}(y)\det(\tilde{u}, \partial_{y_1}\tilde{u}, \dots, \partial_{y_{n-1}}\tilde{u})}{\langle \tilde{u}(y), \xi\rangle^{n+1}}dy\nonumber\\
&=&-\int_{\tilde{\Pi}} u(y) \det(u, \partial_{y_1}u, \dots, \partial_{y_{n-1}}u) dy \nonumber\\
&=&-\frac{1}{|\xi|} \int_{\Pi_\xi} u(y) s dy=-\frac{\vol(\Pi_\xi)}{|\xi|} {\rm bc}(\Pi_\xi),
\end{eqnarray}
where ${\rm bc}(\Pi_\xi)$ is the Euclidean barycenter of the bounded cross section $\Pi_{\xi}$.
In the last identity, we used the expression for the volume element $d\vol_{\Pi}$ of $\Pi_{\xi}:=F(1, \Pi)$ which is equal to $s\cdot dy$ with $s$ being equal to:
\begin{eqnarray*}
s&=&\left(\det\left(\partial_{y_i} u(y), \partial_{y_j}u(y)\right)\right)^{1/2}=\det(u, \partial_{y_1}u, \dots, \partial_{y_{n-1}}u)\frac{|\xi|}{\langle u, \xi\rangle}\\
&=&\det(u, \partial_{y_1}u, \dots, \partial_{y_{n-1}}u)|\xi|.
\end{eqnarray*}
Similarly, we get the expression for the second order derivative of $\vol(\Delta_\xi)$:
\begin{eqnarray*}
{\rm Hess}_{\xi}\left(\Phi(\xi)\right)&=&(n+1) \int_{\tilde{\Pi}}\frac{(\tilde{u}\otimes \tilde{u})\det(\tilde{u}, \partial_{y_1}\tilde{u}, \dots, \partial_{y_{n-1}}\tilde{u})}{\langle\tilde{u}(y), \xi\rangle^{n+2}}dy\\
&=&\frac{(n+1)}{|\xi|}\int_{\Pi_\xi}u\otimes u\;  d\vol_{\Pi_\xi}.
\end{eqnarray*}
We see that $\Phi(\xi)$ is strictly convex with respect to $\xi\in {\rm relint}(\sigma)$:
\begin{equation}
{\rm Hess}_{\xi}(\Phi(\xi))(\eta, \eta)> 0.
\end{equation}
This proves item 1 of Lemma \ref{lem-Gig}. 

Since $H_0$ is affine, $\hvol|_{\hat{H}_0}$ is also strict convex function for $\xi\in \hat{H}_0\cap H_0\cap \sigma$.
As $\xi\rightarrow \partial \hat{H}_0$, $\Delta_{\xi}$ becomes unbounded, so $\vol(\Delta_\xi)$ approaches $+\infty$. So we see that $\vol(\xi)|_{\hat{H}_0}$ is a strictly proper function. As a consequence,
there exists a unique minimizer of $\vol|_{\hat{H}_0}$.

If $\vol(\Delta_\xi)$ obtains the minimum at $\xi_0\in H_0$, then by Lagrangian multiplier method there exists $\lambda_0\in \bR$ such that $(\xi_0, \lambda_0)$ is a critical point of the function:
\[
\tilde{\Phi}(\xi, \lambda)=\Phi(\xi)-\lambda (\langle u_0, \xi\rangle-1).
\]
In other words, $(\xi_0, \lambda_0)$ satisfy:
\begin{align*}
\partial_{\xi}\tilde{\Phi}(\xi_0, \lambda_0)=\partial_\xi \Phi(\xi_0) - \lambda_0 u_0=0, \quad \partial_\lambda \tilde{\Phi}(\xi, \lambda)=-\langle u_0, \xi_0\rangle+1=0.
\end{align*}
Combining this with \eqref{eq-diffvol}, we see that
\[
u_0={\rm bc}(\Pi_{\xi_0}), \quad \lambda_0=-\frac{|\xi|}{\vol(\Pi_{\xi_0})}.
\]
This completes the proof of Lemma \ref{lem-Gig}.
\end{proof}
\begin{rem}
We notice that there is a similarity of the volume formula in \eqref{eq-volDeltaXi} with the formula for the volumes of $\bC^*$-invariant valuations derived in \cite{Li15b}. We will see in Proposition \ref{prop-volpoly} that this is not a coincidence. The properness of $\hvol(\wt_\xi)$ with respect to $\xi$ also follows from the properness estimate in \cite{Li15a}.
\end{rem}

Now we can easily deal with the case of toric singularities. Let  $\sigma \subset N\otimes_{\bZ}\bR$ be a strictly convex rational polyhedral cone and $X=X(\sigma)$ is the associated toric variety. The dual cone is $\sigma^{\vee}=\{u\in M_{\mathbb{R}}; \langle u, \xi\rangle \ge 0 \text{ for any } \xi\in \sigma\}$. There is a one-to-one correspondence between toric valuations in $\Val_{X,x}$ and the vectors in the relative interior ${\rm relint}(\sigma)$ of $\sigma$. Indeed, if we can write $X={\rm Spec} (R)$ where 
\[
R=\bigoplus_{u\in \sigma^{\vee}\cap M} \bC[\chi_u],
\]
then the valuation associated to $\xi\in {\rm relint}(\sigma)$ is given by:
\[
v_\xi(f)=\min\left\{\langle u, \xi\rangle, f=\sum_{u} f_u \text{ with } f_u\neq 0 \right\}.
\]

Then it is easy to verify that:
\[
\vol(v_\xi)={\rm vol}(\Delta_\xi)=:\vol(\xi).
\]
By our assumption $K_X+D=K_X+\sum a_iD_i$ is $\mathbb{Q}$-Cartier. Let $\{v_i\}_{i\in I}$ be primitive integral vectors along the edges of $\sigma$ which corresponds to the divisor $D_i$.
Then there exists $u_0\in {\rm relint}(\sigma^{\vee})\cap M_{\mathbb{Q}}$ such that $\langle u_0, v_i\rangle=1-a_i$ for any $i\in I$. The log discrepancy of any toric valuation has a simple expression:
\[
A(v_\xi)=\langle u_0, \xi \rangle.  
\] 
So the normalized volume of $v_\xi$ is given by:
\[
\hvol(v_\xi)=\langle u_0, \xi\rangle^n\cdot {\rm vol}(\Delta_\xi)={\rm vol}\left(\Delta_{\xi/\langle u_0, \xi\rangle}\right).
\]
Then the existence and uniqueness of the minimizer of $\hvol$ among toric valuations $v_{\xi}$ immediately follows from Lemma \ref{lem-Gig}.

\subsubsection{Toric valuations on $T$-varieties}\label{sss-Ttoric}
In this section, we treat a general Fano cone singularity with varying toric valuations. By using the Newton-Okounkov body technique, we will show that the volume function on the space of toric valuations associated to elements from the Reeb cone can be interpreted as a volume function of convex bodies as considered in the previous section. By applying Lemma \ref{lem-Gig}, we get the strict convexity (see Proposition \ref{prop-Tconvex}). As mentioned before, this is a generalization, from the case of isolated singularities to general klt singularities with good torus actions, of the important convexity property originally discovered in \cite{MSY08}.  Unlike their use of analytic tools, our proof is algebraic. Notice that Proposition \ref{prop-Tconvex} was used in the proof Theorem \ref{thm-semin}. 

\bigskip

Let $X$ be an $n$-dimensional normal affine variety with an effective algebraic action by $T=(\bC^*)^r$. 
Then by \cite{AH06}, there exists a normal semi-projective variety $Y$ of dimension $d:=n-r$ and a proper polyhedral divisor
$\fD$ such that $X={\rm Spec} (R)$, where 
\[
R:=\bigoplus_{u\in \sigma^\vee\cap M} H^0(Y, \cO(\fD(u))).
\]
Fix any $u_0\in {\rm relint}(\sigma^{\vee})$, we shall use the same notation as before. For example, $H_{0}=\{\xi\in N_{\bR}; \langle u_0, \xi\rangle=1\}$ and $\hat{H}_0=H_0\cap \sigma$. 
Denote by $\vol|_{\hat{H}_0}$ the restriction of $\vol$ to $\hat{H}_0$.
The goal of this section is to prove the following result.

\bigskip

\begin{prop}\label{prop-Tconvex}
The function $\vol\colon \xi\mapsto \vol(\wt_{\xi})$ is a strictly convex function of $\xi\in {\rm relint}(\sigma)$.
The function $\vol|_{\hat{H}_0}$ is a strictly convex and proper function of $\xi\in \hat{H}_0=H_0\cap \sigma$. As a consequence, there exists a unique minimizer of $\vol|_{\hat{H}_0}$.
\end{prop}

To prove Proposition \ref{prop-Tconvex}, we apply the ideas from the theory of Newton-Okounkov body to realize the volumes of $\wt_{\xi}$ as volumes of convex bodies, and then apply Lemma \ref{lem-Gig}.
\begin{proof}
We start by choosing a lexicographic order on $\bZ^r$ such that there is a $\bZ^r$ valued valuation:
\[
\bV_1(f)=\min\left\{u; f=\sum_{u\in \sigma^\vee\cap M} f_u \text{ with } f_u\neq 0\right\}.
\]
We extend this valuation to a $\bZ^n$-valued valuation in the following way. Fix a smooth point $p\in Y$ and algebraic coordinates $\{z_1, \dots, z_{n-r}\}$ at $p$. Choose $n-r$ $\mathbb{Q}$-linearly independent positive real numbers: ${\bf \alpha}=\{\alpha_1, \dots, \alpha_{n-r}\}$. Denote by $w_\alpha$ the quasi-monomial valuation on $\bC(Y)$ associated to these data. In other words, for any $f\in \bC(Y)$, we have:
\[
w_\alpha(f)=\min\left\{\sum_{i=1}^{n-r} \alpha_i m_i;\;\; z_1^{m_1}\cdots z_{n-r}^{m_{n-r}} \text{ appears in the Laurent expansion of } f \text{ at } p \right\}.
\]
Then the valuative group $G$ of $w_\alpha$ is a subgroup of $\bR$ and $G$ is isomorphic to $\bZ^{n-r}$. We now define the following lexicographic order on $\bZ^r\times G \cong \bZ^{r}\times\bZ^{n-r}=\bZ^n$:
\[
(u, \nu)\le (u', \nu') \text{ if and only if either } u< u', \text{ or } u=u' \text{ and } \nu \le \nu'.
\]
Any $f\in \bC(X)$ can be decomposed into nonzero weight components $f=\sum_{u\in \sigma^{\vee}\cap M} f_u\cdot \chi^u$ with $f_u\in \bC^*(Y)$. We let $u_f=\bV_1(f)\in \sigma^{\vee}$ and denote by $f_{u_f}$ the corresponding nonzero component. Then we can define a $\bZ^n$-valued valuation on $\bC(X)$ (see Remark \ref{rem-composite}):
\[
\bV(f)=\left(u_f, w_\alpha(f_{u_f})\right).
\]
The valuation group of $\bV$ is isomorphic to $\bZ^n=\bZ^r\times\bZ^{n-r}$ and the valuation semigroup $\cS$ of $\bV$ generates a convex cone $\hat{\sigma}$ in $\bZ^n$. Let $P: \bZ^n\rightarrow \bZ^r$ denote the projection from $\bZ^n$ to $\bZ^r$. Then $P(\hat{\sigma})=\sigma$. For any $\xi\in \sigma\subset N_{\bR}\cong \bR^r$, we can extend it by zeros to become $\hat{\xi}=(\xi, 0)\in \bR^n$. Then we have $\langle y, \hat{\xi}\rangle=\langle P(y), \xi\rangle$. For any $\xi\in {\rm relint}(\sigma)$, denote
\[
\Delta_\xi=\{y\in \hat{\sigma}; \langle P(y), \xi\rangle \le 1\}.
\]
We want to relate the volume of valuation to the volume of $\Delta_\xi$. Following \cite{LM09}, we need to show that $\bV$ satisfies the following properties (which is equivalent to $\bV$ being a {\it good} valuation in the sense of \cite[Definition 7.3]{KK14}). 
\begin{lem}\label{lem-sg3prop}
Let $\cS$ denote the semigroup of the valuation $\bV$. Then $\cS$ satisfies the following three properties:
\begin{enumerate}
\item[(P1)] $\cS\cap \{y\ | \ \langle P(y), \xi\rangle=0\}=\{0\}$;
\item[(P2)] Denote by $e_i$ the $i$-th standard vector of $\bZ^r$. There exist finitely many vectors $\left(e_i, v^{(i)}_k\right)$ spanning a semigroup $B\subset \bZ^n$ such that $\cS\subset B$;
\item[(P3)] $\cS$ generates $\bZ^n$ as a group.
\end{enumerate} 
\end{lem}
\begin{proof}
The last condition follows from the fact the valuative semigroup $\cS$ of $\bV$ generates the valuative group which is isomorphic to $\bZ^n$. 
To verify the first two conditions, we just need to show that there exists a constant $C>0$ such that for any $u$ and $f\in H^0(Y, \fD(u))$, we have:
\begin{equation}\label{eq-bdV2}
|w_\alpha (f) |\le C \langle u, \xi\rangle.
\end{equation} 
The following argument to prove this estimate is motivated by the argument in \cite{LM09}.

For $b$ sufficiently divisible, $\fD(m u)$ is Cartier. Denote by $L_{mu}$ the line bundle associated to $\fD(mu)$ and $L_u$ the $\mathbb{Q}$-line bundle associated to $\fD(u)$. 
Fix a global section $g_{mu}\in H^0(Y, \fD(m u))$ such that $g_{mu}^{-1}$ is a local equation for $\fD(mu)$ near $p\in Y$. Then for any $f\in H^0(Y, \fD(u))$ we have the identity:
\begin{equation}\label{eq-waf}
w_\alpha(f)=\frac{1}{m} w_\alpha(f^m g_{mu}^{-1})-\frac{1}{m}w_\alpha(g_{mu}^{-1}).
\end{equation}
For simplicity of notation, we write $g_u:=g_{mu}^{1/m}$ as a multi-section of the $\bQ$-line bundle $L_u$. Then \eqref{eq-waf} can be written  as:
\begin{equation}\label{eq-waf2}
w_\alpha(f)=w_\alpha(f g_u^{-1})-w_\alpha (g_u^{-1}).
\end{equation}
We will bound both terms on the right hand side of \eqref{eq-waf2}. By Izumi's theorem, $w_\alpha$ is comparable to $\ord_{p}$. In other words there exists a constant $C>0$ such that: 
\begin{equation}
C^{-1}\cdot  \ord_{p}\le w_\alpha\le C\cdot \ord_{p}.
\end{equation}
So to show the inequality \eqref{eq-bdV2}, we can replace $w_\alpha$ by $\ord_{p}$.

Let $k=\ord_{p}(f g_u^{-1})=\frac{1}{m}\ord_{p} (f^m g^{-1}_{mu})$. Then $$f\in H^0(Y, \fD(u)\otimes_{\cO_Y}\fm^k_{p}).$$ Let $\mu: \tilde{Y}\rightarrow Y$ 
be the blow up of $Y$ at $p$. Fix a very ample divisor $H$ on $Y$. Then $\tilde{H}:=\mu^*H-\epsilon E$ is ample for $\epsilon$ sufficiently small. 
So we get $(\pi^*L_u-k E)\cdot \tilde{H}^{d-1}\ge 0$ and hence the estimate:
\begin{equation}
k\le \frac{\pi^*L_u\cdot \tilde{H}^{d-1}}{E\cdot \tilde{H}^{d-1}}=\frac{L_u\cdot H^{d-1}}{\epsilon^{d-1}}.
\end{equation}
So we are left with showing that 
\begin{equation}\label{eq-ub1}
L_u\cdot H^{d-1}\le C \langle u, \xi\rangle.
\end{equation} 
Now by \cite[Proposition 2.11]{AH06}, the polyhedral divisor $\fD: u\rightarrow \fD(u)$ is a convex, piecewise linear, strictly semi-ample maps from $\sigma^{\vee}$ to ${\rm CaDiv_{\bQ}}(Y)$ . More precisely, we have  (see \cite[Definition 2.9]{AH06} for relevant definitions)
\begin{enumerate}
\item $\fD(u)+\fD(u')\le \fD(u+u')$ holds for any two elements $u, u'\in \sigma^{\vee}$,
\item $u\rightarrow \fD(u)$ is piecewisely linear, i.e. there is a quasi-fan $\Lambda$ in $M_{\bQ}$ having $\sigma^{\vee}$ as its support such that $\fD$ is linear on the cones of $\Lambda$,
\item
$\fD(u)$ is always semi-ample and, $\fD(u)$ is big for $u\in {\rm relint}(\sigma^{\vee})$. 
\end{enumerate}
In the second item above, $\Lambda$ being quasi-fan means that it is a finite collection of cones in $M_{\bQ}$ satisfying natural compatible properties. Any $u\in \sigma^{\vee}$ is contained in some cone $\lambda$ of $\Lambda$. Choose a finite set of generators $\{u_i\}_{i\in I}\subset M_\bQ$ of $\lambda$, then $u=\sum_{i\in I} a_i u_i$ with $a_i\ge 0\in \bQ$ and $\fD(u)=\sum_{i\in I} a_i \fD(u_i)$ as $\mathbb{Q}$-Cartier divisors.  

Because $\xi\in {\rm relint}(\sigma)$, $\langle u_i, \xi\rangle>0$ for any $i\in I$.
So it is easy to see that there exists $C_\lambda>0$ depending only on $\lambda$ such that:
\[
L_u\cdot H^{d-1}\le \sum_{i\in I} a_i L_{u_i}\cdot H^{d-1}\le \sum_{i\in I} a_i C_\lambda \langle u_i ,\xi\rangle=C_\lambda \langle u, \xi\rangle. 
\]

Because there are finitely many cones in $\Lambda$, we get estimate \eqref{eq-ub1}. 

By using piecewise linearity, we can use the same argument to show that there exists $C>0$ independent of $u\in \sigma^{\vee}$, satisfying:
\begin{equation}\label{eq-ub2}
|w_\alpha(g_u^{-1})|\le C \langle u, \xi\rangle. 
\end{equation}
Combining \eqref{eq-waf2}-\eqref{eq-ub2} and the above discussions, we get the wanted estimate \eqref{eq-bdV2}. 
\end{proof}

\begin{lem}\label{l-conevolume}
With the above notation, for any $\xi\in {\rm relint}(\sigma)$, we have:
\begin{equation}\label{eq-eqvol}
\vol(\wt_\xi)=\vol(\Delta_{\xi}). 
\end{equation}
\end{lem}
\begin{proof}
With the good properties of $\bV$ obtained via Lemma \ref{lem-sg3prop},
we can prove this result by using the general theory developed in \cite{Oko96, LM09, KK12, KK14}. We use the argument similar to the one used in \cite{KK12, Cut12, KK14}. Denote $\fa_\lambda=\{f\in R; \wt_\xi(f)\ge \lambda\}$. By the estimate \eqref{eq-bdV2}, we know that the cone $\hat{\sigma}\subset \mathbb{Z}^n$ does not contain a line. Moreover, we can choose a linear function 
$\ell: \bR^n\rightarrow \bR$ such that $\hat{\sigma}$ lies in $\ell_{\ge 0}$ and intersects the hyperplane $\ell^{-1}(0)$ only at the origin. In fact, we can choose $\ell$ to be the linear function associated to any $\xi\in {\rm relint}(\sigma)$. Then we have the following identity:
\begin{eqnarray}\label{eq-dimct}
{\rm codim_{\bC}}(R/\fa_\lambda)&=&\# \{(\theta_1, \dots, \theta_n)\in \bV(R); \ell(\theta) \le C \lambda, i=1, \dots, n\}\nonumber\\
&&\hskip 10mm -\#\{(\theta_1, \dots, \theta_n)\in \bV(\fa_\lambda); \ell(\theta) \le C \lambda, i=1, \dots, n \}.
\end{eqnarray}
Define the semigroups of $\bZ^{n+1}$:
\begin{align*}
&\tilde{\Gamma}=\left\{(\theta_1, \dots, \theta_n, \lambda); (\theta_1, \dots, \theta_n)\in \bV(\fa_\lambda) \text{ and } \ell(\theta) \le C\lambda \right\},\\
&\tilde{\Gamma}'=\left\{(\theta_1, \dots, \theta_n, \lambda); (\theta_1, \dots, \theta_n)\in \bV(R) \text{ and } \ell(\theta) \le C\lambda \right\}.
\end{align*}
Because $\cS=\bV(R)$ generates $\bZ^n$, we can choose $C\gg 1$ such that $\tilde{\Gamma}'_1$ generates $\bZ^n$. Then $\tilde{\Gamma}'$ also generates $\bZ^{n+1}$. By the general
theory of Newton-Okounkov bodies from \cite{Oko96, LM09, KK12}, we have:
\begin{equation}\label{eq-OkVol}
\lim_{\lambda\rightarrow+\infty} \frac{\#\; \tilde{\Gamma}_\lambda}{ \lambda^n }=\vol(\Delta(\tilde{\Gamma})),\quad
\lim_{i\rightarrow+\infty} \frac{\#\; \tilde{\Gamma}'_\lambda}{\lambda^n}=\vol(\Delta(\tilde{\Gamma}')),
\end{equation}
where
\[
\Delta(\tilde{\Gamma})=\overline{\bigcup_{\lambda>0}\left\{\frac{\theta}{\lambda}; (\theta, \lambda)\in \tilde{\Gamma} \right\}}, 
\quad
\Delta(\tilde{\Gamma}')=\overline{\bigcup_{\lambda>0}\left\{\frac{\theta}{\lambda}; (\theta, \lambda)\in \tilde{\Gamma}' \right\}}
\]
are the Newton-Okounkov bodies of $\tilde{\Gamma}$ and $\tilde{\Gamma}'$ respectively. By \eqref{eq-dimct} we have:
\begin{equation}\label{eq-codimcount}
{\rm codim}_{\bC}(R/\fa_\lambda)=\cS\setminus \bV(\fa_\lambda)=\tilde{\Gamma}'_\lambda\setminus \tilde{\Gamma}_\lambda.
\end{equation}
Combining this with \eqref{eq-OkVol}, we get:
\begin{eqnarray}\label{eq-OkVol2}
\lim_{i\rightarrow+\infty} \frac{\# (\cS\setminus \bV(\fa_\lambda))}{\lambda^n}=\vol(\Delta(\tilde{\Gamma}'))-\vol(\Delta(\tilde{\Gamma}))=\vol\left(\Delta(\tilde{\Gamma}')\setminus \Delta(\tilde{\Gamma})\right).
\end{eqnarray}
If $\xi=\sum_{j=1}^r a_j e_j$, we have $\ell(\theta)=\ell_{\xi}(\theta)=\sum_j a_j \theta_j$ and:
\begin{equation*}
\tilde{\Gamma}=\left\{(\theta_1, \dots, \theta_n, \lambda); (\theta_1, \dots, \theta_n)\in \cS, \sum_{j=1}^r a_j \theta_j\ge \lambda, \ell(\theta)\le C \lambda \right\}.
\end{equation*}
 By \cite{LM09, KK12, KK14}, we know that:
\[
\Delta(\tilde{\Gamma})=\sigma_{\ell_{\xi}\ge 1} \cap \sigma_{\ell\le C}, \quad \Delta(\tilde{\Gamma}')=\sigma_{\ell\le C}. 
\]
So we get:
\begin{eqnarray}\label{eq-eqbody}
\Delta(\tilde{\Gamma}')\setminus \Delta(\tilde{\Gamma})&=&\{y\in \hat{\sigma}; \sum_{j=1}^r a_i y_i \le 1\}
=\{y\in \hat{\sigma}; \langle y, \hat{\xi}\rangle \le 1 \}=\Delta_\xi.
\end{eqnarray}
By combining the above identities \eqref{eq-codimcount}-\eqref{eq-eqbody}, we get the identity \eqref{eq-eqvol}.

\end{proof}

We can complete the proof of Proposition \ref{prop-Tconvex} by using the similar argument as in the toric case. Indeed,  Lemma \ref{lem-Gig}.1 and Lemma \ref{l-conevolume} together imply that $\vol(\wt_\xi)$ is a proper strictly convex function of $\xi\in \hat{H}_0=H_0\cap \sigma$. So there exists a unique minimizer of $\vol(\wt_\xi)$ among $\xi\in \hat{H}_0=H_0\cap \sigma$. Similarly to the item 2 of Lemma \ref{lem-Gig}, $\xi_0$ is a minimizer of $\vol(\wt_\xi)$ if and only if $\xi_0$ is the barycenter of the measured convex domain 
$(\Pi_{\xi_0}, (P_{\Pi_{\xi_0}})_*(d\vol_{\Pi_{\xi_0}}))$ 
where $\Pi_{\xi_0}:=\{y\in \hat{\sigma}; \langle P(y),\xi_0\rangle=1\}$ and $d\vol_{\Pi_{\xi_0}}$ is the standard Euclidean volume form on $\Pi_{\xi_0}$.
\end{proof}

Now assume that $(X, D)$ is a klt singularity. In particular, $K_X+D$ is $\mathbb{Q}$-Gorenstein. Then by Theorem \ref{t-Tcano}.3, there is a $u_0\in M$ and a principal divisor ${\rm div}(f)=\sum_Z a_Z \cdot Z$ on $Y$ such that ${\rm div}(f\cdot \chi^{u_0})=K_X+D$ and the log discrepancy of the toric valuation $\wt_\xi$ is calculated as
\[
A_{(X,D)}(\wt_\xi)=\langle u_0, \xi\rangle.
\]
So the normalized volume is given by:
\[
\hvol(\wt_\xi)=\langle u_0, \xi\rangle^n\cdot {\rm vol}(\Delta_\xi)=\Vol\left(\Delta_{\xi/\langle u_0, \xi\rangle}\right).
\]
In particular,  we have $\hvol(\wt_\xi)=\vol(\xi)$ for $\xi\in \hat{H}_0=H_0\cap \sigma$. As a corollary of Proposition \ref{prop-Tconvex}, we immediately get Proposition \ref{prop-Tconvex1}.


\subsubsection{T-invariant quasi-monomial valuations on $T$-varieties}\label{sss-TT}

As in the previous section, we assume that there is an effective good $T=(\bC^*)^{r}$ action on an affine normal variety $X={\rm Spec}(R)$. In this section, we aim to show the following theorem
\begin{thm}\label{thm-uniquefano}
Let $(X,D,\xi)$ be a Fano cone singularity. If $\wt_\xi$ is the minimizer of $\hvol_{(X,D)}$, then it is unique among all $T$-invariant quasi-monomial valuations. 
\end{thm}
\begin{rem} 
{\rm (i)}
By Theorem \ref{thm-semin}, the assumption is indeed equivalent to $(X,D,\xi)$ being K-semistable.

{\rm (ii)}
Let $v$ be another $T$-invariant minimizer. If we could show the associated graded ring ${\rm gr}_v(R)$ is finitely generated, then similar to the argument in Theorem \ref{thm-semin}, we can degenerate $X$ to $X_0$ via $v$ and both $\xi$ and $\xi_v$ would be the minimizers of $\hvol_{X_0,D_0}$, which is contradictory to Proposition \ref{prop-Tconvex}. Using this method, we can give another proof of uniqueness of divisorial minimizers proved in \cite{LX16} with a different argument. However, for general quasi-monomial minimizers, since we do not known yet the finite generation of the associated graded ring, we have to adapt a different argument. We also note that later in Proposition \ref{p-Tequiv} we will show any quasi-monomial minimizer is automatically $T$-invariant.
\end{rem}

The idea of the proof is to first connect any $T$-invariant quasi-monomial valuation $v$ with $v_\xi$ by a family of $T$-invariant quasi-monomial valuations $v_t$. This depends on the description of $T$-invariant valuations in \cite{AIPSV11}.  Next we extend $v$ to a valuation $\bV$ of rational rank $n$ and prove that it satisfies properties as in Lemma \ref{lem-sg3prop}. Then we can use the works of Newton-Okounkov bodies to realize the volumes of valuations as volumes of convex bodies as has been done in the previous sub-section. Finally we use the previous convex geometric result to get the strict convexity of the volumes $\vol(v_t)$ with respect to $t$ which implies the uniqueness of the minimizer.

\begin{proof}
By \cite{AH06}, there exists a normal semi-projective $Y$ of dimension $d:=n-r$ and a proper polyhedral divisor $\fD: \sigma^{\vee}\rightarrow {\rm CarDiv}(Y)$ such that:
\begin{equation}\label{eq-repTvar}
X={\rm Spec}_{\bC} \left(\bigoplus_{u\in \sigma^{\vee}} H^0(Y, \fD(u))\right).
\end{equation}

By Theorem \ref{t-Tcano}.1, any (quasi-monomial) $T$-invariant valuation is of the form $v=(v^{(0)}, \zeta)$ defined via the identity:
\[
v(f\cdot \chi^u)=v^{(0)}(f)+\langle u, \zeta\rangle.
\]
If $v$ is quasi-monomial, then $v^{(0)}$ is a quasi-monomial valuation on $\bC(Y)$. Let $s$ be the rational rank of $v^{(0)}$. There exists a birational morphism $\psi: Y'\rightarrow Y$, a regular closed point $p\in Y'$, algebraic coordinates $\{z'; z''\}=\{z_1, \dots, z_{s}; z_{s+1}, \dots, z_{d}\}$ and $s$ rationally independent positive real numbers $\beta=(\beta_1, \dots, \beta_{s})$ such that 
$v^{(0)}$ is the quasi-monomial valuation associated to these data. More precisely, if the Laurent series of an $f\in \bC(Y')$ has the form
\begin{equation}\label{eq-flaurent}
f=\sum_{m\in \bZ^{s}} z_1^{m_1}\cdots z_{s}^{m_{s}} \chi_m(z''),
\end{equation}
we will say that $z_1^{m_1}\dots z_{s}^{m_{s}}$ appears in the Laurent expansion of $f$ if $\chi_m(z'')\neq 0$. Then we have:
\[
v^{(0)}(f)=\min\left\{\sum_{i=1}^{s} \beta_i m_i;  z_1^{m_1}\cdots z_{s}^{m_{s}} \text{ appear in the Laurent expansion of } f \right\}. 
\]
In the representation \eqref{eq-repTvar}, we can replace $Y$ by $Y'$ and $\fD(u)$ by $\psi^*\fD(u)$. So for the simplicity of notations, in the following discussion, we will still denote $Y'$ by $Y$.
Moreover, if we let $D_i=\{z_i=0\}$ for $i=1, \dots, d=n-r$, then by resolving the singularities of $(Y, \sum_i D_i)$, we can also assume $\sum_{i=1}^{n-r} D_i$ has simple normal crossings by possibly replacing $Y$ by a new birational model.

As before, we fix a lexicographic order on $\bZ^{r}$ and define for any $f\in R$, 
\[
\bV_1(f)= \min\{u; f=\sum_u f_u \text{ with } f_u \neq 0\}=\bV_1(f).
\]
Again we will first extend this $\bZ^{r}$-valuation $\bV_1$ to become a $\bZ^n$-valued valuation. Denote $u_f=\bV_1(f)\in \sigma^{\vee}$ and $f_{u_f}$ the corresponding nonzero component. 
Define $\bV_2(f)=v^{(0)}(f_{u_f})$. Because $\{\beta_i\}$ are $\mathbb{Q}$-linearly independent, we can write
$\bV_2(f)=\sum_{i=1}^{s} m^{*}_i \beta_i$ for a uniquely determined $m^*:=m^*(f_{u_f})=\{m^*_i:=m^*_i(f_{u_f})\}$. Moreover, the Laurent expansion of $f$ has the form:
\begin{equation}\label{eq-flaurent*}
f_{u_f}=z_1^{m^*_1}\dots z_{s}^{m^*_{s}} \chi_{m^*}(z'')+\sum_{m\neq m^*} z_1^{m_1}\dots z_{s}^{m_{s}} \chi_m(z'').
\end{equation}
 Then $\chi_{m^*}(z'')$ in the expansion of \eqref{eq-flaurent*} is contained in $\bC(Z)$, where $Z=\{z_1=0\}\cap \dots \{z_s=0\}=D_1\cap \dots\cap D_{s}$. 

Extend the set $\{\beta_1, \dots, \beta_{s}\}$ to $d=n-r$ $\bQ$-linearly independent positive real numbers $\{\beta_1, \dots, \beta_{s}; \gamma_1, \dots, \gamma_{d-s}\}$. 
Define
$\bV_3(f)=w_{\gamma}(\chi_{m^*}(z''))$ where $w_{\gamma}$ is the quasi-monomial valuation with respect to the coordinates $z''$ and the $(d-s)$ tuple $\{\beta_1, \dots, \beta_{s}; \gamma_1, \dots, \gamma_{d-s}\}$.

\medskip

Now we assign the lexicographic order on $\bG:=\bZ^{r}\times G_2\times G_3\cong \bZ^{r}\times \bZ^{s}\times \bZ^{n-r-s}$ and define $\bG$-valued valuation:
\begin{equation}
\bV(f)=(\bV_1(f), \bV_2(f_{u_f}), \bV_3(\chi_{m^*})).
\end{equation}
\begin{rem}\label{rem-composite}
The construction of $\bV$ is an example of composite of valuations (see \cite[VI.16]{ZS60}).
\end{rem}

Let $\cS$ be the valuative semigroup of $\bV$. Then $\cS$ generates a cone $\tilde{\sigma}$. Let $P_1: \bR^n\rightarrow \bR^{r}$, $P_2: \bR^n\rightarrow \bR^{s}$ and $P=(P_1, P_2): \bR^n\rightarrow \bR^{r+s}$ be the natural projections. Then $P_1(\tilde{\sigma})=\sigma\subset \bR^r$.

For any $\xi\in \rint(\sigma)$, denote by $\wt_\xi$ the valuation associated to $\xi$. We can connect $\wt_\xi$ and $v$ by a family of quasi-monomial valuations: $v_t= ((1-t)\xi+t \zeta, t v^{(0)})$ defined as
\[
v_t(f\cdot \chi^u)=t v^{(0)}(f)+\langle u, (1-t)\xi+t\zeta\rangle.
\]
So the vertical part of $v_t$ corresponds to the vector $\Xi_t:=((1-t)\xi+t\zeta, t\beta)\in \bR^{r+s}$. Extend $\Xi_t$ to 
$\tilde{\Xi}_t:=(\Xi_t, 0)\in \bR^n$ and define the following set:
\[
\Delta_{\tilde{\Xi}_t}=\left\{y\in \tilde{\sigma}; \langle y, \tilde{\Xi}_t\rangle \le 1 \right\}=\left\{y\in \tilde{\sigma}; \langle P(y), \Xi_t\rangle \le 1\right\}.
\]
Because $\hvol$ is rescaling invariant, we can assume $A_{(X,D)}(v)=A_{(X,D)}(\xi)=1$. Then by the $T$-invariance of $v_t$, we easily get:
$$A(v_t)=t A(v^{(0)})+A_{(X,D)}((1-t)\xi+t\zeta)=t A_{(X,D)}(v)+(1-t)A_{(X,D)}(\xi)\equiv 1.$$
So by Proposition \ref{prop-volpoly} we have:
\[
\hvol(v_t)=\vol(v_t)=\vol(\Delta_{\tilde{\Xi}_t})
\]
Because $\tilde{\Xi}_t$ is linear with respect to $t$, by Lemma \ref{lem-Gig} $\phi(t):=\vol(\Delta_{\tilde{\Xi}_t})$ is strictly convex as a function of $t\in [0,1]$. By assumption $\phi(0)=\vol(v_0)=\hvol(\wt_{\xi})$ is a minimum. So by the strict convexity we get $\phi(1)=\vol(\Delta_{\tilde{\Xi}_1})=\hvol(v)$ is strictly bigger than $\hvol(\wt_\xi)$.

\end{proof}
\begin{prop}\label{prop-volpoly}
For any $\xi\in {\rm relint}(\sigma)$ and the quasi-monomial valuation $v=(\eta, v^{(0)}) \in \Val_{V, o}$ as above, we
have the identity:
\begin{equation}
\vol(v_t)=\Vol(\Delta_{\tilde{\Xi}_t}).
\end{equation}
\end{prop}
\begin{proof}
The rest of this subsection is devoted to proof of the Proposition \ref{prop-volpoly}. Similar to the proof of Lemma \ref{l-conevolume}, we know Proposition \ref{prop-volpoly} follows if we can show that $\bV$ constructed above satisfies those properties stated in Lemma \ref{lem-sg3prop}, which in turn follow from the following uniform estimates: there exists a constant $C$ such that for any $f\in R_u$, 
\begin{equation}\label{eq-Izuest}
|\bV_2(f)|\le C\langle u, \xi\rangle, \quad |\bV_3(\chi_{m^*})| \le C \langle u, \xi\rangle.
\end{equation}
So we only need to concentrate on proving \eqref{eq-Izuest}.
To get the first inequality, we will use the same argument leading to \eqref{eq-bdV2}. To get the second estimate, we will use a sequence of divisorial valuations to approximate and prove that the estimates obtained are uniform with respect to approximations. 

Fix $u\in \sigma^{\vee}$ and $f=f_u \cdot \chi^u\in R_u$. For $b$ sufficiently divisible, $\fD(b u)$ is Cartier. We will denote by $L_{bu}$ the line bundle associated to $\fD(bu)$ and $L_u$ the $\mathbb{Q}$-line bundle associated to $\fD(u)$. 
Choose a global section $g_{bu}\in H^0(Y, \fD(b u))$ such that $g_{bu}^{-1}$ is a local equation for $\fD(bu)$ near $p\in Y$. Then for any $f\in H^0(Y, \fD(u))$ we have:
\begin{equation}\label{eq-decVf}
\bV_2(f)=\frac{1}{b} \bV_2(f^b g_{bu}^{-1})-\frac{1}{b}\bV_2(g_{bu}^{-1}).
\end{equation}
For the simplicity of notation, we write $g_u:=g_{bu}^{1/b}$ as a multi-section of the $\bQ$-line bundle $L_u$. Then \eqref{eq-decVf} can be written  as:
\begin{equation}\label{eq-decVf2}
\bV_2(f)=\bV_2(f g_u^{-1})-\bV_2 (g_u^{-1}).
\end{equation}
We will bound both terms on the right hand side of \eqref{eq-decVf2}. Using the piecewise linearity, we can  easily show as before that there exists $C>0$ independent of $u\in \sigma^{\vee}$, satisfying:
\begin{equation*}
|\bV_2(g_u^{-1})|\le C \langle u, \xi\rangle. 
\end{equation*}

So we only need to bound the term $\bV_2(f g_u^{-1})$. Denote $\tilde{f^b}=f^b g_{bu}^{-1}$. Consider the Taylor expansion of $F$ at $p$ as in \eqref{eq-flaurent}:
\begin{equation}\label{eq-Ftaylor}
\tilde{f^b}=z_1^{m^*_1}\dots z_s^{m^*_s}\chi_{m^*}(z'')+\sum_{m\in \bN^{s}, m\neq m^*} z_1^{m_1}\dots z_{s}^{m_{s}} \chi_m(z''),
\end{equation}
where $m^*=v^{(0)}(\tilde{f^b})$. Notice that $\tilde{f^b}$ is now regular at $p$. We can choose $\beta=(\beta'_1, \dots, \beta'_{s})\in \bQ^{s}$ sufficiently close to $(\beta_1, \dots, \beta_{s})$ such that:
\begin{equation}\label{eq-defk0}
k_0:=k_0(\beta')=\langle \beta', m^*\rangle < \langle \beta', m\rangle 
\end{equation}
for any $m\neq m^*$ appearing in \eqref{eq-Ftaylor}.

\bigskip

Next consider the weighted blow up of $Y$ along the smooth subvariety $Z=\bigcap_{i=1}^s \{z_i=0\}$ with weights ${\bf a}=(a_1, \dots, a_{s}):=(q \beta'_1, \dots, q \beta'_{s})$ where $q$ is the least common multiple of the denominators of $\beta'$. We will denote this weighted blow up by $\mu_{Y}=\mu_{Y, \beta'}: \tilde{Y}\rightarrow Y$ with the exceptional divisor denoted by $E=E_{\beta'}$. Since $Z$ is a smooth subvariety of $Y$, we have $E=\bP(N_Z, {\bf a})=(N_Z\setminus \{Z\})/\bC^*$, where $N_Z$ is the normal bundle of $Z\subset Y$ and $\tau\in \bC^*$ acts along the fiber of the normal bundle $N_Z\to Z$ by $\tau\circ (x_1, \dots, x_{s})= (\tau^{a_1}x_1, \dots, \tau^{a_{s}} x_{s})$. So we have a fibration $\pi_E: E\rightarrow Z$ with each fiber being isomorphic to the weighted projective space $\bP({\bf a}):=\bP(a_1, \dots, a_s)$. In particular, the inverse image of $p\in Z\subset Y$ under $\mu_Y$, denoted by $E_p$,  is a fiber of $\pi_E$ which is isomorphic to the weighted projective space $\bP({\bf a})$. 

Denoted by $v^{(0)}_{\beta'}$ the quasi-monomial valuation centered at $p\in Y$ but with the weight $\beta'$ instead of $\beta$. Then $v^{(0)}_{\beta'}$ is a divisorial valuation and  equal to $q^{-1}\cdot\ord_{E}$. Because of \eqref{eq-defk0}, 
\begin{eqnarray*}
k_0
=\frac{1}{b}v^{(0)}_{\beta'} (f^{bu}g_{bu}^{-1})=\frac{1}{bq}\ord_E \left(\widetilde{f^b}\right).
\end{eqnarray*} 
Then 
$$e:=(\mu_{Y}^*f^b-k_0 b q E)|_E \in H^0(E, (\mu_{Y}^*L_{bu}-k_0 b q E)|_E)$$ 
and moreover near $\pi_E^{-1}(p)$ we have
$e=\pi_E^*(\chi_{m^*}) \cdot s$,
where $s$ is a local generator of $(\mu_Y^*L_{bu}-k_0bq E)|_E$. 

Fix a very ample divisor $H$ on $Y$. There exists a sufficiently small positive constant $0<\epsilon_0=\epsilon_0(\beta')\ll 1$ such that
$$\tilde{H}:=\tilde{H}_{\epsilon_0}=\mu_Y^*H-\epsilon_0 q E$$ is still ample on $\tilde{Y}$. Then we have:
\[
(\mu_Y^*\fD(bu)-k_0 b q E)\cdot (\mu_Y^*H-\epsilon_0  q E)^{d-1}>0.
\]
This implies:
\begin{equation}\label{e-k0}
k_0 \le \frac{\mu_Y^*\fD(bu)\cdot (\mu_Y^*H-\epsilon_0 q E)^{d-1}}{b q E\cdot (\mu_Y^*H-\epsilon_0 q E)^{d-1}}=\frac{\mu_Y^*\fD(u)\cdot (\mu_Y^*H-\epsilon_0 q E)^{d-1}}{qE\cdot (\mu_Y^*H-\epsilon_0 q E)^{d-1}}.
\end{equation}
Thus it suffices to get a uniform estimate of the right hand side of \eqref{e-k0}, when $\beta'$ converges to $\beta$. This is similar to (but easier than) the argument of Lemma \ref{l-segre}. By Lemma \ref{l-uniformalg}, we know that $\epsilon_0>0$ can be fixed as a uniform constant. Using \eqref{eq-Segre}, the term contain $(qE)^i$ appears in the intersection will become rational Segre classes $s^i(\beta')$ which continuously depend on $\beta'$.   We also see that the denominator will be of the form 
$$\epsilon_{0}^{s-1}((-1)^s(\mu^*H)^{n-s}(qE)^{s}+(\mbox{higher order term of $\epsilon_0$}))$$
which is nonzero, if we fix $\epsilon_0$ to be sufficiently small. We postpone more details into the proof of the more complicated case treated by Lemma \ref{l-segre}. 
\begin{rem}
As mentioned above, this argument for the estimate of $k_0$ is essentially the same as the the argument in the proof of Lemma \ref{lem-sg3prop}. 
\end{rem}

Next we estimate $w_\gamma(\chi_{m^*})$. By Izumi's theorem, there exists a positive constant $C>0$ such that
\[
C^{-1}\cdot \ord_p\le w_{\gamma}\le C\cdot \ord_p 
\]
as valuations on $\bC(Z)$. So to prove the second estimate in \eqref{eq-Izuest}, we just need to estimate $\ord_p(\chi_{m^*})=\ord_p(e)$.

Let $\mu_E: \tilde{E}\rightarrow E$ denote the blow up of $E$ along the fiber $E_p$.  The exceptional divisor denoted by $F$ is isomorphic to $\bP^{d-s-1}\times \bP({\bf a})$. In fact, since we have a fibration $\bP({\bf a})\rightarrow E\rightarrow Z$, if $\mu_Z: \tilde{Z}\rightarrow Z$ denotes the blow up of $p\in Z$ with the exceptional divisor $F_Z\cong \bP^{d-s-1}$, then $\tilde{E}=\mu_Z^*(\bP(N, {\bf a}))$ and there is an induced fibration $\pi_{\tilde{E}}: \tilde{E}\rightarrow \tilde{Z}$. 
If we let 
$$k_1=\ord_{p}(e)=\ord_p(\chi_{m^*}),$$ then
$\mu_E^* e-k_1 b F\in H^0(\tilde{E}, \mu_E^*M-k_1 b F)$ where $M=(\mu_{Y}^*L_{bu}-k_0 b q E)|_E$. 

Since $\tilde{H}_{\epsilon_0}=\mu_Y^*H-\epsilon_0 q E$ is ample on $E$, there exists $\epsilon_1$ such that $\mu_E^*(\tilde{H}_{\epsilon_0}|_E)-\epsilon_1 F$ is ample on $\tilde{E}$. 
Then we have the inequality:
\[
(\mu_E^*e-k_1 b F) \cdot \left(\mu_E^*(\tilde{H}_{\epsilon_0} |_E)-\epsilon_1 F\right)^{d-2}> 0. 
\]
So we get the estimate:
\begin{eqnarray}\label{eq-k1est}
k_1 &\le& \frac{\mu_E^*\left( (\mu_Y^*L_{bu}-k_0 b q E)|_E\right) \cdot \left(\mu_E^*(\tilde{H}_{\epsilon_0} |_E)-\epsilon_1 F\right)^{d-2} }{b F\cdot \left(\mu_E^*(\tilde{H}_{\epsilon_0}|_E)-\epsilon_1 F\right)^{d-2}}\nonumber\\
&=&\frac{\mu_E^*\left((\mu_Y^*\fD(u)-k_0 q E)|_E\right) \cdot \left(\mu_E^*(\tilde{H}_{\epsilon_0}|_E) -\epsilon_1 F\right)^{d-2}}{F\cdot \left(\mu_E^*(\tilde{H}_{\epsilon_0}|_E)-\epsilon_1 F\right)^{d-2}}.
\end{eqnarray}
Finally we want to show that the above estimate can be made uniformly with respect to $\beta'$ that is sufficiently close to $\beta$. We first bound $\epsilon_0$ and $\epsilon_1$ uniformly in the following.
\begin{lem}\label{l-uniformalg}
$\epsilon_0$ and $\epsilon_1$ can be chosen to be uniform with respect to $\beta'$ that is close to $\beta$. More precisely, there exists $\delta=\delta(\beta)>0$, $\epsilon_0=\epsilon_0(\beta)>0$, $\epsilon_1=\epsilon_1(\beta)>0$ such that if $|\beta'-\beta|\le \delta$ then $\tilde{H}_{\epsilon_0}:=\mu_{Y, \beta'}^*H-\epsilon_0 q E_{\beta'}$ is ample on $\tilde{Y}$ and $\mu_E^*(\tilde{H}_{\epsilon_0}|_E)-\epsilon_1 F$ is ample on $\tilde{E}$.
\end{lem}
\begin{proof}We first verify there is a uniform $\epsilon$ such that $\tilde{H}_{\epsilon_0}$ is ample. In fact, it suffices to show the uniform $\epsilon$ for nefness.  We can assume $H$ is sufficiently ample, such that $H-D_i$ is ample for any $i=1,...,s$. Then $\mu_{Y, \beta'}^*D_i=\tilde{D}_i+q\beta_i'E$ where $\tilde{D}_i$ is the birational transform on $\tilde{Y}$. As $\bigcap^s_{i=1}\tilde{D}_i$ is empty, a curve $C$ on $\tilde{Y}$ is not contained in at least one $\tilde{D}_i$, which implies $C\cdot \tilde{D}_i\ge 0$. Thus 
$$C\cdot (\mu_{Y, \beta'}^*H-q\beta_i'E)\ge C\cdot \mu_{Y, \beta'}^*(H-D_i)\ge 0.$$
From this, we can easily see that we could take $\epsilon_0=\frac{1}{2}\min \{\beta_i\}$ and $\tilde{H}_{\epsilon_0}$ is ample for $|\beta'-\beta|<\epsilon_0$. 

\bigskip

Now we can similarly argue for $\epsilon_1$. For this we need  $H-D_i$ is ample for any $i=1,...,n-r$. Denote by $F_1$,..., $F_{d-s}$ the restrictions of the birational transforms of $D_{s+1},..., D_{d}$ on $ \tilde{E}$. Then for an irreducible curve $C$ on $\tilde{E}$, if its image on  it is not contained in one of $F_j$ for some $j$. Then $\mu_E^*(\mu^*_YD_{j}|_E)=F_j+F,$ which implies that 
$$\left(\mu_E^*(\tilde{H}_{\epsilon_0}|_E)-\epsilon_1 F\right)\cdot C\ge  \left(\mu_E^*(\tilde{H}_{\epsilon_0}-\epsilon_1\mu_Y^*D_{j} )|_E\right)\cdot C \ge 0.$$
Thus we can take $\epsilon_1=1$ and replace  $\tilde{H}_{\epsilon_0} $ by $\tilde{H}_{\epsilon_0}+\mu_Y^*H$. 
\end{proof}

Notice that we have
the commutative diagram:
\begin{equation}\label{CD-blow}
\begin{CD}
F @>>> \tilde{E} @>\mu_E>> E @>>> \tilde{Y}\\
@VVV @VV \pi_{\tilde{E}}V @VV\pi_E V @VV\mu_Y V \\
F_Z@>>> \tilde{Z} @>\mu_Z >> Z @>>> Y.
\end{CD}
\end{equation}
We also have $(-E)|_E=\bP(N_Z, {\bf a})(1)$ and $\mu_E^*(-E)|_E=\bP(\mu_Z^*N_Z, {\bf a})(1)$.
So 
\begin{eqnarray}\label{eq-lb1}
\mu_E^*((\mu_Y^*H-\epsilon_0 q E)|_E)-\epsilon_1 F 
&=&\pi_{\tilde{E}}^*(\mu_Z^*H|_Z-\epsilon_1 F_Z)-\epsilon_0 q \cO_{\bP(\mu_Z^*N_Z, {\bf a})}(1)
\end{eqnarray}
is ample.

\begin{lem}\label{l-segre}
The right-hand-side of \eqref{eq-k1est} is uniformly bounded independent of $\beta'$ if $|\beta'-\beta|\le \delta$ where $\delta=\delta(\beta)$ is the same one as that in Lemma \ref{l-uniformalg}.
\end{lem}
\begin{proof}
As $F\cong \bP^{d-s-1}\times \bP({\bf a})$, by \eqref{eq-lb1} we get the denominator of \eqref{eq-k1est} to be 
\[
F\cdot \left(\mu_E^* (\tilde{H}_{\epsilon_0}|_E)-\epsilon_1 F\right)^{d-2}=\epsilon_1^{d-s-1} \epsilon_0^{s-1} q^{s}\frac{1}{a_1\cdots a_{s}}=\epsilon_1^{d-s-1}\epsilon_0^{s-1}\frac{1}{\beta'_1\cdots\beta'_s}.
\]
By the commutative diagram \eqref{CD-blow}, we have:
\begin{equation}\label{eq-lb2}
\mu_E^*\left[(\mu_Y^*\fD(u)-k_0 q E)|_E\right]=\pi_{\tilde{E}}^*(\mu_Z^* (\fD(u)|_Z))-k_0 q \cO_{\bP(\mu_Z^*N_Z, {\bf a})}(1)
\end{equation}
For simplicity, we denote 
$$B_1=\pi_{\tilde{E}}^*(\mu_Z^* \fD(u)|_Z), B_2=\mu_Z^*H|_Z-\epsilon_1 D, G=\cO_{\bP(\mu_Z^*N_Z, {\bf a})}(1).$$
Using \eqref{eq-lb1} and \eqref{eq-lb2}, the numerator of \eqref{eq-k1est} is equal to:
\begin{eqnarray}\label{eq-simplenum}
&&(B_1-k_0 q G)\cdot (B_2-\epsilon_0 q G)^{d-2}\nonumber\\
&&\hskip 7mm =(B_1-k_0 qG)\cdot \sum_{i=0}^{d-2} \binom{d-2}{i} B_2^{d-2-i} (\epsilon_0 q)^{i}(-G)^{i}.
\end{eqnarray}
By the standard intersection theory (cf. \cite[Section 4]{Ful98}), we have:
\begin{equation}\label{eq-Segre}
\pi_{\tilde{E}}^*A_1\cdots\pi_{\tilde{E}}^* A_{d-1-i}\cdot (-G)^i=A_1\dots A_{d-1-i}\cdot s_i(Q, {\bf a}),
\end{equation}
where $s_i(Q, {\bf a})$ is the weighted Segre class which can be defined as follows (cf. \cite[Section 3.1]{Ful98}): Let the total weighted Chern class of $Q=\mu_Z^*N_Z$ be given by 
$$c(Q, {\bf a})=\prod_{i=1}^r \left(1+a_i^{-1} c_1(L_i) t\right) \mbox{ where } L_i=\mu_Z^*(D_i|_Z),$$ 
and 
$\sum_i s_i(Q, {\bf a}) t^i=c(Q, {\bf a})^{-1}$, then $s_i(Q, {\bf a})=q^{-i} \tilde{s}_i(Q, \beta')$ where $\tilde{s}_i(Q, \beta')$ depends only on $\beta'$ and $c_1(L_i)$. Using \eqref{eq-Segre}, we see that each term in \eqref{eq-simplenum} depends continuously on 
$\beta'$. So we can indeed make the numerator of \eqref{eq-k1est} uniform with respect to $\beta'$. 
\end{proof}
\end{proof}

\section{Models and degenerations}\label{s-models}\label{s-constructmodel}

In \cite{LX16}, we showed that a divisorial minimizer always comes from a Koll\'ar component, and it can yield a degeneration which is the key for us to deduce results on general klt singularities from cone singularities. However, it is less clear, at least to us, what should be the corresponding construction for a higher rank quasi-monomial valuations. Nevertheless, in this section, we try to develop an approach to use models to approximate a quasi-monomial valuations, with possibly higher rank.

\subsection{Weak lc model of a quasi-monomial minimizer}\label{s-vmodel}

Let us first fix some notation. Let $v\in \Val_{X,x}$ be a quasi-monomial valuation. We know that there exists a log smooth model $(Y,E)$ over $X$ such that $v$ is computed at its center $\eta$ on $(Y,E)$ (see Definition \ref{d-quasi}). Denote by $E_j$ ($j=1,...,r$) the components of $E$ containing $\eta$.  In the below, we look at valuations $v_{\alpha}$ computed on $\eta\in (Y,E)$ for $\alpha\in  \mathbb{R}^r_{\ge 0}$. In fact, if we rescale $v_{\alpha}$ such that $A_{X,D}(v_{\alpha})=1$, then all such points canonically form a simplex $\Delta\subset \Val_{X,x}$ with the vertices given by $v_j=\ord_{E_j}/ A_{X,D}(E_j)$. 

In the above setting, for any two toroidal valuations which can be written $v_{\alpha}$ and $v_{\alpha'}$ on the fixed model $Y\to (X,D)$, we define 
$$|v_{\alpha}-v_{\alpha'}|:=|\alpha-\alpha'|_{\rm sup}=\max_{1\le j \le r } |\alpha_j-\alpha'_j|.$$  Recall we have defined the volume of a model $\vol(Y/X)$ (or abbreviated as $\vol(Y)$) over a klt singularity $(X,D)$ in \cite{LX16}. 

\begin{defn}\label{d-wlc}
Let $x\in (X,D)$ be a klt singularity and $v\in \Val_{x,X}$ a quasi-monomial valuation. We say that $v$ admits a {\bf weak  lc model} if there exists a birational morphism $\mu\colon Y^{\rm wlc}\to X$ such that
\begin{enumerate}
\item[(a)] ${\rm Ex} (\mu)=\mu^{-1}(x)=\sum^r_{j=1} S_j$;
\item[(b)]  $(Y^{\rm wlc},\mu_*^{-1}D+\sum^r_{j=1} S_j)$ is log canonical;
\item[(c)]  $-K_{Y^{\rm wlc}}-\mu_*^{-1}D-\sum^r_{j=1} S_j$ is nef over $X$; and 
\item[(d)]  $(Y^{\rm wlc},\mu_*^{-1}D+\sum^r_{j=1} S_j)$ is q-dlt  at the generic point $\eta$ of a component of the intersection of $S_j$ $(j=1,2,..., r)$, where $v$ can be computed (see Definition \ref{d-quasi}).
\end{enumerate}
\end{defn}

The main theorem of this section is the following. 

\begin{thm}\label{t-wlcmodel}
Let $x\in (X,D)$ be a klt point. If $v\in \Val_{X,x}$ is a quasi-monomial valuation which minimizes $\hvol_{(X,D)}$, then it admits a $\mathbb{Q}$-factorial weak log canonical model. 
\end{thm}


\begin{proof}
We fix a model as in Definition \ref{d-quasimono} which computes $v=v_{\alpha}$. We can further assume that the codimension of $\eta$ in the model is the same as the rational rank of $v$.  Let $\fa_{\bullet}$ be the valuative ideals of $v$, i.e., $\fa_{k}=\{f |\ v(f)\ge k\}$. Let $c={\rm lct}(X,D, \fa_{\bullet})$. 
\begin{lem}\label{l-extractdivisor}
For any $\epsilon$, there exists toroidal divisors $S_1,...,S_r$ given by vectors $s_1$,..., $s_r\in \mathbb{Z}^r_{>0}$ and $\epsilon_0>0$, such that 
\begin{enumerate}
\item $(X, D+(1-\epsilon_0)c\cdot \frac{1}{m}\mathfrak{a}_m)$ has a positive log discrepancy for any divisor $E$;
\item  $v$ is in the convex cone generated by $s_1,...,s_r$, and 
\item for any $j$, the log discrepancy $a_l(S_j, X,D+(1-\epsilon_0)c\cdot \frac{1}{m}\mathfrak{a}_m)<\epsilon$ for any $m\gg 0$.
\end{enumerate}
\end{lem}
\begin{proof}Applying Lemma \ref{l-approvector}, for any $\epsilon$, we can find vectors $ s_j=(s_{ij})=(\frac{r_{ij}}{q_j})_{1\le i \le r}$  $(j=1,2,...r)$ with $r_{ij}, q_j\in \mathbb{N}$,
such that
\begin{enumerate}
\item the vector $\alpha=(\alpha_i)$ can be written as 
$$\alpha=\sum^r_{j=1} c_js_j \qquad\mbox{with $c_j>0$};$$
\item for any $i$ and $j$,  
$$|s_{ij} -\alpha_j|<\frac{\epsilon}{q_j}.$$
\end{enumerate}
Clearly, we can assume ${\rm gcd}(r_{1j},..., r_{rj},q_j)=1$. Thus we conclude that for the integral divisor corresponding to the toric blow up of
$(E_1,...,E_r)$ with coordinates $(r_{1j},..., r_{rj})$, we get an exceptional divisor $S_j$. Fix  a sufficiently large positive number $N_0$ such that $q_j\le N_0$.

Since $v$ is a minimizer of $\hvol(X,D)$, we know that 
\begin{eqnarray*}
\hvol(v)&= &\lim_{k\to \infty}A_{(X,D)}(v)^n \cdot \frac{\mult(\fa_k)}{k^n}\\
 &\ge& \lim_{k\to \infty}\lct (X,D;\fa_k)^n \cdot \mult(\fa_k)\\
 & \ge&\hvol(v),
\end{eqnarray*}
where $\mathfrak{a}_k=\{f|\ v(f)\ge k\}$. 
 So we conclude that (see \cite{Mus02})
$$\lct(X,D;\fa_{\bullet}):=\lim_{k\to \infty} k\cdot \lct (X,D;\fa_k)=A_{(X,D)}(v),$$ 
which we denote by $c$. 
Denote by $\epsilon_0=\frac \epsilon {N_0}$, and pick up $N$ such that
$$a_{l}(E_i, X, D+c \cdot  \fa_{\bullet})\le N \qquad \mbox{for any $i$},$$
then for any  $m$ sufficiently large such that 
\begin{enumerate}
\item[(a)] $( X,D+(1-\epsilon_0 )c\cdot \frac{1}{m}\fa_m)$ is klt, 
\item[(b)] $a_{l}(E_i, X, D+(1-\epsilon_0)c \cdot \frac{1}{m} \fa_m)\le N$,
\end{enumerate}
 we have
\begin{eqnarray*}
& &  a_l(\frac{S_j}{q_j}; X, D+(1-\epsilon_0 )c\cdot \frac 1 m\fa_m)-a_l(v; X, D+(1-\epsilon_0 )c\cdot \frac{1}{m} \fa_m)\\
&\le & \sum^r_{i=1} |\alpha_i-\frac{r_{ij}}{q_j}|\cdot a_l(E_i; X, D+(1-\epsilon_0 )c\cdot \frac{1}{m} \fa_m) \\
&\le&\sum^r_{i=1}\frac{\epsilon}{q_j}\cdot a_l(E_i; X, D+(1-\epsilon_0 )c\cdot \frac{1}{m} \fa_m).\\
&\le&\frac{rN\epsilon }{q_j}.
\end{eqnarray*}
So for any $j$, since $a_l(v; X, D+c\cdot \frac{1}{m} \fa_m)\le 0$, we have
\begin{eqnarray*}
a_l(S_j; X, D+(1-\epsilon_0 )c\cdot \frac{1}{m} \fa_m) &\le & {rN\epsilon }+q_j\cdot a_l(v; X, D+(1-\epsilon_0 )c\cdot \frac{1}{m} \fa_m)\\
&= &{rN\epsilon }+ q_j( c-(1-\epsilon_0)c\frac{1}{m}v(\fa_m))\\
&\le &{rN\epsilon }+ N_0\epsilon_0c\\
&\le&(rN+c)\epsilon. 
\end{eqnarray*} 
Here all the constants $r$, $N$ and $c$ only depend on the fixed log resolution $Z\to (X,D)$ and $v$ but not the choices of $S_j$. So we can replace the constant $\epsilon$ and obtain the lemma. 
\end{proof}

By Lemma \ref{l-extractdivisor}, we can use \cite{BCHM10} to construct a $\mathbb{Q}$-factorial model $\mu\colon Y\to X$ such that 
\begin{enumerate}
\item $S_1$,..., $S_r$ are the only exceptional divisors,
\item  there is an effective $\mathbb{Q}$-divisor $L$ on $X$, such that
$(Y, \mu_*^{-1}(D+L)+\sum^r_{j=1} a_jS_j)$ is klt with $1-\epsilon<a_j<1$, and 
$$K_Y+\mu_*^{-1}(D+L)+\sum^r_{j=1} a_jS_j\sim_{\mathbb{Q},X}0.$$ 
\end{enumerate}
Furthermore, we can assume $Y$ is obtained by running an MMP from a toroidal blow up of $Z$ extracting $S_1$,..., $S_r$. In particular, $Z\dasharrow Y$ does not change the generic point of the component of the intersection of $\bigcap^r_{j=1} S_j$. 

Then by running a relative MMP $g\colon Y\dasharrow Y_1$ over $X$ for 
$$-K_{Y}-\mu_*^{-1}D-S\sim_{\mathbb{Q},X}-\sum^r_{j=1}a_l(X,D,S_j)S_j,$$
we obtain a model $\mu_1\colon Y_1\to X$,  such that $-K_{Y_1}-\mu_{1*}^{-1}D-g_*S$ is nef. 

By the ACC of the log canonical thresholds (see \cite{HMX14}), we know that there exists a constant $\beta<1$ which only depends on $\dim X$ and the coefficients of $D$ such that if we choose $1-\epsilon>\beta$, then we can conclude that $(K_{Y}, \mu_*^{-1}D+S)$ is log canonical, as  $(K_{Y}, \mu_*^{-1}D+\beta \cdot S)$ is log canonical. Further, the same holds for $Y_1$ as
$\left(Y_1,g_*(\mu_*^{-1}(D+L)+\sum^r_{j=1} a_jS_j)\right)$ is log canonical. This implies that the MMP process $Y\dasharrow Y_1$ is isomorphic around any lc center of $({Y},\mu_*^{-1}D+S)$, since otherwise we will have for a divisor $E$ over this center with
$$-1=a(E; Y,\mu^{-1}_{*}D+ S)>a(E; Y_1, \mu_{1*}^{-1}D+g_*(S)),$$
which is a contradiction. 

Therefore, $Y_1\to X$ gives a weak log canonical model. 
\end{proof}

\begin{rem}The above argument indeed  implies that any  quasi-monomial valuation which computes the log canonical threshold of a graded sequence of ideals has a a weak log canonical model. 
\end{rem}

In the below, we want to show that the volumes of the models produced in Theorem \ref{t-wlcmodel} converge to $\hvol(v)$ if the corresponding simplexes converge to $v$. 

\begin{lem}\label{l-logdisapp}
Fix $f\colon Y\to (X,D)$ a log resolution, with $Z$ a component of the intersection of some exceptional divisors $E_i$ $(i\in J)$ on $Y$. Then there exists a constant $N$  (depending $Y\to (X,D)$) which satisfies the following property: Let $S$ be a toroidal divisor over $\eta(Z)\in (Y,E=\sum_{i\in J} E_i)$ and $\mu\colon Y_{S}\to X$ a weak lc model over $X$ with the only exceptional divisor $S$, then for any quasi-monomial valuation $v$ computed at the generic point $\eta(Z)\in (Y,E)$ satisfying $|a\cdot \ord_S-v|<\epsilon$ for some $a>0$, we have
$$A_{Y_S, \mu^{-1}_*D+S}(v)<N\cdot \epsilon.$$ 
\end{lem}
\begin{proof} Denote by $E_i$ ($i\in I\supset J$) all the exceptional divisors of $Y$ over $X$ and by $A_{(X,D)}(E_i)=a_i$, then we know that 
$$b_i=_{\rm defn}A_{Y_S, \mu^{-1}_*D+S}(E_i)< a_i.$$
We will show $N=\sum_{i\in I} a_i$ suffices. 
Let $Y'\to Y$ be the toroidal blow up extracting $S$ (if $S$ is on $Y$, then we let $Y'=Y$) and denote by  $g\colon Y'\to X$.
Define the divisor $F=g^{-1}_*D+S+\sum_{i\in I}b_iE_i$ on $Y'$.
Then as $|a\cdot \ord_S-v|<\epsilon$ and $A_{Y',F}(\ord_S)=0$ , let $J_0\subset J$ be all the indices $E_i$ which contain the center of $v$, we know 
$$A_{Y',F}(v)\le \epsilon \cdot \left(\sum_{i\in J_0} A_{Y',F}(E_i)\right)=\sum_{i\in J_0} b_i \cdot \epsilon<\sum_{i\in I}a_i \cdot\epsilon.$$ 
Finally, since $K_{Y_S}+\mu^{-1}_*D+S$ is anti-nef, by the negativity lemma, we know that  
$$A_{Y_S, \mu^{-1}_*D+S}(v)\le A_{Y',F}(v).$$
\end{proof}


 \begin{lem}\label{l-modelconverge}
Let $\Delta_i\subset \Delta$ be a sequence of subsimplices with rational vertices, such that each $\Delta_i$ corresponds to a weak lc model constructed in Theorem \ref{t-wlcmodel} and the vertices $v^j_{i}$ of $\Delta_i$ converge to $v$.  
Then $$\lim_{i\to \infty} \vol(Y_i)=\hvol(v). $$
\end{lem}
\begin{proof} 
By the proof of Theorem \ref{t-wlcmodel} we know
\begin{enumerate}
\item  there exists $\mu_i\colon Y_i\to X$ a $\mathbb{Q}$-factorial weak lc model which precisely extracts the divisors $S_{i,j}$ corresponding to the prime integral vector on $\mathbb{R}_{>0}v_i^j$, and  
\item for any $\epsilon$, we can find $i$ sufficiently large such that for any $(i,j)$ there exists a constant $q_j$ such that  $|\ord_{S_{i,j}}-q_j\cdot v|< \epsilon$.
\end{enumerate}

Then for any $\epsilon$, we can find $i$ sufficiently large such that if we pick a rational vector $v_*$ which after a rescaling by $q_j$ is sufficiently close to $v$ with
$$|\ord_{S_{i,j}}-q_j\cdot v_*|\le |\ord_{S_{i,j}}-q_j\cdot v|+|q_j\cdot v_*-q_j\cdot v|\le \epsilon.$$
By the proof of Theorem \ref{t-wlcmodel}, we can assume there exists $Y_*\to X$ a weak lc model of $v_*$ extracting the corresponding divisor $S_*$. Then we claim  
$$A_{Y_*, \mu^{-1}_*D+S_*}({S_{i,j}})<\epsilon \cdot A_{(X,D)}({S_{i,j}}).$$
 Granted it for now, we know the log pull back of $K_{Y_*}+ \mu^{-1}_*D+S_*$ is larger or equal to the pull back of 
$$K_{Y_i}+\sum^r_{j=1} \left(1-\epsilon A_{(X,D)}({S_{i,j}})\right)S_{i,j}+\mu_{i*}^{-1}D,$$
which means $\hvol(v_*)\ge (1-\epsilon)^n\vol(Y_i)$.
 
 \medskip

We continue to show the claim by a similar argument as in Lemma \ref{l-logdisapp}: fix $Z\to X$ a log resolution. Denote by $E_k$ ($k\in \{1,2,...,q\}$)  the exceptional divisors of $Y$ over $X$ whose intersection gives the center ${\rm Center}_Y(v)$ and denote by $A_{(X,D)}(E_k)=a_k$.
Denote the corresponding vector of $S_{i,j}$  by $(n_1,...,n_q)$, so $A_{(X,D)}(S_{i,j})=\sum^q_{k=1} n_ka_k$. We consider a model $Z_*\to Z$, which extracts  the birational transform $E_*$ of $S_*$. Then for any $S_{i,j}$, after relabelling, we can assume its vector is in the fan generated by $E_1,...,E_{q-1}$ and $E_*$. Therefore, by the same argument as in Proposition \ref{l-logdisapp}, since $|\ord_{S_{i,j}}-q_j\cdot v_*|<\epsilon$, we know the log discrepancy of $S_{i,j}$ with respect to $({Y_*}, \mu^{-1}_*D+S_*)$ is less or equal to 
$$\epsilon(\sum^{q-1}_{i=1} a_in_i)\le \epsilon (\sum_{i=1}^{q}a_in_i)= \epsilon \cdot A_{(X,D)}(S_{i,j}).$$
\end{proof}

\medskip
In the below, we take a detour to illustrate on how we use the models to understand the limiting process in \cite[Theorem 1.3]{LX16} when $v$ is quasi-monomial. 
 
 A weak log canonical model provides us an explicit subset of valuations which can be used to understand the original abstract approximation process by Koll\'ar components in \cite[Theorem 1.3]{LX16}:  
Let $v$ a quasi-monomial minimizing valuation of a klt pair $x\in (X,D)$ and $Y^{\rm wlc}\to X$ be a weak log canonical model of $(X,D)$ for $v$ given by Theorem \ref{t-wlcmodel}. Fix $f^{\rm dlt}\colon Y^{\rm dlt}\to Y ^{\rm wlc}$ a dlt modification of $(Y^{\rm wlc},\mu_*^{-1}D+\sum^r_{j=1} S_j)$. Write the  pull back of $K_{Y^{\rm wlc}}+\mu_*^{-1}D+\sum^r_{j=1} S_j$ to be $K_{Y^{\rm dlt}}+\Delta^{\rm dlt}$. As in \cite{dFKX17}, we can formulate the dual complex $\DR(\Delta^{\rm dlt})$, which does not depend on the dlt modification. Moreover, as in the case of simplex, $\DR(\Delta^{\rm dlt})$ also forms a natural subspace of 
$$\Val^{=1}_{X,x}:=\{v\in \Val_{X,x} |\  A_{X,D}(v)=1\}\subset \Val_{X,x}$$
 (see \cite{MN15,NX16} for more discussions on the background).

We will need a strengthening of \cite[Theorem 1.3]{LX16}. 
\begin{prop}\label{p-qdlt} 
There exists a sequence of Koll\'ar components $T_i$ whose rescalings correspond to rational points on $\DR(\Delta^{\rm dlt})$, i.e. $a_l(T_i; Y^{\rm wlc},\mu_*^{-1}D+\sum^r_{j=1} S_{j})=0$, such that,
$$c_i\cdot \ord_{T_i}\to v,$$
where $c_i=\frac{1}{A_{X,D}(T_i)}$.
\end{prop}
\begin{proof}We will construct a sequence of Koll\'ar component $T_i$ and choose $c_i=\frac{1}{A_{X,D}(T_i)}$, such that
\begin{enumerate}
\item [a)] $a_l(T_i; Y^{\rm wlc},\mu_*^{-1}D+\sum^r_{j=1} S_{j})=0$.
\item [b)] $\{ c_i\cdot \ord_{T_i }\}$ has a limit $v'\ge v$.
\item [c)] $\vol(v')=\vol(v)$.
\end{enumerate} 
In fact, from \cite[Proposition 2.3]{LX16}, we can conclude $v=v'$. 

\bigskip

Denote by $\Delta_0$ the simplex in $\Delta^{\dlt}$ generated by $S_1$,..., $S_r$ around $\eta$.  By Theorem \ref{t-wlcmodel}, we can find a sequence of rational simplices $\{\Delta_i\}_{i=1}^{\infty}$ with vertices $S_{i,j}$ $(j=1,...,r)$ such that $\Delta_{i+1}\subset \Delta_{i}$, $\lim_i{\Delta_{i}}=v$ and for each $\Delta_i$ we have a weak log canonical model $\mu_i\colon Y^{\rm wlc}_i\to X$.  Furthermore, we can require the constant $\epsilon_i$ and $\epsilon_{0,i}$ in Lemma \ref{l-extractdivisor} converges to 0. 

\medskip
By the negativity lemma we know that on a common resolution the pull back of $K_{Y_i^{\rm wlc}}+\mu_{i*}^{-1}D+\sum^r_{j=1} S_{i,j}$ is larger or equal to the pull back of $K_{Y_{i+1}^{\rm wlc}}+\mu_{i+1*}^{-1}D+\sum^r_{j=1} S_{i+1,j}$. In particular, for any divisor $T$ and $i$ with 
$$a_l(T; Y_i^{\rm wlc}, \mu_{i*}^{-1}D+\sum^r_{j=1} S_{i,j})=0,$$ $T$ is contained in $\DR(\Delta^{\rm dlt})$. By the proof of \cite[Lemma 3.8]{LX16} we can find a Koll\'ar component $T_i$ such that 
$$\hvol(\ord_{T_i})\le \vol(Y_i^{\rm wlc}) \qquad \mbox{and}\qquad  a_l(T_i; Y_i^{\rm wlc}, \mu_{i*}^{-1}D+\sum^r_{j=1} S_{i,j})=0.$$
We denote by $v_i=c_i\cdot \ord_{T_i}$ where $c_i=\frac{1}{A_{X,D}(T_i)}$, then $\vol(v_i)=\hvol(\ord_{T_i})$. Since $\vol(P)$ is a continuous function on the compact set ${\DR(\Delta^{\rm dlt})}$ (see \cite{BFJ14}), we know after replacing by a sequence, we can assume that $v_i$ has a limit $v'$ and we know that
$$\vol(v')=\lim_i \vol(v_i)\le \lim_i \vol (Y_i^{\rm wlc})=\vol(v),$$
where the last equality follows from Lemma \ref{l-modelconverge}. Since $\vol(v)=\hvol(v)\le \hvol(v')=\vol(v')$, indeed the equality holds. 

\medskip

It remains to show property b) holds, which is similar to the proof of \cite[Theorem 1.3]{LX16}. Denote by $v(f)= p$. For a fixed $i$, choose $l=\lceil i/p \rceil$.  Denote by $m_{i,j}$ the vanishing order of $\fa_i$ along $S_{i,j}$. Since $A_{X,D}(v)=1$ and $v$ computes the log canonical threshold of $\{\fa_{\bullet}\}$, we know $c=\lct(X,D; \fa_{\bullet})=1.$  Then by Lemma \ref{l-extractdivisor}.3, there exists $a_{i,j}>(1-\epsilon_i)$  $(1\le j \le r)$ such that
$$K_{Y^{\rm wlc}_{i}}+\mu_{i*}^{-1}D+\sum^{r}_{j=1} a_{i,j}S_{i,j}\sim_{\mathbb{Q},X} (1-\epsilon_{0,i})\cdot\frac{1}{i} \sum m_{i,j}S_{i,j}.$$ 
Therefore, we have 
$$K_{Y^{\rm wlc}_{i}}+\mu_{i*}^{-1}D+\sum^{r}_{j=1} S_{i,j}\sim_{\mathbb{Q},X}\sum^r_{j=1} \left( 1-a_{i,j}+(1-\epsilon_{0,i}) \cdot \frac{1}{i}  m_{i.j}\right)S_{i,j}.$$ 

Then we have the following implications:
\begin{eqnarray*}
 v(f)= p&\Longrightarrow & v(f^l)= pl,\\
&\Longrightarrow & f^l\in \fa_{pl},\\
&\Longrightarrow & f^l\in \fa_{i},\\
&\Longrightarrow &l\cdot  \ord_{S_{i,j}}(f)\ge m_{i,j} \mbox{\ \  for any $1\le j\le r$},\\
&\Longrightarrow& l\cdot  v_i(f)\ge  \min_j\frac{m_{i,j}}{1-a_{i,j}+(1-\epsilon_{0,i})\frac{1}{i} m_{i,j}},\\
&\Longrightarrow&v_{i}(f)\ge  \min_j\frac{i/l}{(1-a_{i,j})\frac{i}{m_{i,j}}+(1-\epsilon_{0,i}) },
\end{eqnarray*}
where the fifth arrow follows from the negative lemma. Recall 
$$\lim_{i\to \infty} a_{i,j}=1, \qquad \lim_{i\to \infty} \epsilon_{0,i}=0, \qquad \mbox{and  \ \  }\lim_{i\to \infty} \frac{i}{l} =p.$$
As $A_{X,D}(S_{i,j})\to \infty$ when $i\to \infty$, we know that $\lim_{i\to \infty} \frac{m_{i,j}}{i}\to \infty$,  Thus 
$$v'(f)=\lim_{j\to \infty} v_{i}(f)\ge p= v(f).$$

\end{proof}

\subsection{K-semistability and minimizing}\label{s-Ksta}

In this section, we aim to prove the a quasi-monomial valuation $v$ is a minimizer only if it is K-semistable. As we mentioned, we need to make the expected technical assumption that the associated graded ring ${\rm gr}_v(R)$ is finitely generated.

 \begin{defn}\label{d-specialreeb}
  Let $x\in X={\rm Spec}(R)$ be a normal singularity.  Let $v\in \Val_{X,x}$ be a valuation and we assume ${\rm gr}_v(R)$ is finitely generated.  Denote $X_0={\rm Spec}({\rm gr}_vR)$.    Let the rational rank of $v$ be $r$. Then there is a $T={(\mathbb{C}^*)^r}$-action on $X_0$ induced by the $\bZ^r$-grading.   We denote by $\xi_v  \in \ft_{\bR}^+$ the natural vector given by the valuation $v$, namely $\xi_v(f)= \min_{f_{\alpha}\neq 0}\{ \alpha\}$ for any $ f=\sum_{\alpha} f_{\alpha}\in {\rm gr}_v(R)$. 
\end{defn}

In this section, we will always consider the degeneration induced by a valuation in the following case: $x\in (X,D)$ is  a klt singularity, $v$ is a quasi-monomial valuation over $x$  whose associated graded ring is finitely generated. We denote $X_0={\rm Spec}({\rm gr}_v R)$.  By Lemma \ref{l-finiteass}, we can choose a sequence $v_i\to v$, where $v_i$ is a rescaling of a divisorial valuation, denoted by $S_i$ over $x$ and we have ${\rm gr}_vR\cong {\rm gr}_{v_i}R$.

\begin{lem}[see \cite{Ish04}]\label{l-fakekollar}
Under the above assumption, we can construct a model $\mu\colon Y\to X$ such that the only exceptional divisor is $S_i$ and $-S_i$ is ample over $X$. 
\end{lem}
\begin{proof}Let $\{\fa_{\bullet}\}$ be the valuative ideal sequence of $\ord_{S_i}$. It suffices to prove that 
$\bigoplus_i \fa_i$
is finitely generated given the associated graded ring is finitely generated. We lift generators $\bar{f}_i\in \fa_{j_i}/\fa_{j_i+1}$ $(i=1,...,r)$ of ${\rm gr}_v(R)$ to elements $f_i\in \fa_{j_i}$. Let $d=\max j_i$, let $\{g_i\}$ ($1\le i \le k$) be a set of generators of $\fa_j$ for $(0\le j \le d)$, then we show that $\bigoplus_i \fa_i$ is generated by $
\{g_i\}$. We denote the graded ring generated by $\{ g_i\}$ to be $ \bigoplus_i  \fb_i \subset \bigoplus_i  \fa_i$.

Consider $\fa_m$, we claim  $\fb_m+\fa_{m+p}=\fa_m$ for any $p\ge 0$, which clear implies $\fb_m=\fa_m.$ For $p=0$, this is trivial. Assume we have proved this for $p=p_0-1$. Then for any $f\in \fa_m$, we can write $f=g+f'$ where $g\in \fb_m$ and $f'\in \fa_{m+p_0-1}$. Since 
$$[f']\in \fa_{m+p_0-1}/\fa_{m+p_0}=\sum_\alpha a_{\alpha}f_{\overline{i}}^{\alpha},$$
where $f_{\overline{i}}^{\alpha}=f_1^{\alpha_1}\cdots f_r^{\alpha_r}=f_1\cdots f_1\cdot f_2\cdots f_r$ is a product of $\alpha_1+\cdots \alpha_r$ terms and $\sum^r_{i=1} j_i\cdot \alpha_i=m+p-1$. By considering some $f_i$ to be in $\fa_{j_i'}$ instead of $\fa_{j_i}$ for some $0\le j'_{i}\le j_i$ as $\fa_{j_i'}\supset\fa_{j_i}$, we can assume 
$f_1^{\alpha_1}\cdots f_r^{\alpha_r}$ is in $\fa_m$, which is then by definition in $\fb_m$. 
\end{proof}

We assume that $X_0$ is normal, and define $D_0$ to be the closure of $D$ (as $\mathbb{Q}$-divisor) in the following way: for each prime Weil divisor $E$ on $X$ with the ideal $p_E$, we can consider the degeneration $\bin(p_E)$ and let the degeneration $E_0$ to be its divisorial part. Then for a general $\mathbb{Q}$-divisor $D=\sum_i a_iE_i$, we define $D_0=\sum_i a_iE_{i,0}$.

\begin{lem}\label{lem-irrid}
With the above notation, we have $\vol(v)=\vol(\wt_{\xi_v})$. Furthermore,  $K_{X_0}+D_0$ is $\mathbb{Q}$-Cartier and $A_{(X,D)}(v)=A_{(X_0,D_0)}(\wt_{\xi_v})$. 
\end{lem} 
\begin{proof}The first part is straightforward. For the second part, since the closure of a Cartier divisor ${\rm div}(f)$ is given by the Cartier divisor ${\rm div}(\bin_v(f))$, we see that $K_{X_0}+D_0$ is $\mathbb{Q}$-Cartier as $K_X+D$ is $\mathbb{Q}$-Cartier. 

 Then as
${\rm gr}_{v}(R)$ is finitely generated, we conclude that there is a sequence of rescaling of divisorial valuations $\xi_{v_i}$ approximating  $\xi_v$ such that ${\rm gr}_{v_i}(R)$ is isomorphic to $X_0$ (see Lemma \ref{l-finiteass}). By Lemma \ref{l-fakekollar}, we know that the degeneration $X_0$ can be considered as an orbifold cone over $S_i$, which is normal since we assume $X_0$ is normal. In particular, we have
$$A_{(X,D)}(v_i)=A_{(X_0,D_0)}(\wt_{\xi_{v_i}}).$$
Taking the limit $i\to \infty$, we get the statement. 
\end{proof}

If $v$ is a minimizer such that the associated graded ring $\gr_{v}(R)$ is finitely generated, however since the Koll\'ar component produced in Proposition \ref{p-qdlt} may not be in interior of the simplex containing $v$, we can not directly apply  Lemma \ref{l-finiteass}.  We need to show that
\begin{prop}\label{p-closekollar}
Let $v$ be a minimizer of $\hvol$ such that the associated graded ring $\gr_{v}(R)$ is finitely generated. Let $c_i\cdot \ord_{S_i}$ be chosen sufficiently close to $v$ as in Lemma \ref{l-finiteass}, then $S_i$ is a Koll\'ar component. 
\end{prop}
\begin{proof}By the proof of Theorem \ref{t-wlcmodel}, we already know that $S_i$ can extracted alone to get a model $\mu_i\colon Y_i\to X$ such that $(Y_i,\mu^{-1}_*(D)+S_i)$ is log canonical and $-S_i$ is $\mu_i$-ample. Therefore, we can degenerate $x\in (X,D)$ to $o\in (X_0:={\rm Spec}(\gr_{S_i}(R)),D_0)$, where $D_0$ is given by the adjunction. Recall that ${\rm Spec}(\gr_{S_i}(R) \cong {\rm Spec}(\gr_{v}(R))$ and $(X_0,D_0)$ is aways semi-log-canonical (slc), as so is $(S_i, D_{S_i})$ where $(K_{Y_i}+\mu_*^{-1}D+S_i)|_{S_i}:=K_{S_i}+D_{S_i}$.  Here we use the fact that $S_i$ is CM as $Y$ is potentially klt. By Lemma \ref{lem-irrid}, we know that $\hvol_{X,D}(v)=\hvol_{X_0,D_0}(v_0:=\xi_v)$. 

Now we assume $S_i$ is not a Koll\'ar component. By Lemma \ref{v-lccase}, we know that there exists a $T$-invariant valuation $w$ over $o\in (X_0,D_0)$ such that $\hvol_{X_0,D_0}(w)<\hvol_{X,D} (v_0)$. So by Theorem \ref{t-liu}, we know that there exists an $T$-equivariant primary ideal $I$ such that the normalized multiplicity $\mult(I)\cdot \lct^n(I)\le \vol(w)$. Since any valuation over $o$ has its log discrepancy to be positive with respect to $(X_0,D_0)$, by Lemma \ref{l-semilogc}, we can construct a model $Z\to (X_0,D_0)$ which precisely extracts a divisor $G$ calculating the log canonical threshold of $I$ with respect to $(X_0,D_0)$. This yields a $\mathbb{C}^*$-degeneration of $X_0$ to a model $Y$, on which we can define the normalized volume of $\eta$ as in Section \ref{ss-kollar}. Then 
$$\hvol_Y(\eta)=\hvol_{X_0,D_0}(\ord_G)\le \mult(I)\cdot \lct^n(I)\le \hvol(w)<\hvol(v_0)=\hvol_Y(\xi_0).$$
By the argument of Theorem \ref{t-Ksta}, we can construct a degeneration from $(X,D)$ to $(Y={\rm Spec}(\gr_{\ord_F}(\gr_{v}R)),E)$. Moreover,  the calculation as in \eqref{e-doubledeg} (see Remark \ref{r-general}) shows that this contradicts to the assumption that $v$ is a minimizer.
\end{proof}
\begin{lem}\label{v-lccase}Let $(S,D_S)$ be a slc pair with $-K_S-D_S$ ample.
Assume that $(S,D_S)$ is not klt. Let $o\in (Y,D_Y):=C(S,D_S;-r(K_S+D_S))$. Then  $\inf_{v} \hvol_o(v)=0$ where $v$ runs through all $\bC^*$-invariant valuations.  
\end{lem}
\begin{proof}
Since $(S,D_S)$ is not klt, there exists a divisor $E$ over $S$ such that $A_{(S,D_S)}(E)=0$. Consider a set of valuations $\{v_s; s\in[0,\infty)\}\subset \Val_{Y,o}$ defined in the following way. For any $f=\sum_k f_k\in \bigoplus_{k=0}^{\infty}H^0(S,kr(-K_S-D_S))$, $v_s(f)$ is given by:
\begin{equation}
v_s(f)=\min_k\{ k+ s\cdot \ord_E(f_k); f=\sum_k f_k \text{ with } f_k\neq 0\}. 
\end{equation}
Because $A_{(S,D_S)}(E)=0$, we have $A_{Y,D_Y}(v_s)\equiv r^{-1}$.
By using similar method as in \cite[section 4.1-4.2]{Li15b}, we get the formula for the normalized volume of $v_s$:
\begin{equation}
\hvol(v_s)=-r^{-n}\int_1^{+\infty}\frac{d\vol\left(R^{(\frac{t-1}{s})}\right)}{t^n}=-r^{-n}\int_0^{+\infty}\frac{d\vol\left(R^{(x)}\right)}{(1+sx)^n}.
\end{equation}
Since $-d\vol(R^{(x)})$ is a positive measure with finite total mass, it's clear that $\hvol(v_s)\rightarrow 0$ as $s\rightarrow+\infty$.
\end{proof}
\begin{lem}\label{l-semilogc}Let $o\in (X,D)$ be a slc singularity such that ${\rm mld}_{X,D}(o)>0$. Let $I$ be a primary ideal cosupported on $o$. Then we can find a model $\mu \colon Z\to X$ with extracts precisely a divisor $G$ computing the log canonical threshold of $I$ with respect to $(X,D)$ such that $-G$ is $\mu$-ample.
\end{lem}
\begin{proof}By \cite[10.56]{Kol13}, there exists a semi log resolutions $\phi\colon X'\to X$ of $(X,D+cI)$ with the properties there, where $c=\lct(I; X,D)$. By a tie-break argument, using the fact that $\phi_*(\mathcal{O}_{X'})=\mathcal{O}_X$, we can find a small $\mathbb{Q}$-divisor $H$ passing through $o$, such that there exists a unique divisor $G$ over $o$ with the log discrepancy $0$ with respect to $(X,D+(c-\epsilon)I+H)$. 

Choose a sufficiently close $\epsilon'<\epsilon$, then on $X'$ the only divisor with negative log discrepancy with respect to $(X,D+(c-\epsilon')I+H)$ is $G$. By \cite{OX12}, we can construct the semi-log-canonical modification $\mu \colon Z\to X$ of $(X,D+(c-\epsilon')I+H)$, which thus only extracts $G$ with $-G$ being $\mu$-ample. 
\end{proof}

\begin{thm}\label{t-kltdeg}
Let $X_0=\Spec ({\rm gr}_vR)$ and denote by $D_0$ the closure of $D$ on $X_0$. Then $(X_0,D_0)$ is klt.
\end{thm}
\begin{proof}
By Lemma \ref{l-finiteass}, under our assumption, we can choose a valuation divisorial $v'$ such that ${\rm gr}_{v}R$ is isomorphic to ${\rm gr}_{v'}(R)$. Then Proposition \ref{p-closekollar} implies that $v'$ indeed can be chosen to be a Koll\'ar component. Thus by the argument in  \cite{LX16}, we know $(X_0,D_0)$ is a klt pair, as it is an orbifold cone over a log Fano pair. 
\end{proof}

 \begin{defn}\label{d-ksemi}
 Under the above assumptions,  we say that a quasi-monomial valuation over $x\in X$ is K-semistable if ${\rm gr}_v(R)$ is finitely generated and the corresponding triple  $(X_0, D_0,\xi_v)$  as in Definition \ref{d-specialreeb} is K-semistable in the sense of Definition \ref{d-ksemiSE}.
 \end{defn}

\begin{proof}[Proof of Theorem \ref{t-KtoM}] By Theorem \ref{thm-semin}, we already know this for the log Fano cone case when the valuation on the singularity is given by a triple $(X,D,\wt_{\xi})$ as in Theorem  \ref{thm-semin}. In the general case, if $v$ induces  a special test configuration of $(X,D)$ to $ (X_0, D_0, \xi_v)$, then for any ideal $\fa\in {\rm PrId}_{X,x}$, we can get a graded ideal sequence $\fb_{\bullet}=\{\bin(\fa^k)\}$, satisfying
\begin{eqnarray*}
\mult(\fa)\cdot \lct_{(X,D)}^n(\fa)&=&\mult(\fb_{\bullet})\cdot \lct_{(X_0,D_0)}^n(\fb_{\bullet}) \ \ \mbox{(cf. \cite[3.3]{LX16})}\\
 &\ge &\hvol_{(X_0,D_0)}(\wt_{\xi_v})   \ \ \mbox{(Theorem \ref{thm-semin})}\\
 &=& \hvol_{(X,D)}(v)   \ \ \mbox{(Lemma  \ref{lem-irrid})}.
\end{eqnarray*}
\end{proof}

\begin{thm}\label{t-Ksta}
If $x\in (X,D)$ is a klt singularity and $v\in \Val_{X,x}$ which is a quasi-monomial minimizer of $\hvol_{(X,D)}$ such that its associated graded ring ${\rm gr}_v(R)$ is finitely generated, then $v$ is a K-semistable valuation. 
\end{thm}

The proof of this theorem is similar to the case of Koll\'ar component minimizer as in Section 6 of \cite{LX16}. For reader's convenience, we include a brief sketch here. 

\begin{proof}Using the notation, we can assume the quasi-monomial valuation $v=v_{\alpha}$ on a log smooth model with the weight $\alpha=(\alpha_1,...,\alpha_r)$ and $\alpha_1$,..., $\alpha_r$ are $\mathbb{Q}$-linearly independent. Let $\Phi$ and $\Phi^g$ denote the valuative semigroup and valuative group of $v$. Denote by $\cR$ the extended Rees algebra:
\begin{equation}
\cR=\mathcal{R}_v= \bigoplus_{\phi\in \Phi^g} \fa_{\phi}(v) t^{-\phi} \subset \mathcal{R}[t^{\Phi^g}]. 
\end{equation}
Then $\cR$ is faithfully flat over $\Spec(\bC[\Phi])$ (\cite[Proposition 2.3]{Tei03}). The central fiber $X_0$ is given by $\Spec(\gr_vR)$ where 
\[
\gr_v R=\bigoplus_{\phi\in \Phi^g}\left(\fa_{\phi}(v)/\fa_{>\phi}(v) \right).
\]
Let $\xi_0=\xi_v$ denote the induced valuation on the central fiber $X_0$ as in Definition \ref{d-specialreeb}. 

For any Koll\'ar component $S$ over $o'\in (X_0,D_0)$,
let $(\cY, \cD', \xi_0; \eta)$ be the associated special test configuration which degenerates $(X_0, D_0)$ to an orbifold cone $(Y_0, D'_0)$ over $S$, i.e., $S=Y_0/ \langle e^{\bC \xi_{S}}\rangle $. By Proposition \ref{prop-testKollar}, it suffices to show that 
 \[
 \Fut(\cY, \cD',\xi_0;\eta)\ge 0,
 \]
 for any such special test configurations.
 
Let $\Phi\subset \mathbb{R}_{\ge 0}$ be the valuative monoid of $v$. Then we have a $\Phi\times \mathbb{Z}_{\ge 0}$-valued function on $R$.
\begin{eqnarray*}
w: R& \longrightarrow & \Phi\times \mathbb{Z}_{\ge 0}\\
f & \mapsto &\Large( v(f), \ord_{S}(\rin(f)) \Large).
\end{eqnarray*}
We give $ \Gamma:=\Phi\times \mathbb{Z}_{\ge 0}\subset \mathbb{R}_{\ge 0}\times \mathbb{R}_{\ge 0}$ the lexicographic  order $(m_1, u_1)<(m_2, u_2)$ if and only if $m_1<m_2$, or $m_1=m_2$ and $u_1<u_2$. If we denote
\[
\gr_w R=\bigoplus_{(m, u)\in \Gamma} R_{\ge (m, u)}/ R_{> (m, u)},
\]
then $Y_0={\rm Spec}_{\bC} \left(\gr_w R\right)$. 

Also denote:
\[
A=\bigoplus_{m\in \Phi} R_{\ge m}/R_{>m}=:\bigoplus_{m\in  \Phi} A_m.
\]
Then ${\rm Spec} (A)=X_0$. Moreover if we define the extended Rees ring of $A$ with respect to the filtration associated to $\ord_S$:
\[
\mathcal{A}=\bigoplus_{k\in \bZ} \cA_k:=\bigoplus_{k\in \bZ} \mathfrak{b}_k t^{-k} \subset A[t, t^{-1}],
\]
where $\mathfrak{b}_k=\{f\in A; \ord_S(f)\ge k\}$. Then the flat family $\cY\rightarrow \bA^1$ is given by the ${\rm Spec}_{\bC[t]}\left(\mathcal{A}\right)$. In particular, we have
\[
\cA \otimes_{\bC[t]}\bC[t, t^{-1}]\cong A[t, t^{-1}], \quad \cA \otimes_{\bC[t]}\bC[t]/(t)\cong \gr_v R.
\]

\bigskip

Pick up a set of homogeneous generators $\bar{f}_1, \dots, \bar{f}_p$ for $\gr_w R$ with $\deg(\bar{f}_i)=(m_i, u_i)$. Lift them to generators $f_1, \dots, f_p$ for $A$ such that $f_i\in A_{m_i}$. 
Set $P=\mathbb{C}[x_1, \dots, x_p]$ and give $P$ the $\Gamma$-grading by $\deg(x_i)=(m_i, u_i)$ so that the surjective map 
$$P\rightarrow \gr_v R \qquad\mbox{  by }\qquad x_i\mapsto f_i$$ is a map of graded rings. 
Let $\bar{g}_1, \dots, \bar{g}_q\in P$ be a set of homogeneous generators of the kernel and set $\deg(\bar{g}_j)=(n_j, v_j)$. 

Since $\bar{g}_j(\bar{f}_1, \dots, \bar{f}_p)=0 \in {\rm gr}_wR$, it follows 
$$\bar{g}_j(f_1, \dots, f_p) \in (A_{n_j})_{>v_j} \qquad \mbox{ for each } j.$$ 
By
the flatness of $\cA$ over $\bC[t]$, there exists $g_j\in \bar{g}_j+(P_{n_j})_{> v_j}$ such that $g_j(f_1, \dots, f_p)=0$ for $1\le j\le q$. So $\{g_j\}$ form a Gr\"{o}bner basis of $J$ with respect to the order function $\ord_S$, where $J$ is kernel 
surjection $P\rightarrow A$. In other words, if we let $K=(\bar{g}_1, \dots, \bar{g}_q)$ denote the kernel $P\rightarrow \cA_0$. Then $K$ is the initial ideal of $J$ with respect to the order determined by 
$\ord_{S}$. As a consequence, we have:
\[
\cA=P[\tau]/(\tilde{g}_1, \dots, \tilde{g}_q),
\]
where $\tilde{g}_j=\tau^{v_j} g_j(\tau^{-u_1} x_1, \dots, \tau^{-u_p} x_p)$.

\bigskip

Now we lift $f_1, \dots, f_p$ to generators $F_1, \dots, F_p$ of $R$. Then we have: $g_j(F_1, \dots, F_p)$ lies in $R_{> n_j}$. By flatness of $\cR$ over $\bC[\Phi]$, there exist $G_j\in g_j+P_{> n_j}$ such that $G_j(F_1, \dots, F_p)=0$. Let $I$ be the kernel of $P\rightarrow R$, then $\{G_j\}$ form a Gr\"{o}bner basis with respect to the order function $v$ and the associated initial ideal is $J$. 


\bigskip

Given the above data, we know that there is a action of  $T:=(\bC^*)^{r+1}=(\mathbb{C}^*)^r\times \bC^*$-action on $\bC^p$. The valuation $v$ corresponds to a linear holomorphic vector field $\xi_0$ with an associated weight function denoted by $\lambda_0$. $\ord_S$ corresponds to another linear holomorphic vector field $\xi_S$ on $\bC^p$ whose associated weight function will be denoted by $\lambda_\infty$.

Notice that because $v$ is a real valuation, $\Phi^g$ is a subgroup of $\bR$. We denote by $C\subset \Phi^g\times\bZ\subset \bR\times\bZ$ the finite set consisting of the differences $(n'_j, v'_j)-(n_j, v_j)$. Let $M$ be a positive integer that is larger than all coordinates of $(m,u)-(n,v)$ for all pairs of elments $(m,u), (n,v)\in C$ and let $\epsilon$ be sufficiently small such that
$1>M\epsilon$. Denote by $e_0^*$ (resp. $e_1^*$) the projection of $\bR\times\bZ$ to $\bR$ (resp. $\bZ$) and define:
\[
\lambda_\epsilon=e_0^*-\epsilon e_1^*: \bR\times\bZ\rightarrow \bR.
\]
Then for $\epsilon$ sufficiently small, $\lambda_\epsilon: \bR\times\bZ\rightarrow \bR$ satisfies $0<\lambda_\epsilon(n_j, v_j)< \lambda_\epsilon(n'_j, v'_j)$. As a consequence, the linear holomorphic vector field $\xi_\epsilon\in \ft^+_\bR$ associated with $\lambda_\epsilon$ degenerates both $X$ and $X_0$ to $Y_0$. On the other hand, the weight function $\lambda_\epsilon$ determines a filtration on $P=\bC[x_1,\dots, x_p]$ which in turn induces quotient filtrations on $R$ and $\gr_v R$. Because the associated graded rings of the filtrations are both isomorphic to $\gr_wR$, by Lemma \ref{lem-quasi} $\lambda_\epsilon$ induces a quasi-monomial valuation over $X$ and a quasi-monomial valuation over $X_0$, both of which will be denoted by $w_\epsilon$.

Since $\lambda_\epsilon$ is linear with respect to $\epsilon$, $\{\xi_\epsilon\}$ is a ray in $\ft^+_\bR$ emanating from $\xi_0$. Denoting $\eta=-\frac{d}{d\epsilon}|_{\epsilon=0}\xi_\epsilon\in \ft_\bQ$, we then get a special test configuration $(\cX,\cD,\xi_0;\eta)$ of $(X,D)$ to $(Y_0,D'_0)$, and also a special test configuration of $(X_0,D_0)$ to $(Y_0,D'_0)$.

By Lemma \ref{lem-irrid}, we have the identities of normalized volumes:
$$\hvol_{(X,D)}(w_\epsilon)=\hvol_{(X_0,D_0)}(w_\epsilon)=\hvol_{(Y_0,D'_0)}(\wt_{\xi_\epsilon})=\hvol_{(Y_0,D'_0)}(\xi_\epsilon).$$

Now we use the minimizing assumption: $\hvol_{(X,D)}(w_\epsilon)\ge \hvol_{(X,D)}(v)$, which via the above identities gives $\hvol_{(Y_0,D'_0)}(\xi_\epsilon)\ge \hvol_{(Y_0,D'_0)}(\xi_0)$. 
So we get:
\begin{equation}\label{e-doubledeg}
\Fut(\cY,\cD',\xi_0;\eta)=D\hvol_{(Y_0,D_0)}\cdot (-\eta)=\left.\frac{d}{d\epsilon}\right|_{\epsilon=0}\hvol_{(X,D)}(w_\epsilon)\ge 0.
\end{equation}
Because $S$ is an arbitrary Koll\'{a}r component over $(X_0,D_0)$, we get $(X_0,D_0; \xi_0)$ is K-semistable. As discussed above, this implies $v$ is indeed K-semistable.
\end{proof}
\begin{rem}\label{r-general}
The above argument works as long as we can find an equivariant degeneration $Y_0$ of $X_0={\rm Spec}(\gr_v (R))$, such that we can define the normalized volume on the toric valuations for $\wt_{\xi}$ where $\xi\in \ft_{\mathbb{R}}(Y)$, since we still have the convexity of the normalized volumes in this setting (see Proposition \ref{prop-Tconvex}).
\end{rem}
\subsection{Uniqueness in general}\label{s-uniqueness}

In this section, we will verify the uniqueness of  quasi-monomial minimizer if we assume one of them has a finitely generated associated graded ring. This assumption is always fulfilled when the minimizer is divisorial \cite{Blu16, LX16} and is conjectured to hold in general. 

\bigskip

Assume a quasi-monomial valuation $v$ minimizes $\hvol_{(X,D)}$ with a finitely generated associated graded ring. Then $(X_0,D_0, \xi_v)$ is a K-semistable log Fano cone singularity by Theorem \ref{t-Ksta}. By Lemma  \ref{l-finiteass}, there is a divisorial valuation $v'$ which after scaling is sufficiently close to $v$ and satisfies that $${\rm gr}_v(R)\cong {\rm gr}_{v'}R .$$ 

$X_0={\rm Spec}({\rm gr}_{v'}R)$ is the central fiber of a special test configuration $\mathcal{X}\to \mathbb{A}^1$. 

\bigskip

Let $w$ be a quasi-monomial minimizer of $\hvol_{X,x}$ over $x\in X$. There exists a weak log canonical model $Y\to X$ constructed Theorem \ref{t-wlcmodel} with the exceptional divisor $\sum^r_{i=1}S_i$ on which $w$ is computed. 
Denote by $\{\fa_{\bullet}\}$ the valuative graded ideal sequence of $w$. As in \cite[Section 3]{LX16}, we can degenerate $\{\fa_{\bullet}\}$ to $\{\fb_{\bullet}\}:=\{\bin_v(\fa_{\bullet})\}$ to get a flat family of ideal sequences, and we have
\begin{eqnarray*}
\hvol_{(X_0,D_0)}(\xi_v)& =& \hvol_{(X,D)}(v) \\
  &=& \mult(\fa_{\bullet})\cdot \lct^n  (\fa_{\bullet})\\
    &\ge & \mult(\fb_{\bullet})\cdot \lct^n  (\fb_{\bullet}).
\end{eqnarray*}
Moreover, since $(X_0,D_0,\xi)$ is K-semistable, we know $\wt_{\xi}$ minimizes $\vol_{X_0,D_0}$, thus the last inequality is indeed an equality. As  $\mult(\fa_{\bullet})=\mult(\fb_{\bullet})$, we indeed have  $\lct(\fa_{\bullet})=\lct(\fb_{\bullet})$ and we denote it by $c$.

 \begin{lem}\label{l-degquasi}
There is a $\mathbb{Q}$-factorial equivariant  family $\tilde{\mu}\colon \mathcal{Y}\to \mathcal{X}$ over $ \mathbb{A}^1$, whose general fiber  gives $Y\to X$. Furthermore,
$(\mathcal{Y}, \mu_*^{-1}\mathcal{D}+ \sum \mathcal{S}_i+Y_0)$  is log canonical and $Y_0$ is irreducible.
 \end{lem} 
 \begin{proof}
By Lemma \ref{l-extractdivisor}, for any $\epsilon$,  we can choose $S_j$ and $\epsilon_0$ such that $a(S_j, X, D+(1-\epsilon_0 )c\cdot \frac{1}{m} \fa_m)<\epsilon$ for any sufficiently large $m$. For $\fa_m$, we let $\overline{\fa}_m$ be the family on $\mathcal{X}$, which degenerates $\fa_m$ to its initial ideal.  Then $S_i$ in Theorem \ref{t-wlcmodel} will induce divisors $\mathcal{S}_i$ over $\mathcal{X}$ such that 
 $$a(S_j, X, D+(1-\epsilon_0 )c\cdot \frac{1}{m} \fa_m)=a(\mathcal{S}_j, \mathcal{X}, \mathcal{D}+(1-\epsilon_0 )c\cdot \frac{1}{m} \overline{\fa}_m).$$

It then follows from the standard approximation  that for any $\epsilon_0>0$ we can assume $m$ sufficiently large such that $(X_0, D_0+(1-\epsilon_0 )c\cdot \frac{1}{m} \bin(\fa_m))$ is log canonical. By inversion of adjunction, this implies that  
$$(\mathcal{X},\mathcal{D}+X_0+(1-\epsilon_0 )c\cdot \frac{1}{m} \overline{\fa}_m)$$ is log canonical. 
In particular, we know that we can find $\mathcal{Y}\to \mathcal{X}$ which precisely extracts the divisors $\mathcal{S}_i$ such that a general fiber yields $Y$.

It remains to show that $(\mathcal{Y}, \mu_*^{-1}\mathcal{D}+ \sum \mathcal{S}_i+Y_0)$  is log canonical. Again by ACC of log canonical thresholds (\cite{HMX14}), it suffices to show that for the constant $\beta$ chosen in Theorem \ref{t-wlcmodel}, $(\mathcal{Y}, \mu_*^{-1}\mathcal{D}+ \beta\cdot \sum \mathcal{S}_i+Y_0)$ is log canonical. But this is implied by the fact that the log pull back of 
 $K_\mathcal{Y}+\mu_*^{-1}\mathcal{D}+\beta \sum \mathcal{S}_i$ is less or equal to the log pull back of $K_\mathcal{X}+\mathcal{D}+(1-\epsilon_0 )c\cdot  \frac{1}{m} \overline{ \fa}_m$.
 \end{proof}

Then we can define a  quasi-monomial valuation $w_0$ as in Definition-Proposition \ref{dp-deg} over $X_0$.

\begin{def-prop}\label{dp-deg}Let $\eta_0$ be the generic point of a component of $\overline{\eta}\cap Y_0$. 
Let $T_i$ be the (not necessarily irreducible) reduction  divisor of $\mathcal{S}_i$ at the generic point $\eta_0$.  Then $\ord_{T_1}$, $\ord_{T_2}$,..., and $\ord_{T_r}$  generate a rank $r$ sublattice in $\Val_{Y_0,\eta_0}$. 
Furthermore, we can define a quasi-monomial valuation $v_0$ over $\eta_0$ which is of rational rank $r$, such that $v_0(T_i)=\alpha_i$.
\end{def-prop}
\begin{proof}By \cite[Proposition 34]{dFKX17} we know that  $(\mathcal{Y}, Y_0+\sum \mathcal{S}_i+\tilde{\mu}_*^{-1}\mathcal{D})$ is q-dlt at the generic point of the log canoincal center given by an irreducible component $Z_0$ of $Y_0\cap \mathcal{Z}$, so we can define such a quasi-monomial valuation $v_0$ over $\eta$. 
\end{proof}
The following lemma implies that such a degeneration is indeed uniquely determined.

\begin{lem}\label{l-valdeg}
Let $w_0$ be a degeneration of a quasi-monomial minimizer $w$. 
Then for any function $f\in R$, we have 
$$w(f)=w_0(\bin(f)) 
. 
$$
\end{lem}
\begin{proof}
We easily see $w(f)\le w_0(\bin(f))$ and now we assume $w(f)< w_0(\bin(f))$ from some $f$ and we will argue this is contradictory to the fact that $\vol(w)=\vol(w_0)$ as  in the proof of \cite[Proposition 2.3]{LX16}.

We prove it by contradiction. Assume this is not true, we fix $g\in R$ such that 
$$w_0(\bin(g))=l>w(g)=s.$$
Denote by $r= l-s>0.$ Fix $k\in \mathbb{R}_{>0}$. Consider 
$$\fa_k:=\{f_0\in {\rm gr}_vR | \ w_0(f_0)\ge k\}\qquad\mbox{ and} \qquad\fb_k:=\{f\in R |\  w(f)\ge k\}. $$
So $\bin(\fb_k)\subset \fa_k$, and we want to estimate the dimension of 
\begin{eqnarray*}
\dim(R/{\fb_k})-\dim({\rm gr}_vR/{\fa_k})&= & \dim({\rm gr}_vR/{\bin(\fb_k}))-\dim({\rm gr}_vR/{\fa_k})\\
 &=&\dim (\fa_k/\bin(\fb_k)).
\end{eqnarray*}
Fix a positive integer $m<\frac{k}{l}$ and elements
$$g^{(1)}_{m},...,g^{(k_m)}_{m} \in \fb_{k-ml}$$
whose images in $\fb_{k-ml}/\fb_{k-ml+r}$ form a $\mathbb{C}$-linear basis.

We claim that 
$$\{\bin(f^m\cdot  g^{(j)}_{m})\}\ \ (1\le m \le \frac{k}{l}, 1\le j\le k_m)$$ 
are $\mathbb{C}$-linear independent in $\fa_k/\fb_k$. Granted this for now, 
we know since $\vol(v)>0$, then 
$$\limsup_{\lim k\to \infty} \frac{1}{k^n}\sum_{1\le m \le \frac{k}{l}}k_m=\limsup_{\lim k\to \infty} \sum_{1\le m\le \frac{k}{l}} \frac{1}{k^n} \dim (\fb_{k-ml}/\fb_{k-ml+r})>0, $$
which then implies $\vol(w')>\vol(w'_0)$ and yields a contradiction.

\bigskip

Now we prove the claim.

\noindent{\bf Step 1:} For any $1\le m \le \frac{k}{l}, 1\le j\le k_m$,
\begin{eqnarray*}
w_0(\bin(f^m\cdot g^{(j)}_{m}))&=&w_0(\bin(f^m))+w_0(\bin(g^{(j)}_{m}))\\
 &\ge& ml+w(g^{(j)}_{m})\\
 &\ge & ml+k-ml\\
 &\ge &k.
 \end{eqnarray*}
 Thus $\bin(f^m\cdot g^{(j)}_{m})\in \fa_k$. 
 
 \vspace{3mm}
 
 \noindent{\bf Step 2:} Since taking the initial induces an isomorphism of $\mathbb{C}$-linear spaces between $R/\fb_k$ and ${\rm gr}_vR/\bin(\fb_k)$ (\cite[Lemma 4.1]{LX16}), to show $\bin(f^m\cdot  g^{(j)}_{m})$ is linearly independent, it suffices to show that $f^m\cdot  g^{(j)}_{m}$ is linearly independent in $R/\fb_k$. This is verbatim  the same as Step 2 in the proof of \cite[Proposition 2.3]{LX16} once we replace $v$ by $w$. 
 \end{proof}

\begin{prop}\label{p-Tequiv}
Let $x\in X$ be a $T$-singularity. Assume a minimizer $v$ of $\hvol_{X,x}$ is quasi-monomial, then $v$ is $T$-invariant. 
\end{prop}
\begin{proof}
It suffices to prove this for $T=\mathbb{C^*}$, since if a valuation $v$ is $\mathbb{C}^*$-equivariant for any $\mathbb{C}^*$, then it is $T$-invariant. 

We have seen that the degeneration sequence $\{\fb_{\bullet}\}:= \{\bin (\fa_{\bullet})\}$ has the same log canonical threshold $c$ as $v$. 
We fix $\epsilon$, and assume that for $ m\in \Phi$ sufficiently large, 
$$\left(X,D, (c-\epsilon)\frac{1}{m}\bin (\fa_m)\right)$$ 
is klt. 

As in the construction of $S_i$, 
and the proof of Theorem \ref{t-wlcmodel} (see Lemma \ref{l-extractdivisor}), we can indeed assume that
$$a_l(S_i; X, D+(c-\epsilon)\frac{1}{m}\fa_m)<1.$$

We are going to show that $S_i$ is equivariant, which then clearly implies $v$ is equivariant.  
We consider the family of ideals $\fa_{m, \mathbb{A}^1}$ which $\mathbb{C}^*$-equivariantly  degenerates $\fa_m$ to $\fb_m$. Since  $\fa_{m, \mathbb{A}^1}$ is $\mathbb{C}^*$-equivariant,  we can construct a $\mathbb{C}^*$-equivariant model $\mathcal{Y}\to \mathcal{X}$, which extracts exactly the closure $\mathcal{S}_i$ of $S_i\times {\mathbb{C}^*}$ as the log discrepancy
$$a_l(S_i\times {\mathbb{C}^*}, X\times \mathbb{A}^1,  D\times \mathbb{A}^1+(c-\epsilon)\frac{1}{m}\fa_{m, \mathbb{A}^1})<1,$$
and $(X\times \mathbb{A}^1,  D\times \mathbb{A}^1+(c-\epsilon)\frac{1}{m}\fa_{m, \mathbb{A}^1})$ is klt. Furthermore, we can assume $\mathcal{S}_i$ is anti-ample over $X\times \mathbb{A}^1$.

Thus $\mathcal{Y}$ is indeed $Y\times \mathbb{A}^1$ where $Y=\mathcal{Y}\times_{\mathbb{A}^1}\{t\}$ for some $t\neq 0$ as they are isomorphic in codimension 1, and both are the anti-ample model of the same divisorial valuation over $X\times \mathbb{A}^1$.  This implies that $S_i$ is equivariant.  
\end{proof}

Now we can complete the proof of Theorem \ref{t-unique}.

\begin{proof}[Proof of Theorem \ref{t-unique}]
For the first part, it is  Theorem \ref{t-kltdeg}.

\medskip

For the second part,  it is  Theorem \ref{t-Ksta}.

\medskip

For the last part, using Proposition \ref{p-Tequiv} and Theorem \ref{thm-uniquefano}, we know that $w_0$ is the same as $\xi_v$ on $(X_0,D_0)$. Then for any $g\in R$, by Lemma \ref{l-valdeg} we have
$$w(g)=w_0(\bin (g))=v_0(\bin (g))=v(g),$$
thus $w=v$.
\end{proof}

\part{Singularities on GH limits}\label{p-dg}

\section{Canonicity of the semistable cone }\label{sec-DS}

\subsection{Metric tangent cones and valuations}\label{ss-DS}

Let $(M_i, \omega_i)$ be a sequence of smooth K\"{a}hler-Einstein Fano manifolds. By Gromov's compactness in Riemannian geometry, it is known that a subsequence converges to a limit metric space in the Gromov-Hausdorff topology:
$$
\lim_{j\rightarrow+\infty} (M_{i_j}, \omega_{i_j})=(M_\infty, d_\infty).
$$
By the work of Donaldson-Sun (\cite{DS14}) and Tian (\cite{Tia90}, \cite[4.14-4.16]{Tia12}, see also \cite{Tia13}), we know that $M_\infty$ is homeomorphic to a normal algebraic variety. Donaldson-Sun (\cite{DS14}) also showed that $M_\infty$ has at worst normal klt singularities and admits 
a K\"{a}hler-Einstein metric $\omega_\infty$ in the sense of pluripotential theory (see also \cite{BBEGZ}). On the regular locus $M^{\rm reg}_\infty$, $\omega_\infty$ is a smooth K\"{a}hler-Einstein metric. To understand the metric behavior
near the singular locus, it is important to understand the metric tangent cones of $(M_\infty, d_\infty)$ which is a metric cone by\cite{CC97}. From now on, fix $o\in M_\infty$ and denote by $C:=C_o M_{\infty}$ a metric tangent cone at $o\in M_\infty$. In other words, there exists a sequence of positive numbers $\{r_k\}_{k\in \bN}$ converging to $0$ such that 
\[
(C, d_C, o)=\lim_{r_k\rightarrow 0} \left(M_\infty, \frac{d_\infty}{r_k}, o\right),
\]
where the convergence is the pointed Gromov-Hausdorff topology. We know that $C$ admits a Ricci-flat K\"{a}hler cone structure by the work of Cheeger-Colding-Tian (\cite{CCT02}). 
More recently, Donaldson-Sun proved in \cite{DS15} that $C$ is an affine variety with a torus action and can be obtained in three steps. In the first step, they defined a filtration $\{\cF^\lambda\}_{\lambda\in \mathcal{S}}$ of the local ring $R=\cO_{M_{\infty},o}$ using the limiting metric structure $d_\infty$. Here $\mathcal{S}$ is a set of positive numbers that they called the holomorphic spectrum which depends on the torus action on the metric tangent cone $C$. In the second step, they proved that the associated graded ring of $\{\mathcal{F}^\lambda\}$ is finitely generated and hence defines an affine variety, denoted by $W$. In the last step, they showed that $W$ equivariantly degenerates to $C$. Notice that this process depends crucially on the limiting metric $d_\infty$ on $M_\infty$. On the other hand, they made the following conjecture. 
\begin{conj}[Donaldson-Sun]\label{conj-DS}
Both $W$ and $C$ depend only on the algebraic structure of $M_\infty$ near $o$.
\end{conj}
One goal of the project proposed in \cite{Li15a} is to prove this conjecture. We observed in \cite{Li15a, LX16} that $\{\cF^{\lambda}\}$ comes from a valuation $v_0$. This is due to the fact that $W$ is a normal variety since it degenerates to the normal variety $C$. As mentioned in \cite{HS16}, this was implicit in \cite{DS15} and \cite{Don16}. More explicitly, if we denote by $X={\rm Spec}(R)$ the germ of $o\in M_{\infty}$, by the work in \cite{DS15}, one can embed both $X$ and $C$ into a common ambient space $\bC^N$, and $v_0$ on $X$ is induced by the monomial valuation $\wt_{\xi_0}$ where $\xi_0$ is the linear holomorphic vector field with $2{\rm Im}(\xi_0)$ being the Reeb vector field of the Ricci flat K\"{a}hler cone metric on $C$. By this construction, it is clear that the induced valuation by $v_0$ on $W$ is nothing but $\wt_{\xi_0}$.  

\medskip

Here in this paper we also observe that $v_0$ is a quasi-monomial valuation. As pointed out in Lemma \ref{lem-quasi}, this follows from a general fact due to Piltant (\cite{Pil94}, see \cite[Proposition 3.1]{Tei03}) that a valuation $v_0$ is quasi-monomial if and only if the associated graded ring has the same Krull dimension as $\dim X$. See also Lemma \ref{lem-Tqmv} where the quasi-monomial property of $\wt_{\xi_0}$ on $W$ and $C$ is explained.

More importantly we conjectured in \cite{Li15a} that $v_0$ can be characterized as the unique minimizer of $\hvol_{M_\infty, o}$. For now we can not prove this conjecture in the full generality. Nevertheless, as a corollary of the theory developed in this paper (and its predecessors \cite{Li15a, Li15b, LL16, LX16}), we can already prove Theorem \ref{t-finiteunique} and confirm \cite[Conjecture 3.22]{DS15} for $W$. 
\begin{proof}[Proof of Theorem \ref{t-finiteunique}]
By the above discussion, for any valuation $v_0$ as above, we already know that it is quasi-monomial and centered at $o\in M_\infty$ and the induced valuation on $W={\rm Spec}({\gr_{v_0}}R)$ is equal to $\wt_{\xi_0}$. By Theorem \ref{thm-irSE}, we know that $(C, \xi_0)$ is K-semistable. 

We claim that $(W, \xi_0)$ is a klt Fano cone singularity and is K-semistable. To see this, we first note that there exists a prime divisorial valuation $S$ which has a finitely generated associated graded ring and degenerates $X$ to $W$. As in the proof of Theorem \ref{t-Ksta}, such an $S$ can be obtained by perturbing $\xi_0$ in the Reeb cone of $C$ so that we can ensure that this perturbed vector generates a $\bC^*$ in the big torus that preserves the klt Fano cone singularity $C$. As a consequence, $S$ is isomorphic to the quotient of $C$ by the $\bC^*$-action. This implies that $S$ is a Koll\'{a}r component over $C$. By inversion of adjunction, we also conclude that $W$ has klt singularities. Now since $(W, \xi_0)$ equivariantly degenerates to $(C, \xi_0)$, by Lemma \ref{l-degKsemi}, we know that $(W, \xi_0)$ is indeed K-semistable (see also Remark \ref{rem-semiW}).

By Theorem \ref{t-KtoM}, $v_0$ is a minimizer of $\hvol_{M_\infty,o}$ and by Theorem \ref{t-unique}.(3), $v_0$ is the unique minimizer of 
$\hvol_{M_\infty,o}$ among all quasi-monomial valuations in $\Val_{M_{\infty}, o}$, and it only depends on the algebraic structure of $R$. 
Therefore $W$ only depends on the algebraic structure of the germ $o\in M_{\infty}$.
\end{proof}

\begin{rem}\label{rem-semiW}
There is an alternative but essentially equivalent way to show directly that $v_0$ is a minimizer of $\hvol_{M_\infty, o}$. Firstly, in the proof of Theorem \ref{t-Ksta}, we have constructed a degeneration of $(M_\infty, o)$ to $(C, o)$. On the other hand, we know that $\hvol_{M_\infty,o}(v_0)=\hvol_{C,o}(\wt_{\xi_0})$ and $\wt_{\xi_0}$ minimizes $\hvol_{C,o}$. 
We can then use the same ideal-degeneration argument as in the proof of Lemma \ref{l-degKsemi} to conclude that $v_0$ is the minimizer of $\hvol_{M_\infty,o}$. Theorem \ref{t-unique} then also implies $(W, \xi_0)$ is a klt
Fano cone singularity which is indeed K-semistable by Lemma \ref{l-degKsemi}.

\end{rem}

\subsection{Minimizers from Ricci flat K\"{a}hler cone metrics}\label{s-tangentconeksemistable}

Let $(X, \xi_0)$ be a Fano cone singularity. Recall that this implies that $X$ is a normal affine variety with at worst klt singularities. Moreover there is a good $T$ action where $T\cong (\bC^*)^r$ and $\xi_0\in \ft_{\bR}^+$. On $X$ there exists a $T$-equivariant nowhere-vanishing holomorphic $m$-pluricanonical form $s\in |-mK_X|$. Such holomorphic form can be solved uniquely up to a constant as in \cite[2.7]{MSY08}. In the following, we will use the following volume form on $X$:
\begin{equation}
dV=\left(\sqrt{-1}^{mn^2}s\wedge \bar{s}\right)^{1/m}.
\end{equation}

We can assume that $(X, \xi_0)$ is equivariantly embedded into $(\bC^N, \xi_0)$ with $\xi_0=\sum_i a_i z_i\frac{\partial}{\partial z_i}$ with $a_i\in \bR_{>0}$. Choose any reference K\"{a}hler cone metric on $\bC^N$, which is possibly degenerate.\footnote{The following calculations and arguments do not depend on the choice of reference metrics, i.e. remain valid for any choice of reference metric.} For example, we can choose $\delta>0$ such that $\delta a_i<1$, for all $i$, and define a radius function:
\begin{equation}\label{eq-radius}
r^2=\left(\sum_{i=1}^N |z_i|^{\frac{2}{\delta a_i}}\right)^{\delta}.
\end{equation}
Then $r^2$ is a $C^{\lfloor 2/(\delta a_i)\rfloor}$-function on $\bC^N$ and has no critical points on $\bC^N\setminus \{0\}$. The corresponding K\"{a}hler cone metric on $\bC^N$ is equal to:
\begin{eqnarray*}
\omega_{\bC^N}
&=&\sddb r^2\\
&=&\delta(\delta-1) r^{(\delta-2)/\delta}\left(\sum_i \frac{1}{\delta a_i}|z_i|^{2(\frac{1}{\delta a_i}-1)}\bar{z}_idz_i\right)\wedge \left(\sum_j \frac{1}{\delta a_j}|z_j|^{2(\frac{1}{\delta a_j}-1)}\bar{z}_j dz_j\right)+\\
&&\hskip 5mm \delta r^{(\delta-1)/\delta}\left(\sum_i \frac{1}{\delta^2 a_i^2} |z_i|^{2\left(\frac{1}{\delta a_i}-1\right)}dz_i\wedge d\bar{z}_i\right)\ge 0.
\end{eqnarray*}
The restriction $\omega_X:=\omega_{\bC^N}|_{X}$ is a K\"{a}hler cone metric on $X$.
Moreover $2{\rm Im}(\xi_0)=J(r\partial_r)$ is the Reeb vector field of $\omega_{\bC^N}$ and $\omega_X$. For later purpose, we record following identities which can be verified directly:
\begin{equation}\label{eq-xi0r}
\xi_0(r^2)=\bar{\xi}_0(r^2)=r^2, \quad 2{\rm Re}(\xi_0)=r\partial_r;
\end{equation}
\begin{equation}\label{eq-xi0u}
\frac{n \sqrt{-1} 2(\partial v) \wedge (\bar{\partial} r) \wedge (\sddb r^2)^{n-1}}{(\sddb r^2)^n}=\xi_0(v)\cdot \frac{1}{r}.
\end{equation}

Since $T$ acts on $X$, $T$ also acts on the set of functions on $X$ by $\tau\circ f(x)=f(\tau^{-1}x)$ for any $\tau\in T$ and $x\in X$. For convenience, we introduce the following 
\begin{defn}
Denote by $PSH(X,\xi_0)$ the set of bounded real functions $\vphi$ on $X$ that satisfies:
\begin{enumerate}
\item $\tau\circ \vphi=\vphi$ for any $\tau\in T$;
\item $r^2_\vphi:=r^2 e^{\vphi}$ is a proper strictly plurisubharmonic function on $X$. 
\end{enumerate}
\end{defn}
\begin{defn}\label{defn-RFKC}
We say that $r^2_\vphi:=r^2 e^{\vphi}$ where $\vphi\in PSH(X, \xi_0)$ is the radius function of a Ricci-flat K\"{a}hler cone metric on $(X, \xi_0)$ if $\vphi$ is smooth on $X^{\rm reg}$ and there exists a positive constant $C>0$ such that
\begin{equation}\label{eq-RFKC}
(\sddb r^2_\vphi)^n=C\cdot dV.
\end{equation}
\end{defn}
Compared with the weak K\"{a}hler-Einstein case, it is expected that the regularity condition in the above definition is automatically satisfied. With this regularity assumption, on the regular part $X^{\rm reg}$, both sides of \eqref{eq-RFKC} are smooth volume forms and we have $r_{\vphi}\partial_{r_\vphi}=2 {\rm Re}(\xi_0)$ or, equivalently, $\xi_0=r_\vphi\partial_{r_\vphi}-i J(r_\vphi\partial_{r_\vphi})$.
Moreover, taking $\mathcal{L}_{r_\vphi \partial_{r_\vphi}}$ on both sides gives us the identity $\mathcal{L}_{r_\vphi\partial_{r_\vphi}}dV=2n\; dV$. 
Equivalently we have:
\[
\cL_{\xi_0}s= m n\cdot s,
\]
where $s\in |-mK_X|$ is the chosen $T$-equivariant non-vanishing holomorphic section. 
By Lemma \ref{lem-ldwt}, this implies $A_{X}(\wt_{\xi_0})=n$ (see \cite{HS16, LL16} for this identity in the quasi-regular case). The main goal of this section is to give a proof of the following fact.

\begin{thm}\label{thm-irSE}
If $(X, \xi_0)$ admits a Ricci-flat K\"{a}hler cone metric, then $A_X(\xi_0)=n$ and $(X, \xi_0)$ is K-semistable.
\end{thm}

\begin{rem}\label{r-flatcone}
\begin{enumerate}
\item
In the case when $X$ has isolated singularities at $o\in X$, this was proved in \cite{CS12} using an approximation by rational elements in $\ft_\bR^+$ to reduce to the orbifold case studied in \cite{RT11}. The proof given below for the general case is different and is a direct generalization of a corresponding proof in the usual K\"{a}hler case. We also depend heavily the calculations from \cite{MSY08} which have also appeared in different forms in \cite{CS15, DS15}.
\item 
As already mentioned, after Berman's work \cite{Ber15}, it is natural to expect that $(X, \xi_0)$ should actually be K-polystable. Since this requires more technical arguments involving geodesic rays and we do not need this stronger conclusion in this paper, we will leave its verification in \cite{LWX17}. 
\end{enumerate}
\end{rem}

As an immediate corollary, we can also verify that the volume density is equal to the normalized volume at any point on the Gromov-Hausdorff limit. In \cite{HS16} Hein-Sun pointed out the relationship between these two quantities. 
Here the volume density of the limit metric space $M_\infty$ at $o$ is defined to be:
\[
\Theta(M_\infty, o)=\lim_{r\rightarrow 0+}\frac{\vol(B(o,r))}{\omega_{2n} r^{2n}},
\]
where $\omega_{2n}=\frac{\pi^n}{n!}$ denotes the volume of the unit ball in the flat $\bC^n$.

\begin{cor}\label{c-v=v}
Let $o\in M_\infty$ be a closed point. Then we have the identity:
\begin{equation}
n^n \cdot \Theta(M_\infty, o)=\hvol(M_\infty, o).
\end{equation}
\end{cor}
\begin{proof}
If $(C,d_C)$ denotes the metric tangent cone of $(M_{\infty}, d_\infty)$ at $o$, then by standard metric geometry it is known that 
\begin{equation}\label{eq}
\Theta(M_\infty,o)=\Theta(C,o)=\frac{\vol(B_{d_C}(o,1))}{\pi^n/n!}.
\end{equation} 
Let $\xi_0$ be the holomorphic vector field such that $2{\rm Im}(\xi_0)$ is the Reeb vector field of the Ricci-flat K\"{a}hler cone metric $\sddb r_\vphi^2$. Then we claim that 
\begin{equation}\label{eq-hvol2Theta}
\hvol(\xi_0)=n^n\cdot \Theta(C,o).
\end{equation}
As explained in \cite[Appendix C]{HS16} (see also \cite{CS12, DS15}) we know that
\begin{eqnarray}\label{eq-vol2Theta}
\vol_C(\xi_0)=\Theta(C,o).
\end{eqnarray} 
By Theorem \ref{thm-irSE}, 
$A_C(\xi_0)=n$. So we have $n^n\Theta(C,o)=\hvol(\xi_0)=\hvol(M_\infty, o)$, where the second identity is by Theorem \ref{thm-irSE} and Theorem \ref{t-KtoM}.
\end{proof}

\begin{rem}
Both sides of \eqref{eq-vol2Theta} is equal to:
\begin{equation}
\frac{1}{n!(2\pi)^n}\int_C e^{-r^2_\vphi}(\sddb r^2_\vphi)^n=\frac{n!}{\pi^n}\vol\left(\{x\in C; r_\vphi \le 1\}\right),
\end{equation}
where $r_\vphi$ is the radius function for the K\"{a}hler cone metric on $C$.
\end{rem}

The rest of this section is devoted to the proof of Theorem \ref{thm-irSE}. Since $A_X(\xi_0)=n$ has been shown, we will focus on the second statement. 

Now assume that $(\cX, \xi_0;\eta)$ is a special test configuration of $X$ that is induced by a Koll\'{a}r component. Because $\eta$ commutes with $\xi_0$ and generates a $\bC^*$-action, we can assume that $\eta=\sum_i b_i z_i\frac{\partial}{\partial z_i}$ with $b_i\in \bZ$.
Let $\sigma(t): \bC^*\rightarrow Aut(\bC^N)$ be the one-parameter subgroup generated by the vector field $\eta$. 
Then $\sigma(t)(z_i)=t^{b_i} z_i$ and thus for the choice of radius function in \eqref{eq-radius}
\[
r(t)^2:=\sigma(t)^*(r^2)=r^2 \left(\frac{\sum_i |t|^{2 b_i/(\delta a_i)}|z_i|^{2/a_i}}{\sum_i |z_i|^{2/(\delta a_i)}}\right)^{\delta}= r^2 e^{\varphi(t)},
\]
where the function $\vphi(t)$ is given by:
\[
\vphi(t)=\delta\left[\log \left(\sum_i |t|^{2b_i/(\delta a_i)}|z_i|^{2/(\delta a_i)}\right)- \log \left(\sum_i |z_i|^{2/(\delta a_i)}\right)\right].
\]
Notice that $\vphi(t)$ satisfies the condition $\xi_0(\vphi)=\bar{\xi}_0(\vphi)=0$, which corresponds to the fact that $\vphi(t)$ descends to become a basic function on the link of $(X,x)$.

Following \cite{CS15}, we consider the following cone version of Ding energy:
\begin{defn}
For any function $\vphi\in PSH(X, \xi_0)$, we define:
\begin{equation}
D(\varphi)=E(\varphi)-\log\left(\int_X e^{-r_\varphi^2} dV\right)=: E(\vphi)-G(\vphi),
\end{equation}
where $E(\vphi)$ is defined by its variations:
\begin{equation}\label{eq-defMAE}
\delta E(\vphi)\cdot \delta\vphi= -\frac{1}{(n-1)!(2\pi)^n \vol_{X}(\xi_0)}\int_X (\delta\vphi) e^{-r_\vphi^2} (\sddb r_\vphi^2)^n.
\end{equation}
\end{defn}
By \cite{CS15}, $E(\vphi)$ is a well-defined function of $\vphi$. Moreover by calculating in the polar coordinate with respect to $r^2_\vphi$, one easily sees that $D(\vphi+c)=D(\vphi)$. The Euler-Lagrangian equation of $D(\vphi)$ is the equation of Ricci-flat K\"{a}hler cone metric in \eqref{eq-RFKC}.
 The following lemma is a generalization of a well known fact in the regular case. 
\begin{lem}
$E(\vphi(t))$ is a concave function with respect to $s=-\log|t|^2 \in \bR$.
\end{lem}

\begin{proof}
 We want to show that $\frac{d^2}{ds^2} E(\vphi)\le 0$ for any $s\in \bR$. By a change of variable, it's clear that we just need to show this when $|t|=1$ or equivalently when $s=0$. Denoting $\dot{\vphi}=\frac{\partial}{\partial s}\vphi=\frac{\partial}{\partial (-\log|t|^2)}\vphi(t)$, we calculate the second order derivative of $E(\vphi)$ with respect to $s\in \bR$. For the simplicity of notation, we denote $C(n, \xi_0)=(n-1)!(2\pi)^n\vol(\xi_0)$.
\begin{eqnarray}\label{eq-d2E1}
C(n, \xi_0)\left.\frac{d^2}{ds^2} E(\vphi)\right|_{s=0}&=&\left.-\frac{\partial}{\partial s}\int_X \dot{\vphi} e^{-r_\vphi^2}(\sddb r^2_\vphi)^n\right|_{s=0}\nonumber\\
&=&-\int_X \left(\ddot{\vphi}e^{-r^2}-\dot{\vphi}^2 r^2 \right) e^{-r^2} (\sddb r^2)^n\nonumber\\
&&\hskip 1.5cm -\int_X \dot{\vphi} e^{-r^2} n \sddb (r^2 \dot{\vphi})\wedge (\sddb r^2)^{n-1} 
\end{eqnarray}
To simplify the result, we use the identity $\partial_r(\dot{\vphi})=0$ to calculate:
\begin{eqnarray}\label{eq-intbypart}
&&-\int_X \dot{\vphi}e^{-r^2} n \sddb (r^2\dot{\vphi}) (\sddb r^2)^{n-1}\nonumber\\
&=&\int_X n \sqrt{-1}    \left[\partial \dot{\vphi}\wedge \bar{\partial} (r^2\dot{\vphi})-\dot{\vphi} \partial (r^2)\wedge \bar{\partial} (r^2 \dot{\vphi}) \right] e^{-r^2}(\sddb r^2)^{n-1}\nonumber \\
&=&\int_X n \sqrt{-1} \left[r^2 \partial \dot{\vphi}\wedge \bar{\partial} \dot{\vphi}-\dot{\vphi}^2 (\partial r^2)\wedge (\bar{\partial} r^2) \right]\wedge  e^{-r^2} (\sddb r^2)^{n-1}\nonumber \\
&=&\int_X \left(r^2 |\partial \dot{\vphi}|_{\omega_X}^2-r^2 \dot{\vphi}^2\right)e^{-r^2}(\sddb r^2)^n.
\end{eqnarray}
The integration by parts in the first identity is valid, because we have the following estimate, which can be derived from the invariance of $\dot{\vphi}$ under $\partial_r$:
\[
|\partial \dot{\vphi}|_{\omega_X}\le |\partial \dot{\vphi}|_{\omega_{\bC^N}}\le  \frac{C}{r}.
\]
Substituting \eqref{eq-intbypart} into \eqref{eq-d2E1}, we get:
\begin{eqnarray}\label{eq-d2Evphi}
\left.\frac{d^2}{ds^2}E(\vphi)\right|_{s=0}&=&-C(n,\xi_0)^{-1}\int_X (\ddot{\vphi}-r^2 |\partial\dot{\vphi}|_{\omega_X}^2) e^{-r^2} (\sddb r^2)^n.
\end{eqnarray}
To see the negativity of $\frac{d^2 E}{ds^2}$, we can define a (1,1)-form on $X\times \bC$ with respect to the variable $(z, s)$ by the formula:
\begin{eqnarray*}
\Omega&=&\sddb (r(t)^2)=\sddb \left( r^2e^\vphi \right)\\
&=&\sqrt{-1} \left(\partial_X\bar{\partial}_X r(t)^2+r(t)^2 (\dot{\vphi}^2+\ddot{\vphi})dt\wedge d\bar{t}\right)\\
&&\hskip 5mm +dt\wedge \left((\bar{\partial}_X r(t)^2)\dot{\vphi}+r(t)^2 \bar{\partial}_X\dot{\vphi} \right)+ \left((\partial_X r^2(t)) \dot{\vphi}+r(t)^2 \partial_X\dot{\vphi}\right) \wedge d\bar{t}.
\end{eqnarray*}
Because $\Omega$ is the pull back of positive $(1,1)$-form $\sddb r^2$ on $\bC^N$ under the holomorphic mapping $(p, t)\mapsto \sigma(t)\cdot p$, $\Omega$ itself is a smooth positive $(1,1)$-form. 
Using identities \eqref{eq-xi0r}-\eqref{eq-xi0u}, 
it's easy to verify that \eqref{eq-d2Evphi} can be expressed as:
\begin{equation}
-\left(\frac{d^2}{ds^2}E(\vphi(s))\right) \sqrt{-1} ds\wedge d\bar{s}=\frac{1}{C(n,\xi_0)(n+1)}\int_{X\times\bC/\bC} \Omega^{n+1} r(s)^{-2} e^{-r(s)^2}\ge 0.
\end{equation}
The inequality is in the sense of positivity of currents.

\end{proof}

Now set $\xi_{\epsilon}=\xi-\epsilon \eta=\sum_i (a_i-\epsilon b_i)z_i\partial_{z_i}$ for $0\le \epsilon\ll 1$ and consider a new radius function with Reeb vector field $\xi_\epsilon$. For example, corresponding to \eqref{eq-radius} we can choose
\[
r_{\epsilon}^2:=\left(\sum_i |z_i|^{2/(\delta (a_i-\epsilon b_i))}\right)^{\delta}.
\]
We use the following identity expressing the volume of $\wt_{\xi_\epsilon}$ (see \cite{MSY08, CS12, DS15, HS16}):
\begin{equation}
\vol_{X_0}(\xi_\epsilon):=\vol_{X_0}(\wt_{\xi_\epsilon})=\frac{1}{(2\pi)^n n!}\int_{X_0} e^{-r^2_{\epsilon}}(\sddb r_{\epsilon}^2)^n.
\end{equation}

We need the following important formula due to Martelli-Sparks-Yau (see also \cite{DS15})
\begin{lem}[{\cite[Appendix C]{MSY08}}]
The first order derivative of $\vol_{X_0}(\xi_\epsilon)$ is given by the formula:
\begin{eqnarray}\label{eq-dvolint}
\left.\frac{d}{d\epsilon}\vol(\xi_\epsilon)\right|_{\epsilon=0}&=&\frac{1}{(2\pi)^nn!}\int_{X_0} (r^2\theta) e^{-r^2}(\sddb r^2)^n\nonumber\\
&=&\frac{1}{(2\pi)^n (n-1)!}\int_{X_0}\theta e^{-r^2} (\sddb r^2)^n,
\end{eqnarray}
where $\theta=\theta_\eta=\eta(\log r^2)$ is a bounded function on $\bC^N\setminus\{0\}$.
\end{lem}
Since our notations may be different from that in the literature, for the reader's convenience we provide a brief calculation.
\begin{proof}
Let $\left.\frac{d}{d\epsilon}\right|_{\epsilon=0}r^2_\epsilon=r^2 u$. Then we have
\begin{eqnarray}\label{eq-depvol1}
\left.\frac{d}{d\epsilon}\right|_{\epsilon=0}\vol_{X_0}(\xi_\epsilon)&=&\frac{1}{(2\pi)^n n!}\left(\int_{X_0} (-r^2u)e^{-r^2}(\sddb r^2)^n\right.\nonumber\\
&&\hskip 1.5cm +\left.\int_{X_0} e^{-r^2} n\sddb (r^2u)\wedge (\sddb r^2)^{n-1}\right).
\end{eqnarray}
To simplify the expression, we do the integration by parts:
\begin{eqnarray}\label{eq-depvol2}
&&\int_{X_0}e^{-r^2} n \sddb (r^2u)\wedge (\sddb r^2)^{n-1}\nonumber\\
&=&\int_{X_0} e^{-r^2} n \sqrt{-1} \partial(r^2u)\wedge (\bar{\partial} r^2)\wedge (\sddb r^2)^{n-1}\nonumber\\
&=&\int_{X_0} e^{-r^2} n \sqrt{-1}(2r u\partial r+r^2 \partial u)\wedge 2r \bar{\partial}r\wedge (\sddb r^2)^{n-1}\nonumber\\
&=&\int_{X_0} e^{-r^2} (r^2u+r^2 \xi_0(u)) (\sddb r^2)^{n}.
\end{eqnarray}
In the last identity, we have used \eqref{eq-xi0u} for $v=r$ and $v=u$ respectively. Now the key is to take the variation of the following identity:
\[
r^2_{\epsilon}=\xi_{\epsilon}(r_{\epsilon}^2)
\]
to get:
\begin{equation}
r^2u=
-{\eta}(r^2)+\xi_0(r^2 u)= -\eta(r^2)+ r^2 u+ r^2\xi_0(u),
\end{equation}
which implies $r^2\xi_0(u)=\eta(r^2)$. Combining this with \eqref{eq-depvol1}-\eqref{eq-depvol2}, we get the first identity of \eqref{eq-dvolint}. 
The second identity follows from the first one by using polar coordinate and the fact that $\theta$ does not depend on $r$. 
\end{proof}

\begin{prop}[see \cite{CS15}]\label{prop-Eslope}
The limiting slope of $E(\vphi(s))$ is equal to the derivative of the volume up to a constant:
\begin{equation}\label{eq-limitslopeE}
\lim_{s\rightarrow+\infty} \frac{d}{ds}E(\vphi(s))=\frac{(D_{-\eta} \vol)(\xi)}{\vol(\xi_0)}.
\end{equation}
\end{prop}
\begin{proof}
Let $\sigma(t)$ be the $\bC^*$-action generated by $\eta=\sum_i b_i z_i \partial_{z_i}$. 
Recall that we have $r(t)^2=(\sigma^* r^2)|_{X_1}$ and $\dot{\vphi}=\frac{\partial}{\partial s}\vphi$ is equal to $-\sigma(t)^*(\theta)$ (recall that $s=-\log|t|^2$).
So we get the following identities:
\begin{eqnarray*}
\frac{d}{ds}E(\vphi(s))&=&-\frac{1}{(2\pi)^n(n-1)!\vol(\xi_0)}\int_{X}\dot{\vphi}(t) e^{-r(t)^2}(\sddb r(t)^2)^n\\
&=&\frac{1}{(2\pi)^n(n-1)!\vol(\xi_0)}\int_X\sigma(t)^*(\theta) e^{-\sigma^*(r^2)}(\sddb \sigma^*(r^2))^n\\
&=&\frac{1}{(2\pi)^n(n-1)!\vol(\xi_0)}\int_{X_t}\theta e^{-r^2}(\sddb r^2)^n.
\end{eqnarray*}
By using polar coordinate it's easy to see that:
\begin{equation}
\int_{X_t}\theta e^{-r^2}(\sddb r^2)^n=C_n\cdot \int_{\{r\le 1\}\cap X_t}\theta e^{-r^2} (\sddb r^2)^n,
\end{equation}
where
\[
C_n=\frac{\int_0^{\infty} e^{-r^2} r^{2n-1}dr}{\int_0^1 e^{-r^2}r^{2n-1}dr}=\frac{1}{1-e^{-1}\sum_{i=0}^{n-1}1/i!}.
\]
Then using the boundedness of $\theta$ and the argument in \cite[pp.67-68]{Li13} (see also \cite{Ber15}), we get that:
\begin{eqnarray}\label{eq-cutlim}
\lim_{s\rightarrow +\infty}\int_{\{r\le 1\}\cap X_t} \theta e^{-r^2} (\sddb r^2)^n &=& \int_{\{r\le 1\}\cap X_0}\theta e^{-r^2}(\sddb r^2)^n.
\end{eqnarray}
Combining the above identities and \eqref{eq-dvolint}, we get the identity \eqref{eq-limitslopeE}. 
\end{proof}

Next we need to deal with the part $G(\vphi(s))$:
\begin{eqnarray*}
G(\vphi(s))&=&  \log\left(\int_{X} e^{-r(t)^2} dV\right).
\end{eqnarray*}
The flat family $\cX\rightarrow \bC$ has a $\bC^*$-equivariant volume form $dV_{\cX/\bC}$ such that $\left.dV_{\cX/\bC}\right|_{X_t}$ is a volume form on $X_t$. In the case when $\cX$ is induced by a Koll\'{a}r component which is the main case that we used in the main text, we have given an explicit description in Proposition \ref{prop-Teqsec}. Moreover, by Remark \ref{rem-Ltpt}, we have
$
\cL_{\eta}dV_{\cX/\bC}=A(\eta) \cdot dV_{\cX/\bC},
$
which implies 
\begin{equation}
\sigma(t)^*dV_{\cX/\bC}=e^{A(\eta)\log |t|^2}\cdot dV_{\cX/\bC}.
\end{equation}
Then we have:
\begin{eqnarray}\label{eq-limitslopeG}
G(\vphi_t)&=& \log\left( \int_{X} e^{-\sigma(t)^* r^2} (\sigma^*dV_{X_t})e^{-A(\eta)\log |t|^2}\right)\nonumber\\
&=&A(\eta)(-\log|t|^2)+\log\left(\int_{X_t}e^{-r^2}dV_{X_t}\right)\nonumber\\
&=:& A(\eta)s+\tilde{G}(t),
\end{eqnarray}
where we have denoted 
\begin{equation}\label{eq-tildeG}
\tilde{G}(t)=\log\left(\int_{X_t} e^{-r^2} dV_{X_t}\right).
\end{equation}
We need the following variation of a result from \cite[Lemma 3.7]{Li13}:
\begin{lem}
The function $\tilde{G}(t)$ in \eqref{eq-tildeG} is a bounded continuous function with respect to $t\in \bC$.  
\end{lem}
\begin{proof}
As in the proof of Proposition \ref{prop-Eslope}, we first transform the integration domain to a compact set. Because $\cL_{\xi_0}dV_{X_t}=A(\xi_0)dV_{X_t}$ and $r\partial_r=2 {\rm Re}(\xi_0)$, we also have
$\cL_{r\partial_r}dV_{X_t}=2 A(\xi_0) dV_{X_t}$. So it's easy to verify that:
\begin{equation}
\int_{X_t} e^{-r^2} dV_{X_t} = C_n\cdot \int_{\{r\le 1\}\cap X_t}e^{-r^2}dV_{X_t}.
\end{equation}
Now on the part $\cX\cap (\{r\le 1\}\times\bC) \subset \bC^N\times\bC$, we can then use the same calculation as in \cite[section 4]{Li13} to get the conclusion.
\end{proof}

\begin{prop}[see also \cite{CS15}]
If $A_X(\xi_0)=n$, then the asymptotic slope of the Ding energy is equal to the Futaki invariant of the special test configuration:
\begin{equation}\label{eq-limitslopeD}
\lim_{s\rightarrow+\infty} \frac{D(\vphi(s))}{s}=\frac{\Fut(X_0, \xi_0; \eta)}{n^{n}\cdot \vol(\xi_0)}.
\end{equation}
\end{prop}
\begin{proof}
Combining \eqref{eq-limitslopeE} and \eqref{eq-limitslopeG}, we get:
\begin{eqnarray*}
\lim_{s\rightarrow+\infty}\frac{D(\vphi(s))}{s}&=&\lim_{s\rightarrow +\infty} \frac{E(\vphi(s))}{s}-\lim_{s\rightarrow+\infty}\frac{G(\vphi(s))}{s}\\
&=& \frac{(D_{-\eta}\vol)(\xi_0)}{\vol(\xi_0)}-A(\eta)\\
&=& \frac{(D_{-\eta}\vol)(\xi_0)+A(-\eta)  \vol(\xi_0)}{\vol(\xi_0)}.
\end{eqnarray*}
On the other hand, using the assumption that $A_X(\xi_0)=A_{X_0}(\xi_0)=n$, we have
\begin{eqnarray*}
\Fut(X_0, \xi_0; \eta)&=&(D_{-\eta}\hvol_{X_0})(\xi_0)\\
&=&
A^n_{X_0} (\xi_0) (D_{-\eta}\vol_{X_0}) (\xi_0)+n A_{X_0}(\xi_0)^{n-1} A_{X_0}(-\eta) \vol_{X_0}(\xi_0)\\
&=&n^n \left((D_{-\eta}\vol)(\xi_0)+ A(-\eta) \vol(\xi_0)\right)\\
&=&n^n \cdot \vol(\xi_0) \cdot \lim_{s\rightarrow+\infty} \frac{D(\vphi(s))}{s}.
\end{eqnarray*}
Now \eqref{eq-limitslopeD} follows from the above identity. 
\end{proof}

Finally we can complete the proof of Theorem \ref{thm-irSE}.
\begin{proof}[Completion of the proof of Theorem \ref{thm-irSE}]If there exists a Ricci-flat K\"{a}hler cone metric on $(X, \xi)$, then the Ding energy $D(\vphi)$ is bounded from below. As pointed out in \cite{DS15} this can be proved by following the same proof for the K\"{a}hler-Einstein case. Indeed, for any $\vphi\in PSH(X, \xi_0)$, we get transversal K\"{a}hler potential which is still denoted by $\vphi$. By using the same proof as in \cite[pp. 156-157]{Bern15}, there exists a bounded geodesic $\vphi_t$ connecting $0$ and $\vphi$. On the other hand, adapting Berndtsson's proof of subharmonicity to the Sasakian case, Donaldson-Sun showed that $D(\vphi_t)$ is convex with respect to $t$. Because $\vphi_0=0$ is a critical point of $D(\vphi_t)$, one knows that $D(\vphi)\ge D(0)$. 

 Since $D(\vphi(s))$ is uniformly bounded from below, by \eqref{eq-limitslopeD}, $\Fut(\cX, \xi_0; \eta) \ge 0$. As this holds for any special test configuration induced by any Koll\'{a}r component, we get the conclusion.
\end{proof}

\subsection{Finite degree formula}\label{s-finitedegree}

Now we can verify the degree multiplication formula in Theorem \ref{t-finitedegree}.

\begin{proof}[Proof of Theorem \ref{t-finitedegree}] We first assume $\pi$ is a Galois covering with the Galois group $G$.

Let $v_0$ be the valuation defined in Section \ref{ss-DS}, which induces the degeneration of $(o\in M_{\infty})$ to $W$.
We can fix a sequence of $v_i\to v_0$ such that $v_i$ is a rescaling of Koll\'ar component. Since the pull back of a Koll\'ar component is a $G$-invariant Koll\'ar component, we conclude that we can pull back $v_0$ to get a $G$-invariant valuation $v'\in \Val_{Y,y}$.  It suffices to prove $v'$ is a minimizer of $\hvol_{Y}$ (see e.g., \cite[Theorem 2.6]{LX17}).

\medskip

Since $v'$ is $G$-invariant,  $v'(f)=\frac{1}{G}v'({\rm Nm}(f))$, we know that a $G$-invariant element in $R'$ has its valuation under $v'$ is at least $k$ if and only if it is an element in $R$ whose valuation under $v$ is at least $k$, i.e. $(\fa_{v'})^G_k=(\fa_v)_k$. In particular, ${\rm gr}_{v'}R'$ is finitely generated as it is finite over  ${\rm gr}_{v}R$. We denote by $W_Y={\rm Spec}({\rm gr}_{v'}R')$. Then $\pi_W\colon W_Y\to W$ is quasi-\'etale with Galois group $G$. In fact this is clear in the quasi-regular case, and in the general case, we can use Lemma \ref{l-finiteass} to reduce to the quasi-regular case.  Furthermore, $G$ commutes with the $T$-action on $W_Y$ as it preserves the $v$-degree, and the vector associated to $v'$ on $W_Y$ is $\pi^*(\xi_0)$, which is denoted by $\xi_0'$.

\medskip

Consider the special test configuration $(\mathcal{W},\xi_0;\eta)$ which degenerates $(W,\xi_0)$ to $(C, \xi_0)$. Taking the finite normalization of $\mathcal{W}$ in the function field $K(W_Y\times \mathbb{A}^1)$, we obtain a normal variety $\Pi\colon \mathcal{W}_Y\to \mathcal{W}$ with the special fiber denoted by $C_Y$, which is reduced as $G$ acts on trivially on the base $\mathbb{C}^*$. Since $(\mathcal{W},C)$ is plt, it implies $(\mathcal{W}_Y, C_Y)$ is plt.  Therefore, $(\mathcal{W}_Y,\xi_0';\eta')$ is a special test configuration of $(W_Y,\xi'_0)$ to a Fano cone singularity  $(C_Y,\xi'_0)$ which is a $G$-covering of $(C,\xi_0)$. The latter has a Ricci-flat K\"ahler cone metric (see \cite{CCT02,DS15}) which can be pulled back to give such a metric on $(C_Y,\xi_0)$. More precisely, there exists a
radius function $r_\vphi=r e^{\vphi/2}$ that is a solution to the Monge-Amp\'{e}re equation (see \eqref{eq-RFKC})
\begin{equation}\label{eq-RFMA}
(\sddb r^2_\vphi)^n=C\cdot \left(\sqrt{-1}^{mn^2}s\wedge \bar{s}\right)^{1/m},
\end{equation}
where $s$ is a $T$-equivariant non-vanishing holomorphic section of $\left|mK_{C}\right|$. We can pull back both sides of \eqref{eq-RFMA} to $C_Y$ and get a solution to the corresponding Monge-Amp\'{e}re equation on $C_Y$:
\begin{equation}
(\sddb \pi^*(r_\vphi^2))=C\cdot \left(\sqrt{-1}^{mn^2}(\pi^*s')\wedge (\overline{\pi^*s'})\right)^{1/m}.
\end{equation}
The identity holds in the sense of pluripotential theory. Moreover, because $\pi$ is quasi-\'{e}tale, it is easy to see that $\pi^*(r_\vphi^2)$ is a Ricci-flat K\"{a}hler cone metric in the sense of Definition \ref{defn-RFKC} and the associated Reeb vector field on the regular part of $C_Y$ is nothing but $2{\rm Im}(\xi'_0)$. So by Theorem \ref{thm-irSE}, $\wt_{\xi'_0}$ is indeed a minimizer of $\hvol_{C_Y}$. Arguing in the proof of Theorem \ref{t-finiteunique}, we know that $v'$ is indeed a minimizer of $\hvol_{Y}$.

\bigskip

Now we treat the case that $\pi$ is a general quasi-\'etale morphism. Let $\pi'\colon (Y',y')\to (M_{\infty},o)$ be the Galois morphism generated by $\pi$ which factors through $\pi$ and  denote its Galois group by $G$. Then we know that the minimizer $v'$ of $(Y',y')$ is $G$-invariant by the above discussion on the Galois case. Therefore it is also ${\rm Galois}(Y'/Y)$-invariant, which implies $\hvol(y',Y')=\deg(Y'/Y)\cdot\hvol(y,Y)$. So we conclude 
$$\hvol(Y, y)=\deg(Y'/Y)\cdot\hvol(M_{\infty}, o),$$
as $\deg(\pi)=|G|/\deg(Y'/Y)$.
\end{proof}

\appendix
\section{Example: $D_{k+1}$-singularities}\label{sec-exmp}
In this section, we verify that the candidate minimizers computed in \cite{Li15a} for $D_{k+1}$ singularities induced by monomial valuations on the ambient spaces are indeed the unique quasi-monomial minimizers of $\hvol$, except possibly for the case of $4$-dimensional $D_4$ singularity for which we can not confirm yet.

\begin{exmp}
Consider the 3-dimensional $D_{k+1}$ singularity for $k\ge 4$:
\[
o:=\{0,0,0,0\}\in X=\left\{f(z_1,\dots, z_4):=z_1z_2+z_3^2z_4+z_4^k=0\right\}.
\]
$X={\rm Spec} R$ with $R=\bC[z_1, \dots, z_4]/(f(z))$. In \cite{Li15a}, we calculated the candidate minimizing valuation $v_0$ of $\hvol_{X,x}$. $v_0$ is induced by the weight 
\[
w_0=(1, 1, \sqrt{3}-1, 4-2\sqrt{3}).
\]
We verify here that this is indeed a global minimizer of $\hvol_{X,x}$. First notice that the weight $w_0$ degenerates $X$ to the following klt singularity:
\begin{equation}\label{eq-spinch}
X_0=\left\{z_1 z_2+z_3^2 z_4=0 \right\}.
\end{equation}
$X_0$ is called the suspended pinch point in \cite{MSY06}. It is a toric singularity. Indeed it admits an effective action by $T=(\bC^*)^3$ given by:
\[
(t_1, t_2, t_3)\circ (z_1, \dots, z_4)=(t_1 z_1, t_2 z_2, t_3 z_3, t_1 t_2 t_3^{-2} z_4).
\]
It is easy to see that the polyhedral cone $\sigma$ and its dual (moment cone) $\sigma^{\vee}$ are given by:
\[
\sigma={\rm Span}\left\{
\left(
\begin{array}{c}
1\\0\\0
\end{array}
\right), 
\left(
\begin{array}{c}
0\\1\\0
\end{array}
\right),
\left(
\begin{array}{c}
2\\0\\1
\end{array}
\right),
\left(
\begin{array}{c}
0\\2\\1
\end{array}
\right)
\right\};
\]
\[
\sigma^\vee={\rm Span}\left\{
\left(
\begin{array}{c}
1\\0\\0
\end{array}
\right), 
\left(
\begin{array}{c}
0\\1\\0
\end{array}
\right),
\left(
\begin{array}{c}
0\\0\\1
\end{array}
\right),
\left(
\begin{array}{c}
1\\1\\-2
\end{array}
\right)
\right\}.
\]
Moreover, it was known that there exists a Sasaki-Einstein metric on $X_0$ and its Reeb vector field can be calculated explicitly (see \cite{MSY06}). Here we can calculate  the Reeb vector field using the above combinatorial data. $J(r\partial_r)=2 {\rm Im}(\xi_0)$ where the holomorphic vector field $\xi_0$ corresponds to an element $\xi_0\in\ft^{+}_\bR$ which satisfies two conditions:
(i) $A_X(\xi_0)=3$, (ii) $\xi_0$ minimizes $\hvol(\xi)$ among all $\xi\in \ft^{+}_\bR$. Notice that $X_0$ is a Gorenstein singularity and $A(\xi)=\langle u_0, \xi\rangle$ with $u_0=(1,1,-1)$. By using the $\bZ_2$ symmetry of the cones, it's elementary to get the unique minimizer
\begin{equation}
\xi_0=\left(\frac{3+\sqrt{3}}{2}, \frac{3+\sqrt{3}}{2}, \sqrt{3}\right).
\end{equation}
Now the weight corresponding to $\xi_0$ on the $(z_1, \dots, z_4)$ is equal to:
\begin{equation}
\left(\frac{3+\sqrt{3}}{2}, \frac{3+\sqrt{3}}{2}, \sqrt{3}, 3-\sqrt{3} \right)=\frac{3+\sqrt{3}}{2}\left(1, 1, \sqrt{3}-1, 4-2\sqrt{3} \right)=\frac{3+\sqrt{3}}{2}w_0.
\end{equation}
So $w_0$ is indeed a global minimizer of $\hvol_{X,x}$. 
\end{exmp}

\begin{exmp}\label{exmp-Tvar}
Consider the $4$-dimensional $D_{k+1}$ singularity for $k\ge 4$:
\[
X=\left\{z_1z_2+z_3^2+z_4^2z_5+z_5^k=0\right\}.
\]
The candidate minimizing valuation calculated in \cite{Li15a} is induced by the following weight:
\[
w_0=\left(1,1, \frac{-3+\sqrt{33}}{4}, \frac{7-\sqrt{33}}{2}\right).
\]
The weight $w_0$ degenerates $X$ to the non-isolated singularity:
\[
X_0=\left\{z_1 z_2+z_3^2+z_4^2z_5=0\right\}.
\]
We observe that $X_0$ is a $T$-variety of complexity one. $T=(\bC^*)^3$ acts by:
\[
t\cdot z=(t_1 z_1,  t_1^{-1} t_2^2 z_2, t_2 z_3, t_3 z_4, t_2^2 t_3^{-2} z_5).
\] 
We want to show that $(X_0, \xi_0)$ is K-semistable by using the theory of $T$-varieties as has been used in \cite{CS15} which is based on the study of $T$-equivariant special test configurations in \cite{IS17}. Notice that because we have been studying the question purely algebraically, we can indeed deal with K-semistability of general (non-isolated) klt singularities like $X_0$. 

Using the process in \cite[Section 11]{AH06}, we can write down the polyhedral divisor determining $X_0$. First we write down the polyhedral divisor for $\bC^5$ as the $T$-variety. Following \cite{AH06}, for the above $T$-action, we have the exact sequence:
\begin{equation}
0\longrightarrow N_1:=\bZ^3\stackrel{F}{\longrightarrow} N_2:=\bZ^5 \stackrel{P}{\longrightarrow} N_3:=\bZ^2\longrightarrow 0,
\end{equation}
where $F$ and $P$ are given by the following matrices:
\begin{equation}
F=\left(
\begin{array}{ccc}
1&0&0\\
-1&2&0\\
0&1&0\\
0&0&1\\
0&2&-2
\end{array}
\right), 
\quad
P=
\left(
\begin{array}{ccccc}
-1&-1&0&2&1\\
-1&-1&2&0&0
\end{array}
\right)
\end{equation}
We then find $s: N_2\rightarrow N_1$ satisfying $s\circ F={\rm id}_{N_1}$. $s$ can be chosen simply to be:
\begin{equation}
s=\left(
\begin{array}{ccccc}
1&0&0&0&0\\
0&0&1&0&0\\
0&0&0&1&0
\end{array}
\right).
\end{equation}
The generic fiber of $\widetilde{\bC^5}\rightarrow Y_{\rm toric}$ is the toric variety associated to the following cone:
\begin{eqnarray*}
\sigma&=&s\left(\bQ^5_{\ge 0}\cap F(\bQ^3)\right)=\left\{x\ge 0, y\ge 0, z\ge 0, -x+2y\ge 0, y-z\ge 0 \right\}\\
&=&{\rm Span}_{\bR_{\ge 0}}\left\{
\left(
\begin{array}{c}
0\\1\\0
\end{array}
\right), 
\left(
\begin{array}{c}
2\\1\\0
\end{array}
\right),
\left(
\begin{array}{c}
2\\1\\1
\end{array}
\right),
\left(
\begin{array}{c}
0\\1\\1
\end{array}
\right)
\right\}.
\end{eqnarray*}
The dual cone $\sigma^{\vee}$ is given by:
\begin{eqnarray*}
\sigma^{\vee}&=&{\rm Span}_{\bR_{\ge 0}}\left\{
\left(
\begin{array}{c}
1\\0\\0
\end{array}
\right), 
\left(
\begin{array}{c}
0\\0\\1
\end{array}
\right),
\left(
\begin{array}{c}
-1\\2\\0
\end{array}
\right),
\left(
\begin{array}{c}
0\\1\\-1
\end{array}
\right)
\right\}\\
&=&
\left\{y\ge 0, 2x+y\ge 0, 2x+y+z\ge 0, y+z\ge 0\right\}.
\end{eqnarray*}

The base of $Y_{\rm toric}$ of $\bC^5$ as the $T$-variety is given by the toric variety associated to the fan cutted out by the column vectors of $P$. So it's clear that $Y_{\rm toric}=\bP^2$. The associated polyhedral divisor, denoted by
\begin{equation}
\fD=\Delta_{(1,0)}\otimes \{w_0=0\}+\Delta_{(0,1)}\otimes \{w_1=0\}+\Delta_{(-1,-1)}\otimes \{w_2=0\},
\end{equation}
can be calculated using the recipe from \cite{AH06}:
\begin{align*}
\Delta_{(1,0)}&=s\left(\bQ^5_{\ge 0} \cap P^{-1}(1,0)\right)=\{x\ge 0, y\ge 0, z\ge 0, -x+2y\ge 0, 2y-2z+1\ge 0\}\\
&=\{(0,0,t); 0\le t\le 1/2\}+\sigma=: \Delta_0\\
\Delta_{(0,1)}&=s\left(\bQ^5_{\ge 0} \cap P^{-1}(0,1)\right)=\{x\ge 0, y\ge 0, z\ge 0, -x+2y-1\ge 0, 2y-2z-1\ge 0\}\\
&=\{(0,t,0); 0\le t\le 1/2\}+\sigma=: \Delta_1\\
\Delta_{(-1,-1)}&=s\left(\bQ^5_{\ge 0}\cap P^{-1}(-1,-1)\right)=\{x\ge 0, y\ge 0, z\ge 0, -x+2y+1\ge 0, 2y-2z\ge 0\} \\
&=\{(t,0,0); 0\le t\le 1\}+\sigma=: \Delta_2.
\end{align*} 
Notice that $\Delta_0$ and $\Delta_1$ are non-integral while $\Delta_2$ is integral. Now the base $Y$ of $X$ is the normalization of the closure of image
of $X\cap (\bC^*)^5$ in $Y_{\rm toric}$. The map $(\bC^*)^5\rightarrow (\bC^*)^2$ is induced by the ring homomorphism $\bC[N_3^\vee]\rightarrow \bC[N_2^\vee]$ and hence is given under the coordinate by: 
\[
(\bC^*)^5\rightarrow (\bC^*)^2,\quad (z_1, z_2, z_3, z_4, z_5)=\left(\frac{z_4^2 z_5}{z_1z_2}, \frac{z_3^2}{z_1z_2}\right).
\]
So $Y$ is given by 
\[
Y=\{w_0+w_1+w_2=0\}\cong \bP^1.
\] 
We can restrict $\fD_{\rm toric}$ to $Y$ and thus obain a proper polyhedral divisor for the $T$-variety $X$:
\[
\fD=\Delta_0\otimes\{0\}+\Delta_1\otimes \{1\}+\Delta_\infty\otimes\{\infty\}.
\]
By the argument in \cite{IS17}, one knows that normal test configurations are determined by a triple $(q, v, m)$ where $q\in \bP^1$, $v$ is a vertex of $\sigma\cap (N_1)_\bQ$ and $m\in \bZ$ and they need to satisfy the following admissible condition (\cite[Definition 3.8]{IS17}): for all $u\in \sigma^\vee\cap N_1^\vee$, there is at most one $p\in \bP^1$ with $p\neq q$ such that the function
\[
\Delta_p(u)=\min_{v\in \Delta_p} \langle u, v\rangle
\] 
is non-integral. In the current example, if we choose $u=(0,1,-1)\in\sigma^\vee\cap N_1^\vee$ then
\[
\Delta_0(u)=-\frac{1}{2}, \quad \Delta_1(u)=\frac{1}{2}.
\]
So to get a normal test configuration, by the admissibility condition we are forced to choose either $q=0$ or $q=1$. On the other hand, the data $(v, m)$ only changes the action and does not change the total space of the test configuration. We can now easily guess the special test configurations whose special fibers are given by 
\begin{align*}
&X'_0=\{z_1z_2+z_3^2=0\}=\bC^2/\bZ_2\times \bC^2;\\
&X''_0=\{z_1z_2+z_4^2z_5=0\}=\hat{X}^3 \times \bC.
\end{align*}
Here $\hat{X}^3$ is the $3$-dimensional suspended pinch point that appeared in \eqref{eq-spinch}. One can verify by the same calculation in \cite{Li15a} or \cite{CS15} that these two special test configurations have positive Futaki invariants. So we conclude that $(X_0, \xi_0)$ is K-semistable and hence $v_0$ induced by $w_0$ is indeed a global minimizer of $\hvol$.

\end{exmp}

\begin{thm}For any $(n+1)$-dimensional $D_{k+1}$ singularity, except for 4-dimensional $D_4$ singularity, we know its unique quasi-monomial minimizer. \end{thm}

In fact, combining the above examples with the calculations in \cite{Li15a} and the arguments in \cite{LX16}, we have the following almost complete picture:
\begin{enumerate}
\item $n+1=2$, then $X\cong \{z_1^2+z_2^2z_3+z_3^k=0\}=\bC^2/D_{k+1}$ where $D_{k+1}$ is the $(k+1)$-th binary dihedral group. By \cite{LL16}, the valuation $v_0$ induced by the weight $\left(1,1-\frac{1}{k} , \frac{2}{k}\right)$ is a global minimizer of $\hvol$.
\item $n+1=3$, $k=3$. $X=\{z_1z_2+z_3^2z_4+z_4^3=0\}$ is a $T$-variety of complexity one with an isolated singularity. By \cite{CS15}, $X$ admits a quasi-regular Ricci flat K\"{a}hler cone metric whose Reeb vector field up to rescaling is associated to the natural weight $(1, 1, 2/3, 2/3)$. 
\item $n+1=4$, $k=3$. In this case, we expect that $X$ admits a quasi-regular Ricci-flat K\"{a}hler cone metric whose Reeb vector field is associated to the natural weight $(1, 1,1,2/3,2/3)$.
\item $n+1=5$, $k=3$. $X=\{z_1z_2+z_3^2+z_4^2+z_5^2z_6+z_6^3=0\}$. The minimizer $v_0$ is induced by the weight $w_0=(1,1,1,1,2/3,2/3)$ which preserves $X$. $X$ is strictly semistable because it specially degenerates to $X'=\{z_1z_2+z_3^2+z_4^2=0\}\cong A^3_1\times \bC^2$ with zero Futaki invariant. 
\item $n+1=3$ or $4$, and $k\ge 4$. These are the examples considered above. The minimizers found are quasi-monomial valuations of rational rank 3.
\item $n+1=5$ and $k\ge 4$. The minimizer $v_0$ is induced by the weight $w_0=(1,1,1,2/3,2/3)$. $w_0$ specially degenerates $X$ to 
$X_0=\{z_1z_2+z_3^2+z_4^2+z_5^2z_6=0\}$ which is strictly semistable since $X_0$ further degenerates to $X'=\{z_1 z_2+z_3^2+z_4^2=0\}\cong A^3_1\times \bC^2$ with zero Futaki invariant. 

\item $n+1\ge 6$ and $k\ge 3$. The minimizer $v_0$ is induced by $w_0=\left(1, \dots, 1, \frac{n-2}{n-1}, \frac{n-2}{n-3}\right)$. $w_0$ degenerates $X$ to 
$\{z_1^2+\dots+z_n^2=0\}=A^{n-1}_1\times \bC^2$. 
\end{enumerate}


\vspace{9mm}

\noindent
Chi Li, Purdue University.   \\
li2285@purdue.edu

\medskip
\noindent
Chenyang Xu, Beijing International Center for Mathematical Research. \\
cyxu@math.pku.edu.cn.


\begin{thebibliography}{99}


\bibitem[AH06]{AH06}
K. Altmann and J. Hausen, Polyhedral divisors and algebraic torus actions, {\it Math. Ann}. {\bf 224}, Issue 3 (2006) 557-607.

\bibitem[AIPSV12]{AIPSV11}
K. Altmann, N.O. Ilten, L. Petersen, H. S\"{u}ss and R. Vollmert, The geometry of $T$-varieties. {\it Contributions to algebraic geometry}, 17--69, EMS Ser. Congr. Rep., Eur. Math. Soc., Z\"urich, 2012.

\bibitem[Amb06]{Amb06}
F. Ambro, The set of toric minimal log discrepancies, {\it Cent. Eur. J. Math.} {\bf 4} (2006), no. 3, 358-370.


\bibitem[Ber15]{Ber15}
R. Berman, K-polystability of $\mathbb{Q}$-Fano varieties admitting K\"ahler-Einstein metrics, {\it Invent. Math.}, {\bf 203} (2015), no. 3, 973-1025.


\bibitem[Bern15]{Bern15}
B. Berndtsson, A Brunn-Minkowski type inequality for Fano manifolds and the Bando-Mabuchi uniqueness theorem, {\it Invent. Math.} {\bf 205} (2015), no. 1, 149-200.


\bibitem[BBEGZ11]{BBEGZ}
R. Berman, S. Boucksom, P. Eyssidieux, V. Guedj, A. Zeriahi, K\"{a}hler-Einstein metrics and the K\"{a}hler-Ricci flow on log Fano varieties, to appear in {\it J. Reine Angew. Math.} , arXiv:1111.7158.


\bibitem[BCHM10]{BCHM10}
C. Birkar, P. Cascini, C. D. Hacon, and J. M$^{\rm c}$Kernan, Existence of minimal models for varieties of log general type, {\it J. Amer. Math. Soc}. {\bf 23} (2010), 405-468.


\bibitem[Blu18]{Blu16}
H. Blum, Existence of Valuations with Smallest Normalized Volume, {\it Compos. Math.} {\bf 154} (2018), no. 4, 820--849.



\bibitem[BJ17]{BJ17}  
H. Blum and M. Jonsson, Thresholds, valuations, and K-stability, arXiv:1706.04548.


\bibitem[BFJ14]{BFJ14}
 S. Boucksom, C. Favre and M. Jonsson,  A refinement of Izumi's theorem. {\it Valuation theory in interaction}, 55-81, EMS Ser. Congr. Rep., Eur. Math. Soc., Z\"urich, 2014.





\bibitem[BHJ17]{BHJ17}
S. Boucksom, T. Hisamoto and M. Jonsson, Uniform K-stability, Duistermaat-Heckman measures and singularities of pairs,  {\it Ann. Inst. Fourier (Grenoble)} {\bf 67} (2017), no. 2, 743--841.


\bibitem[CC97]{CC97}
J. Cheeger, T.H. Colding, On the structure of spaces with Ricci curvature bounded below. I. {\it J. Differential Geom.} {\bf 46} (1997),  no. 3, 406-480.

\bibitem[CCT02]{CCT02}J. Cheeger, T.H. Colding and G. Tian, On the singularities of spaces with bounded Ricci curvature. {\it Geom. Funct. Anal.} {\bf 12}(2002),  no.5, 873-914.

\bibitem[CDS15]{CDS15}
X.X. Chen, S. K. Donaldson and S. Sun, K\"ahler-Einstein metrics on Fano manifols, I-III, {\it J. Amer. Math. Soc.} {\bf 28} (2015), 183-197, 
199-234, 235-278.


\bibitem[CS12]{CS12}
T. Collins and G. Sz\'{e}kelyhidi, K-semistability for irregular Sasakian manifolds, to appear in {\it J. Differential Geom.}, arXiv:1204.2230.

\bibitem[CS15]{CS15}
T. Collins and  G. Sz\'ekelyhidi, Sasaki-Einstein metrics and K-stability. arXiv:1512.07213.



\bibitem[Cut12]{Cut12}
S.D. Cutkosky, Multiplicities associated to graded families of ideals. {\it Algebra Number Theory} 7 (2013), no. 9, 2059-2083. 




\bibitem[dFKX17]{dFKX17}
T. de Fernex, J. Koll\'ar and C. Xu,  The dual complex of singularities.  Higher dimensional algebraic geometry, 103-130,  Adv. Stud. in Pure Math. {\bf 74}, 2017.


\bibitem[Don01]{Don01}
S. Donaldson, 
Scalar curvature and projective embeddings. I. 
{\it J. Differential Geom}. {\bf 59} (2001), no. 3, 479--522. 

\bibitem[Don16]{Don16}
S. Donaldson, K\"{a}hler-Einstein metrics and algebraic structures on limit spaces, {\it Surveys in differential geometry 2016. Advances in geometry and mathematical physics,} 85--94, Surv. Differ. Geom., 21, Int. Press, Somerville, MA, 2016.

\bibitem[DS14]{DS14}
S. Donaldson and S. Sun, Gromov-Hausdorff limits of K\"ahler manifolds and algebraic geometry. {\it Acta Math}. {\bf 213} (2014), no. 1, 63--106.

\bibitem[DS17]{DS15}
S. Donaldson and S. Sun, Gromov-Hausdorff limits of K\"ahler manifolds and algebraic geometry, II, {\it J. Differential Geom.} {\bf 107} (2017), no. 2, 327--371


\bibitem[ELS03]{ELS03}
L. Ein, R. Lazarsfeld and K. Smith, Uniform approximation of Abhyankar valuation ideals in smooth function fields,
{\it Amer. J. of Math.} {\bf 125} (2003), no. 2, 409-440.

\bibitem[Fuj18]{Fuj15}
 K. Fujita, Optimal bounds for the volumes of K\"ahler-Einstein Fano manifolds, {\it Amer. J. Math.} {\bf 140} (2018), no.2, 391-414.

\bibitem[Ful98]{Ful98}
W. Fulton, Intersection theory. Second edition. Ergebnisse der Mathematik und ihrer Grenzgebiete. 3. Folge. A Series of Modern Surveys in Mathematics, 2. Springer-Verlag, Berlin, 1998. 


\bibitem[Fut83]{Fut}A. Futaki, An obstruction to the existence of Einstein
K\"{a}hler metrics, {\it Invent. Math.} {\bf 43}, 437-443 (1983)



\bibitem[Gig78]{Gig78}
S. Gigena, Integral invariants of convex cones, {\it J. Differential Geom.} {\bf 13} (1978), 191-222.



\bibitem[HMX14]{HMX14}
C. Hacon, J. M$^{\rm c}$Kernan and C, Xu, ACC for log canonical thresholds, {\it Annals of Math.} {\bf 180} (2014), no. 2, 523--571.

\bibitem[HS17]{HS16}
H.-J. Hein and S. Sun, Calabi-Yau manifolds with isolated conical singularities, {\it Publ. Math. IHES} {\bf 126} (2017), 73--130.

\bibitem[IS17]{IS17}
N. Ilten and H. S\"{u}ss, K-stability for Fano manifolds with torus action of complexity 1, {\it Duke Math. J.}  {\bf 166} (2017), no. 1, 2017, 177--204.

\bibitem[Ish04]{Ish04} 
S. Ishii,
Extremal functions and prime blow-ups. Comm. Algebra {\bf 32} (2004), no. 3, 819-827.

\bibitem[JM12]{JM12}
M. Jonsson and M. Musta\c{t}\u{a}, Valuations and asymptotic invariants for sequences of ideals, {\it Ann. Inst. Fourier (Grenoble)} {\bf 62} (2012), no. 6, 2145-2209. 

\bibitem[KK12]{KK12}
K. Kaveh and A. Khovanskii, Newton-Okounkov bodies, semigroups of integral points, graded algebras and intersection theory, {\it Ann. of Math.} {\bf 176} (2012), no. 2,  925-978.
\bibitem[KK14]{KK14}
K. Kaveh and A. Khovanskii, Convex bodies and multiplicities of ideals. {\it Proc. of the Steklov Inst. Math.} {\bf 286} (2014), no. 1,  268-284.



\bibitem[KM98]{KM98} 
J. Koll\'ar and S. Mori,  
Birational geometry of algebraic varieties, {\it Cambridge Tracts in Math.} {\bf 134}, 
Cambridge Univ. Press, Cambridge, 1998. 

\bibitem[Kol13]{Kol13}
J. Koll\'{a}r, Singularities of the Minimal Model Program, {\it Cambridge Tracts in Math.} {\bf 200}, Cambridge Univ. Press, Cambridge, 2013.


\bibitem[LM09]{LM09}
R. Lazarsfeld and M. Musta\c{t}\u{a}, Convex bodies associated to linear series,  {\it Ann. Sci. \'{E}c. Norm. Sup\'{e}r.} {\bf 42} (2009), no. 5, 783-835.


\bibitem[Li15]{Li15a}
C. Li, Minimizing normalized volumes of valuations, to appear in {\it Math. Z.}, arXiv:1511.08164.


\bibitem[Li17a]{Li13}
C. Li, Yau-Tian-Donaldson correspondence for K-semistable Fano manifolds, {\it J. Reine Angew. Math.}  {\bf 733} (2017), 55-85.


\bibitem[Li17b]{Li15b}
C. Li, K-semistability is equivariant volume minimization, {\it Duke Math. J.}  {\bf 166} (2017), no. 16, 3147-3218.



\bibitem[Liu18]{Liu16}
Y. Liu; The volume of singular K\"{a}hler-Einstein Fano varieties, {\it Compos. Math.} {\bf 154} (2018), no. 6,  1131-1158.

\bibitem[LL16]{LL16}
C. Li and Y. Liu, K\"{a}hler-Einstein metrics and volume minimization,  to appear in {\it Adv. Math.}, arXiv:1602.05094.

\bibitem[LS13]{LS13}
A. Liendo and H. S\"{u}ss, Normal singularities with torus actions, {\it Tohoku Mathematical Journal}, Second Series {\bf 65}
(2013), no. 1, 105-130.

\bibitem[LWX14]{LWX15} 
C. Li, X. Wang and C. Xu,
On proper moduli space of smoothable K\"ahler-Einstein Fano varieties,
 arXiv:1411.0761v3.

\bibitem[LWX18]{LWX17}
C. Li, X. Wang and C. Xu, Algebraicity of the Metric Tangent Cones and Equivariant K-stability, in preparation. 


\bibitem[LX14]{LX14}
C. Li and C. Xu, Special test configurations and K-stability of Fano varieties, {\it Ann. of Math.} (2) {\bf 180} (2014), no.1, 197-232.

\bibitem[LX16]{LX16}
C. Li and C. Xu, Stability of valuations and Koll\'{a}r components,  arXiv:1604.05398.

\bibitem[LiuX17]{LX17}
Y. Liu and C. Xu, K-stability of cubic threefolds,   arXiv:1706.01933.

\bibitem[MN15]{MN15}
J. Nicaise and M. Musta\c{t}\u{a}, Weight functions on non-Archimedean analytic spaces and the Kontsevich-Soibelman skeleton, {\it Algebraic Geometry} {\bf 2} (3) (2015) 365--404.



\bibitem[MSY06]{MSY06}
D. Martelli, J. Sparks and S.-T. Yau, The geometric dual of $a$-maximisation for toric Sasaki-Einstein manifolds, {\it Commun. Math. Phy.} {\bf 268} (2006), 39-65.

\bibitem[MSY08]{MSY08}
D. Martelli, J. Sparks and S.-T. Yau, Sasaki-Einstein manifolds and volume minimisation, {\it Commu. Math. Phys}. {\bf 280} (2008), 611-673.

\bibitem[Mus02]{Mus02}
M. Musta\c{t}\u{a}, On multiplicities of graded sequences of ideals, {\it Journal of Algebra} {\bf 256} (2002), 229-249.

\bibitem[NX16]{NX16}
J. Nicaise and C. Xu, The essential skeleton of a degeneration of algebraic varieties, {\it Amer. J. Math.}  {\bf 138} (2016), no. 6, 1645-1667.

\bibitem[Oko96]{Oko96}
A. Okounkov, Brunn-Minkowski inequality for multiplicities, {\it  Invent. Math.} {\bf 125} (1996),  no. 3, 405-411.

\bibitem[OX12]{OX12}
 Y. Odaka and C. Xu,  Log-canonical models of singular pairs and its applications, {\it Math. Res. Lett.} {\bf 19} (2012), no. 2, 325--334


\bibitem[PS08]{PS08}
L. Petersen and H. S\"{u}ss, Torus invariant divisors, {\it Israel J. Math.} {\bf 182} (2011), 481-504.

\bibitem[Pil94]{Pil94}
O. Piltant, Graded algebras associated with a valuation, preprint Ecole Polytechnique.



\bibitem[RT11]{RT11}
J. Ross and R. Thomas, Weighted projective embeddings, stability of orbifolds, and constant scalar curvature K\"ahler metrics. {\it J. Differential Geom.} {\bf 88} (2011), no. 1, 109--159.



\bibitem[SS17]{SS17}
C. Spotti and S. Sun, Explicit Gromov-Hausdorff compactifications of moduli spaces of K\"ahler-Einstein Fano manifolds,
 arXiv:1705.00377.

\bibitem[SSY16]{SSY16} 
C. Spotti, S. Sun and C. Yao, 
Existence and deformations of K\"ahler-Einstein metrics on smoothable $\mathbb{Q}$-Fano varieties,
{\it Duke Math. J.} {\bf 165} (2016), no. 16, 3043--3083. 

\bibitem[Tei03]{Tei03}
B. Teissier,  Valuations, deformations, and toric geometry. {\it Valuation theory and its applications, Vol. II} (Saskatoon, SK, 1999), 361--459, Fields Inst. Commun., {\bf} 33, Amer. Math. Soc., Providence, RI, 2003.

\bibitem[Tei14]{Tei14}
B. Teissier, Overweight deformations of affine toric varieties and local uniformization,  {\it Valuation theory in interaction}, 474--565, EMS Ser. Congr. Rep.,  Eur. Math. Soc., Z\"urich, 2014.

\bibitem[Tev14]{Tev14}
J. Tevelev, On a Question of Teissier, {\it Collect. Math.}, {\bf 65} (2014), no. 1, 61-66.

\bibitem[Tia90]{Tia90}
G. Tian, On Calabi's conjecture for complex surfaces with positive first Chern class. {\it Invent. Math.} {\bf 101} (1990), 101-172.

\bibitem[Tia97]{Tia97}
G. Tian, K\"{a}hler-Einstein metrics with positive
scalar curvature. {\it Invent. math.} {\bf 137} (1997), 1-37.

\bibitem[Tia12]{Tia12}
G. Tian, 
Existence of Einstein metrics on Fano manifolds, in {\it Metric and Differential Geometry}, Progr. Math.,
297, pp. 119-159. Birkh\"{a}user/Springer, Basel, 2012.

\bibitem[Tia13]{Tia13}
G. Tian, 
Partial $C^0$-estimate for K\"ahler-Einstein metrics,
{\it Commun. Math. Stat.} {\bf 1} (2013), no. 2, 105--113. 
\bibitem[Tia15]{Tia15}G. Tian, K-stability and K\"{a}hler-Einstein metrics, {\it Comm. Pure App. Math.}, {\bf 68} (2015), Issue 7, 1085-1156.  


\bibitem[Xu14]{Xu14}
C. Xu, Finiteness of algebraic fundamental groups, {\it Compos. Math.} {\bf  150} (2014), no. 3, 409-414. 

\bibitem[ZS60]{ZS60}
O. Zariski and P. Samuel, \textit{Commutative Algebra II, Grad. Texts in Math.}, Vol. 29, Springer-Verlag, New York, 1960.

\end{thebibliography}
\end{document}